\documentclass[11pt, reqno]{amsart}
\usepackage[T1]{fontenc}
\usepackage[margin=1.5in]{geometry}
\usepackage{amsmath, amssymb, amsthm}
\usepackage{txfonts}
\usepackage{hyperref}
\usepackage{verbatim}
\usepackage{mathtools}
\usepackage{enumitem}
\usepackage{mathrsfs}
\usepackage{dsfont}
\usepackage{enumitem}
\usepackage{upgreek}
\usepackage{tikz-cd}
\mathtoolsset{showonlyrefs}

\let\oldtocsection=\tocsection
\let\oldtocsubsection=\tocsubsection
\let\oldtocsubsubsection=\tocsubsubsection
\renewcommand{\tocsection}[2]{\hspace{0em}\oldtocsection{#1}{#2}}
\renewcommand{\tocsubsection}[2]{\hspace{1em}\oldtocsubsection{#1}{#2}}
\renewcommand{\tocsubsubsection}[2]{\hspace{2em}\oldtocsubsubsection{#1}{#2}}

\DeclareFontEncoding{FMS}{}{}
\DeclareFontSubstitution{FMS}{futm}{m}{n}
\DeclareFontEncoding{FMX}{}{}
\DeclareFontSubstitution{FMX}{futm}{m}{n}
\DeclareSymbolFont{fouriersymbols}{FMS}{futm}{m}{n}
\DeclareSymbolFont{fourierlargesymbols}{FMX}{futm}{m}{n}
\DeclareMathDelimiter{\VERT}{\mathord}{fouriersymbols}{152}{fourierlargesymbols}{147}
\DeclareMathOperator{\supp}{supp}

\usepackage[textsize=tiny]{todonotes}

\newcommand{\R}{\mathbb{R}}
\newcommand{\N}{\mathbb{N}}

\newcommand{\pr}{\mathbb{P}}
\newcommand{\ex}{\mathbb{E}}

\newcommand{\Cc}{\mathcal{C}}
\newcommand{\F}{\mathcal{F}}
\newcommand{\h}{\mathcal{H}}
\newcommand{\Tt}{\mathcal{T}}

\newcommand{\half}{\frac{1}{2}}
\newcommand{\Mf}{\mathfrak{M}}
\newcommand{\LSI}{\mathrm{LSI}}
\newcommand{\WLSI}{\mathrm{WLSI}}
\newcommand{\WLSIs}{\mathrm{WLSIs}}

\newtheorem{thm}{Theorem}[section]

\newtheorem{lem}[thm]{Lemma}
\newtheorem{cor}[thm]{Corollary}
\newtheorem{prop}[thm]{Proposition}

\newtheorem{dfn}[thm]{Definition}
\newtheorem{assu}[thm]{Assumption}

\numberwithin{equation}{section}

\theoremstyle{definition}
\newtheorem{rem}[thm]{Remark}
\newtheorem{Ex}[thm]{Example}

\title{Transportation-Cost inequalities for non-linear Gaussian functionals}
\author{Ioannis Gasteratos}
\address{Department of Mathematics, Imperial College London}
\email{i.gasteratos@imperial.ac.uk}

\author{Antoine Jacquier}
\address{Department of Mathematics, Imperial College London, and the Alan Turing Institute}
\email{a.jacquier@imperial.ac.uk}

\thanks{The authors acknowledge financial support from the EPSRC grant EP/T032146/1.}

\begin{document}
\begin{abstract} We study concentration properties for laws of non-linear Gaussian functionals on metric spaces. Our focus lies on measures with non-Gaussian tail behaviour which are beyond the reach of Talagrand's classical Transportation-Cost Inequalities (TCIs). 
Motivated by solutions of Rough Differential Equations and relying on a suitable contraction principle, 
we prove generalised TCIs for functionals that arise in the theory of regularity structures and, in particular, in the cases of rough volatility and the two-dimensional Parabolic Anderson Model. 
In doing so, we also extend existing results on TCIs for diffusions driven by Gaussian processes.
\end{abstract}

\maketitle

\tableofcontents

\section{Introduction}

 Talagrand's Transportation-Cost Inequalities (TCIs) have been well studied in the literature, especially due to their connections with concentration of measure, exponential integrability and deviation estimates. 
 Typically, they are of the following form: a probability measure~$\mu$, defined on a metric space 
 $(\mathcal{X}, g)$, satisfies the $p$-TCI for  $p\geq 1$ if there exists $C>0$ such that for any other probability measure~$\nu$ on~$\mathcal{X}$,
 \begin{equation}
     \label{eq: pTCI}
        W_p(\mu, \nu)\leq \sqrt{C H(\nu\;|\;\mu)},
 \end{equation}
 where $W_p$ is the $p$-Wasserstein distance with respect to~$g$ and~$H$ the relative entropy or Kullback-Leibler divergence (see Section~\ref{Section: TCIs} for the precise definitions). 
 Such inequalities were first considered by Marton~\cite{marton1996bounding} and Talagrand~\cite{talagrand1996transportation} and further investigated by Bobkov-G\"otze~\cite{bobkov1999exponential}, Otto-Villani~\cite{otto2000generalization} to name only a few.

 A crucial feature of $p$-TCIs is their close relation to the Gaussian distribution. In particular, Talagrand~\cite{talagrand1996transportation} first showed that Gaussian measures satisfy a $2$-TCI with a dimension-free constant~$C$, while Feyel-\"Ust\"unel~\cite{feyel2002measure} proved a similar inequality on abstract Wiener spaces and with respect to the Cameron-Martin distance. 
It was moreover established~\cite{bobkov1999exponential, djellout2004transportation} 
that the 1-TCI, the weakest among $p$-TCIs, is equivalent to Gaussian concentration.

 Apart from Gaussian measures themselves, the laws of many Gaussian functionals of interest, such as solutions to stochastic differential equations, satisfy $p$-TCI inequalities. In particular, the $1$-TCI, on pathspace for the law of a multidimensional diffusion 
 process
 \begin{equation}\label{eq: SDE}
      dY_t= b(Y_t) dt+\sigma(Y_t)dX_t,
 \end{equation}
 with~$X$ a standard Brownian motion and $b, \sigma$ bounded and Lipschitz continuous, was proved in~\cite{djellout2004transportation}. The case where~$X$ is a fractional Brownian motion (fBm) with Hurst parameter $H>\half$ was studied by Saussereau~\cite{saussereau2012transportation}. There, it was shown that the law of~$Y$ satisfies a $1$-TCI on pathspace if either~$X$ is one-dimensional or~$X$ is multi-dimensional and $\sigma$ is not state-dependent. 
 Subsequently, Riedel~\cite{riedel2017transportation} extended the results of~\cite{saussereau2012transportation} by showing that, for a wide class of Gaussian drivers~$X$ with paths of finite $r$-variation for some $r\leq 2$, the law of~$Y$ satisfies a $(2-\epsilon)$-TCI for all $\epsilon \in (0,2)$ and with respect to finite $r$-variation metrics.
 
   The work of~\cite{riedel2017transportation} establishes TCIs under the assumption that the driver~$X$ has equal or higher path regularity than standard Brownian motion. In this regime, pathwise solution theories are available via Young integration. For the case of rougher signals, namely fBm with Hurst parameter $H\in(\frac{1}{4}, \half]$,~\eqref{eq: SDE} can be treated in the framework of Lyons' theory of rough paths~\cite{lyons1998differential}. In a nutshell, a pathwise solution theory is available upon considering an enhanced driver~$\mathbf{X}$ consisting of~$X$ along with its iterated integrals. The integral in~\eqref{eq: SDE} is then considered in the sense of rough integration against the rough path~$\mathbf{X}$ (the interested reader is referred to~\cite{friz2020course, friz2010multidimensional} for a thorough exposition of rough paths theory). 

   In contrast to the aforementioned examples, solutions of Gaussian Rough Differential Equations (RDEs) provide a class of Gaussian functionals that fall beyond the reach of Talagrand's $p$-TCIs for any $p\geq 1$. Indeed, consider a Gaussian process~$X$ with paths of finite $r$-variation and Cameron-Martin space~$\h$.
   Cass, Litterer and Lyons~\cite{cass2013integrability} 
   (see also~\cite{friz2013integrability}) establish non-Gaussian upper bounds for tail probabilities of $Y$.
   In particular, if~$X$ admits a rough path lift~$\mathbf{X}$ and there exists $q\in[1, 2)$ with $1/r+1/q>1$ such that $\h\hookrightarrow\Cc^{q-var}$ then the upper bound is that of a Weibull distribution with shape parameter $2/q$ (note that for Brownian or smoother paths one can take $q=1$ and hence one obtains a Gaussian tail estimate). 
   Moreover, this non-Gaussian tail behaviour was shown to be sharp in the recent work~\cite{boedihardjo2022lack}, where a non-Gaussian tail-lower bound was provided for an elementary RDE. 
   In light of these facts, one deduces that the law of~$Y$ does not enjoy Gaussian concentration and hence cannot satisfy the $p$-TCI for any $p\geq 1$.

   Another important class of Gaussian functionals with non-Gaussian tail behaviour is provided, in mathematical finance, by rough (stochastic) volatility models~\cite{bayer2016pricing}. 
   These describe the dynamics of an asset price~$S$ via SDEs of the form
   \begin{equation}\label{eq:RVintro}
       dS_t=S_tf(\hat{W}^H_t, t)dB_t,
   \end{equation}
   where $\hat{W}^H$ is an fBm (of Riemann-Liouville type) with $H< \half$ and~$B$ is a standard Brownian motion that is typically correlated with $\hat{W}^H$. 
   Such models have been proposed due to their remarkable consistency with financial time series data (see~\cite{bayer2016pricing, gatheral2018volatility} and references therein) and calibrated volatility models suggest a Hurst parameter~$H$ of order~$0.1$.

   Besides the fact that~$S$ solves an SDE with unbounded (linear) diffusion coefficients, typical choices for the volatility function~$f$ are also unbounded (e.g. exponential as in the rough Bergomi model~\cite{bayer2016pricing}, 
   or polynomial~\cite{jaber2022joint}). 
   It is thus clear that neither the (driftless) log-price $fdB$ nor~$S$ itself fit into the framework of Talagrand's TCIs. 
   Moreover, it is well known that, for a broad class of volatility functions, 
   $p$-th moments of~$S$ for $p>1$ are infinite for $t>0$; 
   see for example~\cite{lions2007correlations} and~\cite{gassiat2019martingale} for the cases $H=\half$ and $H<\half$ respectively.
   
    Motivated by RDEs and rough volatility, our primary goal here is to identify a class of TCIs that is both general enough to include Talagrand's $p$-TCIs and also sufficient to capture non-Gaussian concentration and tail behaviour. In particular, we establish TCIs of the form  
      \begin{equation}
     \label{eq: acTCI}
        \alpha\big(W_c(\mu, \nu)\big)\leq H(\nu\;|\;\mu),
 \end{equation}
 where $\alpha:\R^+\rightarrow\R^+$ is a non-decreasing deviation function vanishing at the origin, 
 $W_c$ the transportation-cost with respect to a measurable cost $c:\mathcal{X}\times\mathcal{X}\rightarrow[0,\infty]$ (typically a concave function of a metric), namely
 $$
 W_{c}(\mu, \nu)=\inf_{\pi\in\Pi(\mu, \nu)}\iint_{\mathcal{X}\times\mathcal{X}}c(x,y)d\pi(x,y),
 $$
 and $\Pi(\mu, \nu)$ the family of all couplings between $\mu, \nu$ (see Section~\ref{Section: TCIs} for precise definitions). 
 A measure~$\mu$ for which~\eqref{eq: acTCI} holds for all measures~$\nu$ on~$\mathcal{X}$ is said to satisfy an $(\alpha, c)$-TCI. Similar types of TCIs have been considered by Gozlan and L\'eonard~\cite{10.1214/ECP.v11-1198,gozlan2007large} 
 (see also the survey paper~\cite{gozlan2010transport}) under slightly different assumptions on $\alpha$ and~$c$ 
 (in particular, $\alpha$ convex and cost~$c$ convex function of a metric, neither satisfied in our examples of interest).

 A natural question upon inspection of the previous examples is whether~\eqref{eq:RVintro} can be treated in a pathwise sense, similar to~\eqref{eq: SDE}. 
 Bayer, Friz, Gassiat, Martin and Stemper~\cite{bayer2020regularity} argued that, while~\eqref{eq:RVintro} falls outside the scope of classical geometric rough paths, a pathwise treatment is possible via Hairer's theory of regularity structures~\cite{hairer2014theory}. 
 In brief, after constructing an appropriate lift of the noise to a random Gaussian model~$\Pi$ (akin to the lift $X\mapsto\mathbf{X}$) and "expanding"~$f$ with respect to the higher-order functionals in~$\Pi$ (in the sense of Hairer's modelled distributions), 
 they obtain a pathwise formulation of~\eqref{eq:RVintro}. 
 Moreover the solution is continuous with respect to an appropriate topology in the space of models.

 This leads to the second objective of the present work, which is to obtain TCIs for Gaussian functionals that arise in the theory of regularity structures. 
 With rough volatility in mind, an initial observation is that both Gaussian rough paths~$\mathbf{X}$ (Theorem~\ref{Thm: TCIRDE}(1)) and models~$\Pi$ (Theorem~\ref{thm:TCIRV1}(2)) satisfy $(\alpha, c)$-TCIs in rough path/model topology. 
 In both these results,~$\alpha$ and~$c$ reflect the smallest order of Wiener chaos to which the components of~$\Pi$ (or $\mathbf{X}$) belong. Furthermore, we show in Theorem~\ref{Thm: TCIRDE}(2) that solutions of Gaussian RDEs satisfy a similar $(\alpha, c)$-TCI where $\alpha$ and~$c$ are related to the smoothness of the driving path and in particular to the exponent $q$ mentioned above. An $(\alpha, c)$-TCI for a class of modelled distributions, on the rough volatility regularity structure, under the assumption that~$f$ grows at most polynomially is provided in Theorem~\ref{thm:TCImodelled}. 
 In this case, $\alpha$ and~$c$ reflect both the growth of~$f$ and the order of the fixed Wiener chaos to which~$\Pi$ belongs. 
 Finally, we obtain in Theorem~\ref{thm:TCIRV1}(3)-(4) $(\alpha, c)$-TCIs for the driftless log-price $fdB$ in the case of polynomially or exponentially growing~$f$.

Our approach for proving TCIs for the aforementioned functionals  relies on two main steps and is summarised as follows: Letting $(E, \h, \gamma)$ be the abstract Wiener space that corresponds to the underlying Gaussian noise, we consider a functional~$\Psi$, defined on~$E$, along with a shifted version 
$\Psi^h(\omega)=\Psi(\omega+h)$ in the direction of a Cameron-Martin space element $h\in\h$.
First, we obtain estimates of the distance between~$\Psi$ and~$\Psi^h$ in the topology of interest (essentially $\h$-continuity estimates).
Then, we use either a generalised contraction principle, Lemma~\ref{Lem: contraction1}, to obtain an $(\alpha, c)$-TCI with appropriate  cost and deviation functions or a generalised Fernique theorem~\cite{friz2010generalized} to obtain a $1$-TCI for a non-negative function of $\Psi$. Moreover, we show that the $(\alpha, c)$-TCIs we consider imply non-Gaussian concentration properties in Proposition~\ref{alphaprop}. 

A similar methodology can be applied to obtain a different type of inequalities for functionals $\Psi: E\to\R^m$.
As explained in Section~\ref{Section:WLSIs}, if $\Psi$ is $\h$-continuously Fr\'echet (or Malliavin) differentiable then it is possible to obtain Weighted Logarithmic Sobolev Inequalities (WLSIs) via contraction 
(see Proposition~\ref{Prop:logSobolev} for the corresponding contraction principle) under some additional assumptions on the $\h$-gradient. $\WLSIs$ and their connections with concentration properties and weighted Poincar\'e inequalities have been explored by Bobkov-Ledoux~\cite{bobkov2009weighted} and further studied by Cattiaux-Guillin-Wu~\cite{cattiaux2011some} 
(see also~\cite{cattiaux2019entropic, kolesnikov2016riemannian, wang2008super}).

The contribution of this work is thus threefold: 
(a) We prove new TCIs for Gaussian functionals arising in rough volatility and in rough path and regularity structure contexts. 
In passing, Lemma~\ref{Lem: contraction1} and Theorem~\ref{Thm: TCIRDE}(2) extend~\cite{riedel2017transportation} to the setting of RDEs driven by Gaussian noise rougher than Brownian motion;
(b) The $(\alpha, c)$-TCIs we consider imply well-known tail estimates (via Corollary~\ref{Cor:Exponential Moments}): 
(i) the TCIs for Gaussian rough paths and models imply the tail upper bounds for random variables on a fixed Wiener chaos  from~\cite{latala2006estimates};
(ii) the TCI for Gaussian RDEs allows us to recover the tail upper bounds from~\cite{cass2013integrability, friz2013integrability} (Corollary~\ref{cor: LatalaCLL});
(c) Apart from rough volatility, we transfer some of our arguments and prove TCIs in the setting of the 2d-Parabolic Anderson Model (2d-PAM) (Theorem~\ref{thm:PAM}), a well-studied singular SPDE that can be solved in the framework of regularity structures. 
This case highlights significant differences from rough volatility due both to the infinite dimensionality of the dynamics and to the essential requirement for renormalisation needed to define the solution map. 

The rest of this article is organised as follows: We introduce $(\alpha, c)$-TCIs along with their consequences and characterisation in Section~\ref{Section: TCIs}. In Section~\ref{Section: RPs}, we present TCIs for Gaussian rough paths and RDEs. Section~\ref{Section: Rough Vol} is devoted to the rough volatility regularity structure. Apart from presenting our results on TCIs, this section also serves as an elementary introduction to some notions and language of the general theory. In Section~\ref{Section:PAM} we present our results on TCIs for the 2d-PAM. Finally, in Section~\ref{Section:WLSIs} we obtain a generalised contraction principle for WLSIs and leverage tools from Malliavin calculus to demonstrate examples of Gaussian functionals that satisfy such inequalities. The proofs of some technical lemmas from Section~\ref{Section: Rough Vol} are collected in Appendix~\ref{Section:App}.

    \section{Transportation-cost inequalities} \label{Section: TCIs} In this section, we introduce a family of transportation-cost inequalities (TCIs) for probability measures on an arbitrary metric space $\mathcal{X}$. In contrast to the majority of the literature, our definition does not require~$\mathcal{X}$ to be Polish. The reason for choosing this degree of generality is that some of our results apply to situations where the underlying space does not necessarily satisfy this property (e.g. the "total" space of modelled distributions~\eqref{eqref: MDspace} on the rough volatility regularity structure). After introducing the necessary notation, we provide some characterisations and consequences for the class of $(\alpha, c)$-TCIs of interest in Section~\ref{subsec:TCIchar}. In Section~\ref{subsec:TCIcon} we prove a generalised contraction principle which is used to obtain several of our main results in the following sections.

Throughout the rest of this work, the lattice notation $\wedge, \vee$ is used to denote the minimum and maximum of real numbers and~$\lesssim$ denotes inequality up to a multiplicative constant. 
The Borel $\sigma$-algebra and space of Borel probability measures on $\mathcal{X}$ are denoted by  $\mathscr{B}(\mathcal{X}), \mathscr{P}(\mathcal{X})$ respectively.
We use the notation $\nu\ll\mu$ to denote absolute continuity of a measure $\nu$ with respect to $\mu$. For $i=1,\dots, n$, the $i$-th marginal of a measure $\pi\in\mathscr{P}(\mathcal{X}^n)$ is denoted by 
$[\pi]_i$ and the $m$-product measure by $\pi^{\otimes m}$. 
For an interval $I\subset\R$ (resp. $I\subset\R^+$) the convex conjugate (resp. monotone convex conjugate) of a convex function $f:I\rightarrow\R$ is defined for 
$s\in I$ by $f^*(s):=\sup_{t\in I}\{st-f(t)\}$ (resp. $f^\star(s):=\sup_{t\in I}\{st-f(t)\}$).

\begin{dfn} Let $(\mathcal{X}, g)$ be a metric space, 
$c:\mathcal{X}\times\mathcal{X}\rightarrow [0,\infty]$ a measurable function 
and $\mu, \nu\in\mathscr{P}(\mathcal{X})$. 
\begin{enumerate}
\item The transportation cost between~$\mu$ and~$\nu$ with respect to the cost function~$c$ reads
$$
W_{c}(\mu, \nu) :=\inf_{\pi\in\Pi(\mu, \nu)}\iint_{\mathcal{X}\times\mathcal{X}}c(x,y)d\pi(x,y),
$$
where $\Pi(\mu,\nu)$ is the collection of couplings between~$\mu$ and $\nu$:
$$
\Pi(\mu,\nu):=\bigg\{\pi\in\mathscr{P}(\mathcal{X}\times\mathcal{X}): [\pi]_1=\mu,\;[\pi]_2=\nu \bigg\}.
 $$
\item The relative entropy of~$\nu$ with respect to~$\mu$ is given by
\begin{equation}
	H(\nu\;|\; \mu) :=
 \left\{
 \begin{array}{ll}
	\displaystyle
	\large \int_{\mathcal{X}}\log\left(\frac{d\nu}{d\mu}\right)d\nu, & \text{if }\nu\ll\mu,\\
 +\infty, & \text{otherwise}.
\end{array}
\right.
\end{equation}
\end{enumerate}
\end{dfn}
\begin{rem} For $p\in[1,\infty]$, a Polish space~$\mathcal{X}$ and a metric~$g$ that induces the topology of~$\mathcal{X}$, $W_{g^p}^{1/p}$ is the $p$-Wasserstein distance between~$\mu$ and~$\nu$.
\end{rem}

\begin{dfn} 
Let $(\mathcal{X}, g)$ be a metric space and  $\mu\in\mathscr{P}(\mathcal{X})$.
\begin{enumerate}
\item Let $c:\mathcal{X}\times\mathcal{X}\rightarrow[0,\infty]$ be a measurable function with $c(x,x)=0$ for all~$x\in\mathcal{X}$ and $\alpha:[0,\infty]\rightarrow[0,\infty]$ be a lower semicontinuous function with $\alpha(0)=0$.
We say that~$\mu$ satisfies the $(\alpha, c)$-TCI (and write $\mu\in\mathscr{T}_{\alpha}(c)$) with cost function~$c$ and deviation function~$\alpha$ if 
for all $\mathscr{P}(\mathcal{X})\ni\nu\ll\mu$,
	\begin{equation}\label{alphap}
	\alpha\bigg(W_c(\mu, \nu)\bigg)\leq H(\nu\;|\; \mu).
	\end{equation}
\item Let $ p\in[1, \infty)$. We say that~$\mu$ satisfies Talagrand's $p$-TCI and write $\mu\in\mathscr{T}_p(C)$ for some $C>0$ if $\mu\in\mathscr{T}_{\alpha}(g^p)$ with $\alpha(t)=Ct^{2/p}$. 
 \end{enumerate}
\end{dfn}

\subsection{Consequences and characterisation}\label{subsec:TCIchar}
\noindent $(\alpha, c)$-TCIs with convex $\alpha$ and $c$ given by a metric (or a convex function thereof), as well as  the larger family of norm-entropy inequalities, were studied in~\cite{10.1214/ECP.v11-1198, gozlan2007large}. 
Here, we are interested in a class of TCIs where~$\alpha$ is piecewise convex and, for most of the applications of interest, 
$c$ is a concave function of a metric. 
At this point, we provide a characterisation for the TCIs of interest and then show some of their consequences in terms of exponential integrability and deviation estimates.



\begin{prop}\label{alphaprop} Let~$\mathcal{X}$ be a Polish space, $\alpha_1,\alpha_2: [0,\infty]\rightarrow [0,\infty]$ be convex, increasing, continuous functions such that $\alpha_1(0)=\alpha_2(0)=0$ and $c:\mathcal{X}\times\mathcal{X}\rightarrow[0,\infty]$ be lower semicontinuous. The following are equivalent:
\begin{enumerate}
\item[(i)] $\mu\in\mathscr{P}(\mathcal{X})$ satisfies $\mathscr{T}_{\alpha}(c)$ with $\alpha=\alpha_1\wedge\alpha_2$.
\item[(ii)] For all $s\geq 0$ and $f, g\in L^1(\mu)$ such that for $\mu^{\otimes 2}$-almost every $(x,y)\in\mathcal{X}^2$,
	\begin{equation}\label{eq:fgc}
	    f(x)+g(y)\leq c(x,y),
	\end{equation}
we have 
$$
\int_{\mathcal{X}} e^{s g} d\mu\leq 
\exp\left\{-s\int_{\mathcal{X}}fd\mu+ \alpha^\star_1\vee \alpha^\star_2(s)\right\}.
$$
\item[(iii)] Let $g\in L^1(\mu)$ such that $P_c(g)(\cdot):=\sup_{x\in\mathcal{X}}\big\{ g(x)-c(x,\cdot)   \}\in L^1(\mu)$. For all $s\geq 0$,
$$
\int_{\mathcal{X}} e^{s g} d\mu\leq
\exp\left\{s\int_{\mathcal{X}}P_c(g)d\mu+ \alpha^\star_1\vee \alpha^\star_2(s)\right\}.
$$
\end{enumerate}
\end{prop}

\begin{proof} $(i)\iff (ii)$ Let $\nu\in\mathscr{P}(\mathcal{X}), \nu\ll\mu$ and denote by $\tilde{\alpha}$ the extension of $\alpha$ to $\R$ by setting $\tilde{\alpha}=0$ on $(-\infty, 0)$. By Kantorovich duality~\cite[Theorem 2.2]{gozlan2010transport},
$$
W_c(\mu, \nu)=\sup\bigg\{\int_{\mathcal{X}}fd\mu+\int_{\mathcal{X}}gd\nu : f\in L^1(\mu), g\in L^1(\nu),\;f(x)+g(y)\leq c(x,y)\;\;\mu^{\otimes 2}\text{-a.e. on}\; \mathcal{X}^2     \bigg\}.
$$ 
By continuity and monotonicity there exists, modulo re-labelling, $t^*\in[0,\infty]$ such that the set $\{a_1\leq \alpha_2\}$ coincides with $[0, t^*]$.
Since $\tilde{\alpha}$ is non-decreasing, $\mathscr{T}_\alpha(c)$ is equivalent to
\begin{equation}
\tilde{\alpha}\bigg(\int_{\mathcal{X}}fd\mu+\int_{\mathcal{X}}gd\nu  \bigg)\leq H(\nu\;|\;\mu).
\end{equation}
 for all such test functions $f, g$.
 
\noindent \textit{Case} $1:$ Assume that $f,g, \nu$ are such that $\int_{\mathcal{X}}fd\mu+\int_{\mathcal{X}}gd\nu \leq t^*$. Then 
$$\tilde{\alpha}\bigg(\int_{\mathcal{X}}fd\mu+\int_{\mathcal{X}}gd\nu\bigg)=\tilde{\alpha}_1\bigg( \int_{\mathcal{X}}fd\mu+\int_{\mathcal{X}}gd\nu\bigg)$$
and $\tilde{\alpha}_1:\R\rightarrow\R$ is continuous and convex. 
By properties of convex-conjugate functions,  $$
\tilde{\alpha}_1(t)=\tilde{\alpha}_1^{**}(t)=\sup_{s\in\R}\big\{st-\tilde{\alpha}^*_1(s)\big\}
$$
holds for all $t\in\R$, and thus, for all $s\in\R$,
$$s\int_{\mathcal{X}}gd\nu-H(\nu\;|\;\mu)\leq- s\int_{\mathcal{X}}fd\mu+\tilde{\alpha}^*_1(s).$$
\textit{Case} $2:$ $f,g, \nu$ are such that $\int_{\mathcal{X}}fd\mu+\int_{\mathcal{X}}gd\nu \geq t^*$. 
Similarly to Case~1, we obtain 
$$s\int_{\mathcal{X}}gd\nu-H(\nu\;|\;\mu)\leq- s\int_{\mathcal{X}}fd\mu+\tilde{\alpha}^*_2(s).$$
Hence for all $f,g, \nu$ and all $s\in\R$ we have 
$$ s \int_{\mathcal{X}}gd\nu-H(\nu\;|\;\mu)\leq- s\int_{\mathcal{X}}fd\mu+\tilde{\alpha}^*_1(s)\vee \tilde{\alpha}^*_2(s).$$
Optimising over~$\nu$ and noting that $H^*(\cdot|\mu)(g)=\log\int e^gd\mu$ and that for $i=1, 2, s\geq 0$ $\tilde{\alpha}^*_i(s)=\alpha_i^\star(s)$ the conclusion follows.\\
$(ii)\iff (iii)$. Assume~(ii); notice that $h=-P_cg$ is the smallest function  satisfying
$g(x)+h(y)\leq c(x,y)$. The inequality thus follows by applying $(i)$ to $g$ and $f=-P_c(g)$.
The converse follows from the fact that for all~$f$ satisfying~\eqref{eq:fgc}
for some fixed~$g$,  $P_cg\leq -f$.
\end{proof}

\begin{rem}
The previous proposition generalises the characterisation of convex TCIs given in Theorem 3.2. of~\cite{gozlan2010transport} which is obtained by setting $\alpha_1=\alpha_2$.
\end{rem}

In the case where $c=d^p$, for some $p\leq 1$, $\mathscr{T}_{\alpha}(c)$ implies the following  exponential integrability properties.

\begin{cor}\label{Cor:Exponential Moments} 
Let $(\mathcal{X}, d)$ be a Polish space, 
$x_0\in\mathcal{X}$, $\mu\in\mathscr{P}(\mathcal{X})$ such that $d(x_0,\cdot)\in L^1(\mu)$.
If $\mu\in\mathscr{T}_{\alpha}(d^p)$ for some $p\in(0, 1]$ and $\alpha=\alpha_1\wedge\alpha_2$ as in Proposition~\ref{alphaprop}, then
\begin{enumerate}
\item[(i)]
for all $s\geq 0$,
$\int_{\mathcal{X}}\exp\left\{s d^p(x_0, x)\right\}d\mu(x)$ is finite;
\item[(ii)] if for some $t_0\geq 0, C>0$ and all $t\in (0, t_0)$ we have  $\alpha(t)\geq C t^2$, then there exists $\lambda_0$ such that for all $\lambda<\lambda_0$,
$$\int_{\mathcal{X}}\exp\left\{\frac{\lambda^2}{2}d^{2p}(x_0, x)\right\}d\mu(x)<\infty.$$
\end{enumerate}
\end{cor}
\begin{proof}\ 
\begin{enumerate}
\item[(i)]
Let $f=d^p(x_0, \cdot)$, $\langle d^p(x_0, \cdot)\rangle:=\int_{\mathcal{X}}d^p(x_0, y)d\mu(y)$. From the elementary inequality $|x^p-y^p|\leq |x-y|^p$ valid for all $x,y\geq 0$, along with the triangle inequality, then
$$
f(x) - d^p(x,y) = d^p(x, x_0) - d^p(x,y) \leq d^p(x_0, y),
\qquad\text{for all }x\in\mathcal{X}.
$$
Taking the supremum over~$\mathcal{X}$, it follows that $P_cf(y)\leq d^p(x_0, y)$, hence by assumption $P_cf\in L^1(\mu)$. In view of Proposition~\ref{alphaprop}(ii) we obtain 
$$
\int_{\mathcal{X}}\exp\Big\{s d^p(x_0, x)\Big\}d\mu(x)
\leq \exp\Big\{s\langle d^p(x_0, \cdot)\rangle+ \alpha^\star_1\vee \alpha^\star_2(s)\Big\}, $$
and the proof is complete.
\item[(ii)] We apply an argument from~\cite[Page~2704]{djellout2004transportation}.
 Let $\gamma\in\mathscr{P}(\R)$ be a standard Gaussian measure. An application of Fubini's theorem then yields
 \begin{equation}
 \begin{aligned}
 \int_{\mathcal{X}}\exp\left\{\frac{\lambda^2}{2}d^{2p}(x_0, x)\right\}d\mu(x)
 & = \int_{\mathcal{X}}\int_{\mathcal{\R}}\exp\Big\{\lambda s d^{p}(x_0, x)\Big\}d\gamma(s)d\mu(x)\\
 & \leq \int_{\mathcal{\R}}\int_{\mathcal{X}}\exp\Big\{|\lambda s| d^{p}(x_0, x)\Big\}d\mu(x)d\gamma(s)\\
 & \leq    \int_{\mathcal{\R}}     \exp\Big\{|\lambda s|\langle d^p(x_0, \cdot)\rangle+ \alpha^\star_1\vee \alpha^\star_2\big(|\lambda s|\big) \Big\}d\gamma(s),
 \end{aligned}
 \end{equation}
 where the last inequality follows from $(i)$. From the assumptions on $\alpha$ 
 then both~$\alpha_1$ and~$\alpha _2$ are super-quadratic near the origin,
hence $\alpha^*_1, \alpha^* _2$ are  sub-quadratic away from the origin and in particular there exists $C, s_0>0$ such that $\alpha^*_1(s)\vee\alpha_2^*(s)\leq C s^2$ for all $s>s_0$.
The integral in the last display is then clearly finite for $|s| < (1\vee s_0)/|\lambda|$.
As for $| s|\geq    (1\vee s_0)/|\lambda|$ we have  $|\lambda s|\leq \lambda^2 s^2$, 
 $\alpha^\star_1\vee \alpha^\star_2\big(|\lambda s|\big)\leq C\lambda^2 s^2$, 
 hence the integrand is upper bounded by 
$\exp\{(\langle d^p(x_0, \cdot)\rangle + C)\lambda^2 s^2\}$.
Since~$\gamma$ is a Gaussian measure, the latter is integrable, provided that~$\lambda$ is small enough (and in fact for all $|\lambda|<1/\sqrt{2(\langle d^p(x_0, \cdot)\rangle +C})$.
\end{enumerate}
\end{proof}

Apart from exponential integrability, $(\alpha, c)$-TCIs are useful to obtain deviation estimates from the Law of Large Numbers (LLN). In particular, let $\{X_n; n\in\N\}$ be an independent and identically distributed sample from 
a measure $\mu\in\mathscr{P}(X)$ and 
\begin{equation}\label{eq:Ln}
L_n := \frac{1}{n}\sum_{k=1}^{n}\delta_{X_i}\in\mathscr{P}(\mathcal{X})
\end{equation}
the $n$-sample empirical measure. A consequence of Sanov's theorem in large deviations ( see e.g. \cite{dupuis2011weak}, Theorem 2.2.1)  is  that, for all $r>0$ and $d$ a metric for the topology of weak convergence in $\mathscr{P}(\mathcal{X})$,
\begin{equation}
\limsup_{n\to\infty}\log \pr\Big[ d(L_n, \mu)\geq r    \Big]
\leq -\inf_{\{\nu : d(\nu, \mu)\geq r\}} H(\nu\;|\;\mu).
\end{equation}
In other words, $H(\cdot\;|\;\mu)$ provides an asymptotic, exponential decay rate for the probability of being "far" from the LLN limit $\mu$. In many practical applications, the pre-asymptotic terms that are ignored from large deviation estimates play an important role. The following proposition and Corollary~\ref{cor:deviation}  show that it is possible to get non-asymptotic (i.e. for all $n$ as opposed to "large" $n$) deviation estimates under the assumption that $\mu\in\mathscr{T}_{\alpha}(c)$.

\begin{prop}\label{prop:deviation}(Deviation estimates)
Let~$\mathcal{X}$ be a Polish space, $\mu\in\mathscr{P}(\mathcal{X})$ and $L_n$ as in~\eqref{eq:Ln}. The following are equivalent:
\begin{enumerate}
\item[(i)] $\mu\in\mathscr{T}_{\alpha}(c)$ with $\alpha=\alpha_1\wedge \alpha_2$ as in Proposition~\ref{alphaprop}.
\item[(ii)] For all $f,g\in L^1(\mu)$ such that $f(x)+g(y)\leq c(x,y)$ for $\mu^{\otimes 2}$-almost every $(x,y)\in\mathcal{X}^2$, then, 
for all $n\in\N$ and $r>0$,
$$ \frac{1}{n}\log \pr\left[ \int_{\mathcal{X}}fdL_n+\int_{\mathcal{X}}gd\mu\geq r   \right]
 = \frac{1}{n}\log \pr\left[ \frac{1}{n}\sum_{k=1}^{n}f(X_k)+\int_{\mathcal{X}}gd\mu\geq r   \right]
 \leq -\alpha(r).
$$
\end{enumerate}
\end{prop}
\begin{proof} The proof is identical to that of Theorem 2 from~\cite{gozlan2007large} with the difference that $\alpha$ is not convex but rather piecewise convex. To avoid repetition, we shall only sketch the main steps. To this end, we have from Kantorovich duality, monotonicity and continuity of $\alpha$ that
$$\mu\in\mathscr{T}_\alpha(c)\iff 
\tilde{\alpha}\bigg(\int_{\mathcal{X}}fd\mu+\int_{\mathcal{X}}gd\nu  \bigg)\leq H(\nu\;|\;\mu)
$$
for all $f, g$ that satisfy the assumptions of (2) and $\tilde{\alpha}$ the extension of $\alpha$ to $\R$ by $0$. In turn the latter is equivalent to 
$$
t\tilde{\alpha}(t)\leq\inf\bigg\{ H(\nu\;|\;\mu) :\nu \;\text{s.t.} \int_{\mathcal{X}} f d\nu + \int_{\mathcal{X}} g d\mu=t\bigg\}=\Lambda_{\phi}^*(t), 
\quad\text{for all }t\in\R,
$$
where for each fixed $\phi=(f,g)$  (see Equation~(21) from the aforementioned reference) 
$\Lambda_{\phi}^*$ is the convex conjugate of the log-Laplace transform
$\Lambda(s)=\int_{\mathcal{X}}\exp\{ sf(x)+ \int gd\mu \}d\mu(x)$
(this is essentially a consequence of Cram\'er's theorem for large deviations of iid random variables, \cite{dupuis2011weak}, Theorem 3.5.1).
The forward implication is then complete by a Markov inequality argument which makes no use of convexity for~$\alpha$ and is thus omitted. For the converse, one has that the deviation estimate in (2) implies $\alpha\leq \Lambda_{\phi}^*$ (which is also independent of convexity assumptions on $\alpha$) which in turn is equivalent to $\mu\in\mathscr{T}_{\alpha}(c)$.    
\end{proof}

\begin{cor}\label{cor:deviation}
Let $(\mathcal{X}, d)$ be a Polish space, $\mu\in\mathscr{P}(\mathcal{X})$, $L_n$ as in~\eqref{eq:Ln} and  $x_0\in\mathcal{X}$ such that $d(x_0,\cdot)\in L^1(\mu)$.
Moreover, assume that $\mu\in\mathscr{T}_{\alpha}(d^p)$ for some $p\in(0, 1]$, 
with $\alpha:=\alpha_1\wedge \alpha_2$ a super-quadratic deviation function as in Corollary~\ref{Cor:Exponential Moments}(ii).
Then there exist $C_p, s_0>0$ such that for all $n\in\N$ and $s>s_0$,
\begin{equation}
\begin{aligned}
 \frac{1}{n}\log \pr\left[\frac{1}{n}\sum_{k=1}^{n}d(X_k,x_0)\geq s \right] \leq -C_ps^{2p}.
\end{aligned}
\end{equation}
\end{cor}

\begin{proof} The functions $f=-g=d^p(x_0, \cdot)$ satisfy all the assumptions of Proposition~\ref{prop:deviation}. Indeed,~\eqref{eq:fgc} holds by triangle inequality and integrability is satisfied by assumption. 
With this choice and another application of the triangle inequality we obtain for all $s>\ex^{\mu} g$,
\begin{equation}
\begin{aligned}
 \frac{1}{n}\log \pr\bigg[ \frac{1}{n}\sum_{k=1}^{n}d^p(X_k,x_0)\geq 2s   \bigg]\leq \frac{1}{n}\log \pr\bigg[ \frac{1}{n}\sum_{k=1}^{n}f(X_k)-\int_{\mathcal{X}}gd\mu\geq s   \bigg]\leq -\alpha(s).
\end{aligned}
\end{equation}
From H\"older's inequality and the super-quadratic growth of $\alpha$ it follows that 
\begin{equation}
\frac{1}{n}\log \pr\bigg[ \frac{1}{n}\sum_{k=1}^{n}d(X_k,x_0)\geq (2s)^{1/p}   \bigg]\leq -Cs^2,
\end{equation}
and the proof is complete by substituting~$S$ by $(s/2)^p$.
\end{proof}

\begin{rem} The assumption that~$\mathcal{X}$ is Polish is sufficient to guarantee the validity and well-posedness of the Kantorovich-dual formulation of the transportation cost~$W_c$,
used implicitly for example in the proof of Proposition~\ref{alphaprop}.
\end{rem}

\subsection{A generalised contraction principle }\label{subsec:TCIcon} Contraction principles for Talagrand's inequalities~\eqref{eq: pTCI} have been proved in~\cite[Lemma 2.1]{djellout2004transportation} and~\cite[Lemma 4.1]{riedel2017transportation}. 
The previous results concern Lipschitz maps of measures that satisfy Talagrand's $\mathscr{T}_{2}(C)$. 
We now prove a generalised contraction principle for maps that satisfy a certain type of "uniform continuity" condition.
We start with an assumption on the domain space of the contraction principle.
\begin{assu}\label{assu:Setting}
$\mathcal{X}$ is a Polish space, $c_{\mathcal{X}}: \mathcal{X}\times\mathcal{X}\rightarrow[0,\infty]$ is a measurable function and 
there exists a measure $\mu\in\mathscr{P}(\mathcal{X})$ and a constant $C>0$ such that, 
for every $\nu \in\mathscr{P}(\mathcal{X})$,
$$  \left(\inf_{\pi\in\Pi(\nu,\mu)}\iint_{\mathcal{X}\times\mathcal{X}}c^2_{\mathcal{X}}(x_1, x_2)d\pi(x_1,x_2)\right)^{\half}\leq \sqrt{CH(\nu\; |\; \mu)}.
$$
\end{assu}

\begin{lem}[Extended contraction principle]\label{Lem: contraction1}
Under Assumption~\ref{assu:Setting},
let $\mathcal{Y}$ be a metric space, $c_{\mathcal{Y}}: \mathcal{Y}\times\mathcal{Y}\rightarrow[0,\infty]$ a measurable function and assume that there exists a measurable map $\Psi:\mathcal{X}\rightarrow\mathcal{Y}$ and $r\geq 1$ such that for all $x_1, x_2\in\mathcal{X}_0\subset \mathcal{X}, $ with $\mu(\mathcal{X}_0)=1$,
\begin{equation}
c_{\mathcal{Y}}\big(\Psi(x_1),  \Psi(x_2)\big)\leq L(x_1) \bigg[c_{\mathcal{X}}(x_1, x_2)\vee c_{\mathcal{X}}(x_1, x_2)^{\frac{1}{r}}\bigg],
\end{equation}
where $L\in L^{p^*}(\mathcal{X}, \mu)$ for  $p^*=2\vee(r/r-1)$. 
Then $\tilde{\mu}=\mu\circ\Psi^{-1}\in\mathscr{T}_{\alpha}(c_{\mathcal{Y}})$, where, for some constant $C>0$, $\alpha(t)=Ct^2\wedge t^{2r}$.
\end{lem}

\begin{proof}	
Without loss of generality we may assume $C = 1$. Let $\tilde{\nu}\in\mathscr{P}(\mathcal{Y})$ and assume that $H(\tilde{\nu}\;|\;\tilde{\mu})$ is finite.
Choose $\nu\in\mathscr{P}(\mathcal{X})$ such that $\tilde{\nu}=\nu\circ\Psi^{-1}$ and   $\nu\ll \mu$  (note that there is at least one~$\nu$ which fulfills this condition;  e.g. $\nu_0(dx):= d\tilde{\nu}/d\tilde{\mu} (\Psi(x))\mu(dx))$. 
Then, an application of H\"older's inequality yields
\begin{equation}
\begin{aligned}
\inf_{\tilde{\pi}\in\Pi(\tilde{\nu},\tilde{\mu})}\iint_{\mathcal{Y}\times\mathcal{Y}}c_{\mathcal{Y}}(y_1, y_2)&d\tilde{\pi}(y_1,y_2)\leq \inf_{\pi\in\Pi(\nu,\mu)}\iint_{\mathcal{X}\times\mathcal{X}}c_{\mathcal{Y}}\big(\Psi(x_1), \Psi(x_2)\big)d\pi(x_1,x_2)\\&
\leq \inf_{\pi\in\Pi(\nu,\mu)}\iint_{\mathcal{X}\times\mathcal{X}}L(x_1)\bigg[c_{\mathcal{X}}(x_1, x_2)\vee c_{\mathcal{X}}(x_1, x_2)^{\frac{1}{r}}\bigg]d\pi(x_1,x_2)\\&
= \inf_{\pi\in\Pi(\nu,\mu)}\iint_{\{c_{\mathcal{X}}\leq 1\}}L(x_1) c_{\mathcal{X}}(x_1, x_2)^{\frac{1}{r}}d\pi(x_1,x_2)\\&
+\inf_{\pi\in\Pi(\nu,\mu)}\iint_{\{c_{\mathcal{X}}>1\}}L(x_1) c_{\mathcal{X}}(x_1, x_2)d\pi(x_1,x_2)\\&
\leq \|L\|_{L^{r/(r-1)}}\bigg(\inf_{\pi\in\Pi(\nu,\mu)}\iint_{\{c_{\mathcal{X}}\leq 1\}}c_{\mathcal{X}}(x_1, x_2)d\pi(x_1,x_2)\bigg)^{\frac{1}{r}}\\&+
\|L\|_{L^{2}}\bigg(\inf_{\pi\in\Pi(\nu,\mu)}\iint_{\{c_{\mathcal{X}}> 1\}}c_{\mathcal{X}}(x_1, x_2)^{2}d\pi(x_1,x_2)\bigg)^{\frac{1}{2}}%
\\&
\lesssim\bigg(H(\nu\; |\; \mu)\bigg)^{\frac{1}{2r}}+
\bigg(H(\nu\; |\; \mu)\bigg)^{\frac{1}{2}}\lesssim \bigg(H(\nu\; |\; \mu)\bigg)^{\frac{1}{2r}}\vee\bigg(H(\nu\; |\; \mu)\bigg)^{\frac{1}{2}}.
\end{aligned}
\end{equation}
The proof is complete upon invoking the identity 
\begin{equation}\label{entropyid}
H(\tilde{\nu}\;|\;\tilde{\mu})=\inf\big\{ H(\nu\; |\; \mu); \nu\in\mathscr{P}(\mathcal{X}): \nu\circ\Psi^{-1} =\tilde{\nu}    \big\}
\end{equation}
which holds when~$\mathcal{X}$ is a Polish space.
\end{proof}
\begin{rem}
The previous lemma is used in the proof of most of our main results. 
We emphasise here that for the identity~\eqref{entropyid} to hold, it is sufficient to require that~$\mathcal{X}$ is Polish (and not~$\mathcal{Y})$. In view of the latter, the same is true for the contraction principle.
\end{rem}

\section{TCIs for Gaussian rough differential equations}\label{Section: RPs} Our first result concerns the solutions of Rough Differential Equations (RDEs) driven by a Gaussian process with continuous paths. Throughout this section, $\Cc^{\text{p-var}}(I;\R^d)$ is the Banach space of continuous $\R^d$-valued paths of finite $p$-variation, defined on the compact interval $I\subset[0,\infty];$ the $p$-variation distance is denoted by $g_{\text{p-var}}$. 

\begin{dfn}\label{dfn:RP} (Gaussian rough paths)  Let $T>0, p\in[1, 3)$ and~$X$ be a $d$-dimensional, continuous Gaussian process on $[0,T]$ with paths of finite $p$-variation. 
\begin{enumerate}
\item A geometric $p$-rough path $\mathbf{X}$ over~$X$ is a pair 
$$
\mathbf{X}=\mathscr{L}(X):=\left(X, \mathbb{X}\right)\in C\left([0,T]^2;\R^d\oplus \R^{d\otimes d}\right),
$$
such that the following hold $\pr$-almost surely:
\begin{enumerate}
    \item (Chen's relation) $X_{s,t}=X_{s,u}+X_{u,t}$ and $\mathbb{X}_{s,t}=\mathbb{X}_{s,u}+\mathbb{X}_{u,t}+X_{s,u}\otimes X_{u,t} $ for all $0\leq s\leq u\leq t\leq T$.
    \item ($p$-variation regularity) $\|\mathbf{X}\|_{\text{p-var}}^p:=\sup_{(t_i)\in\mathcal{P}[0,T]}\sum_{i}\bigg(|X_{t_{i},t_{i+1}}|+\big|\mathbb{X}_{t_i,t_{i+1}}\big| \bigg)^{p}$ is finite,
where $\mathcal{P}[0,T]$ is the collection of finite dissections of $[0,T]$.
\end{enumerate}
\item The inhomogeneous $p$-variation metric $\mathbf{g}_{\text{p-var}}$ is defined for two geometric $p$-rough paths by $\mathbf{g}_{\text{p-var}}(\mathbf{X}, \mathbf{Y}) := \|\mathbf{X}-\mathbf{Y}\|_{\text{p-var}}$.
\item The space $\mathcal{D}^{0,p}_g([0,T];\R^m)$ of geometric $p$-rough paths is defined as the completion of the set $\{ \mathscr{L}(f), f\in C^{\infty}\}$ with respect to $\mathbf{g}_{\text{p-var}}$.
\end{enumerate}
\end{dfn}

\begin{rem}
The metric space  $(\mathcal{D}^{0,p}_g([0,T];\R^d), \mathbf{g}_{\text{p-var}}$) is Polish~\cite[Proposition~8.27]{friz2010multidimensional}. 
The second-order process $\mathbb{X}$ is typically given by the iterated integral $\mathbb{X}_{s,t}=\int_{s}^{t}(X_r-X_s)\otimes dX_r$ which is defined as a limit (in probability) of piecewise linear approximations.
\end{rem}

\begin{thm}\label{Thm: TCIRDE} Let $T>0, p\in[1, 3)$ and $X=(X_t)_{t\in[0,T]}$ be a $d$-dimensional, continuous, mean-zero Gaussian process with Cameron-Martin space $\h$ such that 
\begin{enumerate}
\item[(i)] $X$ has a natural lift to a geometric $p$-rough path $\mathbf{X}$;
\item[(ii)] there exists $q$ with $\frac{1}{p}+\frac{1}{q}>1$ such that $\h\hookrightarrow \Cc^{q-var}([0,T];\R^d)$. 
\end{enumerate}
Next, let $\gamma>p$, $V=(V_1, \dots, V_d)$ be $Lip^\gamma$-vector fields on $\R^d$  \cite[Definition~10.2]{friz2010multidimensional} and consider the solution $(Y_t)_{t\in[0,T]}$ of the RDE
\begin{equation}\label{eq:RDE1}
    dY_t= V(Y_t)d\mathbf{X}_t, 
    \qquad Y_0\in\R^m.
\end{equation}
Then the following hold:
\begin{enumerate}
    \item The law $\mu\in\mathscr{P}(\mathcal{D}^{0,p}_g[0,T])$ of~$\mathbf{X}$ 
    satisfies $\mathscr{T}_{\alpha}(\mathbf{g}_{\text{p-var}}^{1/2})$  with $\alpha(t)=C(t\wedge t^2)$, $C>0$.
    \item The law $\mu\in \mathscr{P}(\Cc^{\text{p-var}}([0,T];\R^m))$ of~$Y$ 
    satisfies $\mathscr{T}_{\alpha}(g_{\text{p-var}}^{1/q})$ with $\alpha(t)=C(t^{2q}\wedge t^2)$, $C>0$.
\end{enumerate}
\end{thm}

\begin{rem} Talagrand's inequalities for $q=1$, 
which corresponds to Brownian or "smoother" paths, have been proved in~\cite[Theorem~2.14]{riedel2017transportation}. 
Our result shows that solutions of RDEs with "rougher" drivers (e.g. fBm with Hurst parameter 
$H\in (\frac{1}{3},\half)$) also satisfy TCIs with a different cost and deviation function. In fact, setting $q=1$, we recover $\mathscr{T}_1(C)$ from~\cite{riedel2017transportation}.
\end{rem}

\begin{proof} (1) Let $s,t\in[0,T]^2$ and 
$$
T_h\mathbf{X}_{s,t} := \bigg(X_{s,t}+h_{s,t}, \mathbb{X}_{s,t}+\int_{s}^{t} h_{s,r}\otimes dX_r+\int_{s}^{t} X_{s,r}\otimes dh_r+\int_{s}^{t} h_{s,r}\otimes dh_r\bigg)
$$ denote the translation of~$\mathbf{X}$ in the direction of $h\in\h$ (note that the assumptions on $X$ guarantee that the last two integrals on the right-hand side are well-defined Young integrals).
From~\cite[Lemma~2.10]{riedel2017transportation}, 
we have for some constant $C>0$
$$
\mathbf{g}_{\text{p-var}}(T_h\mathbf{X}, \mathbf{X})^{1/2}\leq C(1\vee\|x\|_{\text{p-var}})\|h\|_{\h}\vee \|h\|^{1/2}_{\h}.
$$ 
Appealing to Lemma~\ref{Lem: contraction1} with $\mathcal{X}=\Cc^{\text{p-var}}([0,T];\R^m)$, $\mathcal{Y}=\mathcal{D}^{0,p}_g[0,T]$,
$\Psi=\mathbf{X}$ and $c_{\mathcal{Y}}$ be the Cameron-Martin pseudometric on $\mathcal{D}^{0,p}_g[0,T]$ the conclusion follows.\\
(2) Let $Y^h$ be the solution of~\eqref{eq:RDE1} driven by $T_h\mathbf{X}$. From~\cite[Lemma~2.11]{riedel2017transportation}, we have
$$
g_{\text{p-var}}(Y^h, Y)^{1/q}\leq C\exp\Big(N_1(\mathbf{X};[0,T])+1\Big)\|h\|_{\h}\vee \|h\|^{1/q}_{\h},
$$ 
for some constant $C>0$, where the random variable $N_1(\mathbf{X};[0,T])$ has finite moments of all orders. Appealing to Lemma~\ref{Lem: contraction1} with $\mathcal{X}=\mathcal{D}^{0,p}_g[0,T], \mathcal{Y}=\Cc^{\text{p-var}}([0,T];\R^m), \Psi=Y$ and $c_{\mathcal{Y}}$ the Cameron-Martin pseudometric on $\mathcal{D}^{0,p}_g[0,T]$ the result follows.   
\end{proof}

\noindent The well-known estimates of~\cite[Theorem 6.23]{cass2013integrability} (see also~\cite{friz2013integrability} and the lower bound from~\cite[Theorem 1.1]{boedihardjo2022lack}) show that the laws of solutions of RDEs with bounded and sufficiently smooth vector fields are Weibull-tailed with shape parameter $2/q$. This non-Gaussian tail behaviour is a consequence of rough integration which takes into account not just the noise~$X$ but also iterated integrals of~$X$ with itself.   Due to the lack of Gaussian integrability, such measures are not expected to satisfy Talagrand's $\mathscr{T}_r(C)$ inequalities for any $r\geq 1$. Nevertheless, the more general TCI $\mathscr{T}_\alpha(c)$ allows us to recover the "correct" tail behaviour. Indeed, we have the following:

\begin{cor}\label{cor: LatalaCLL} 
For $X, Y, q$ as in Theorem~\ref{Thm: TCIRDE}, there exist $C_1, C_2>0$ such that for all $R>0$,
$$
\pr\bigg[ \|\mathbf{X}\|_{\text{p-var}}\geq R \bigg]   \leq C_1e^{-C_2 R}
\qquad\text{and}\qquad
\pr\bigg[ \|Y\|_{\text{p-var}}\geq R \bigg]   \leq C_1 e^{-C_2 R^{2/q}}.
$$
\end{cor}
\begin{proof}
From Theorem~\ref{Thm: TCIRDE}(1) and Corollary~\ref{Cor:Exponential Moments}(ii) with 
$\mathcal{X}=\mathcal{D}_g^{0,p}[0,T]$, $x_0=0$, $d=\mathbf{g}_{\text{p-var}}$, $p=\half$, 
there exists $\lambda_0>0$ such that
$\ex[\exp(  \lambda\|\mathbf{X}\|_{\text{p-var}})]$ is finite for all $\lambda<\lambda_0$.
An application of Markov's inequality allows us to conclude that 
$$
\pr\bigg[ \|\mathbf{X}\|_{\text{p-var}}\geq R \bigg]\leq e^{-\lambda R}\ex\Bigg[\exp\bigg(  \lambda\|\mathbf{X}\|_{\text{p-var}}  \bigg)\Bigg].
$$
Similarly to Theorem~\ref{Thm: TCIRDE}(2) and Corollary~\ref{Cor:Exponential Moments}(ii) with 
$\mathcal{X}=\Cc^{\text{p-var}}[0,T]$, $x_0=0$, $d=g_{\text{p-var}}$, $p=1/q$, 
there exists $\lambda_0>0$ such that 
$
\ex[\exp(\lambda\|Y\|^{2/q}_{\text{p-var}})]$ is finite for all $\lambda<\lambda_0$.
Once again the conclusion follows from Markov's inequality
$$
\pr\bigg[ \|Y\|_{\text{p-var}}\geq R \bigg]\leq e^{-\lambda R^{2/q}}\ex\Bigg[\exp\bigg(  \lambda\|Y\|^{2/q}_{\text{p-var}}  \bigg)\Bigg].
$$
\end{proof}

\begin{rem} Examples of Gaussian processes~$X$ that satisfy the assumptions of Theorem~\ref{Thm: TCIRDE} include Brownian motion (in which case~$Y$ is interpreted in the Stratonovich sense), 
fractional Brownian motion with Hurst exponent $H\in(\frac{1}{3}, 1)$,
Brownian (and more generally Gaussian) bridges, Ornstein-Uhlenbeck processes driven by Brownian motion and bifractional Brownian motion~\cite[Example~2.6]{riedel2017transportation} with parameters $H,K$ satisfying $HK\in(\frac{1}{3}, 1)$.
\end{rem}

\section{TCIs for regularity structures: Rough Volatility }\label{Section: Rough Vol}
In this section we focus on a simple rough volatility model.
To this end, let $H\in(0,\half]$ and consider the evolution of log-prices governed by the stochastic differential equation
	\begin{equation}\label{roughvolmodel}
	    dS_t/S_t= f\left( \widehat{W}_t^H, t\right)dW_t,
     \qquad S_0=s_0 > 0.
	    \end{equation}
Here~$W$ is a standard Brownian motion, $fdW$ is an It\^o integral and $\widehat{W}^H$ is a Riemann-Liouville (or type-II) fractional Brownian motion, 
in particular,  
$$
\widehat{W}^H_t = \int_{0}^{t}K^H(t-r)dW_r, \qquad \text{for }t\geq 0,
$$
	    where $K^H:[0,\infty)\rightarrow\R$ denotes the power-law Volterra kernel $K^H(t)=\sqrt{2H}t^{H-\half}$. 
	    As is well known, the solution map of an It\^o SDE is not continuous with respect to the driving Brownian motion~$W$.
     The theory of rough paths, initially developed by Lyons~\cite{lyons1998differential}, provides a remedy for this lack of continuity for a large class of SDEs. In particular, Lyons' universal limit theorem~\cite[Theorem 8.5]{friz2020course} asserts that the solution of an SDE is a continuous image of the canonical rough path lift of the noise with respect to an appropriate rough path topology.  
	    
As explained in~\cite[Section 2]{bayer2020regularity}, the SDE~\eqref{roughvolmodel} is beyond the reach of rough paths theory since $\widehat{W}^H$ and~$W$ are not independent and because calibrated rough volatility models feature~\cite{gatheral2018volatility} Hurst indices $H<\frac{1}{4}$. 
Nevertheless, as shown in~\cite{bayer2020regularity}, the continuity of the rough volatility solution map can be recovered in the framework of Hairer's theory of regularity structures~\cite{hairer2014theory}. 
The main idea is to enhance the noise $"dW"$ with sufficient higher-order functionals defined with respect to a fixed regularity structure~$\Tt$. 
The enhanced solution map is then continuous with respect to the topology of models defined on~$\Tt$ (see Definition \ref{dfn:modeltop}) .
	    
\begin{dfn} Let $A\subset\R$ be a locally finite (e.g. discrete) index set that is bounded from below. A regularity structure is defined as a pair $(\Tt, G)$ of a vector space~$\Tt$ (the structure space) and a group~$G$ (the structure group) with the following properties: 
\begin{enumerate}
\item[(i)] $\Tt=\bigoplus_{\alpha\in A} \Tt_\alpha$, where for each $\alpha\in A, \Tt_\alpha$   is a Banach space. 
Each element $\tau\in \Tt_a$ is said to have degree (or homogeneity) $\alpha$ and we write $|\tau|=\alpha$. For each $\tau\in\Tt$, $\|\tau\|_\alpha$ denotes the norm of the component of~$\tau$ in~$\Tt_{\alpha}$.
\item[(i)] $G$ is a group of linear transformations on $\Tt$ such that
for each $\Gamma\in G$, $\alpha\in A$, $\tau_\alpha\in\Tt_\alpha$,
\begin{equation}\label{reexpansion}
\Gamma\tau_\alpha-\tau_{\alpha}\in\bigoplus_{\beta<\alpha}\Tt_{\beta}.
\end{equation}
\end{enumerate}
\end{dfn}
    
       \begin{rem} A useful analogy is that of a regularity structure as an abstraction of Taylor expansions. In particular, one can think that for each $\alpha\in A$, $ \Tt_\alpha$ and $\Tt$ contain monomials of degree $\alpha$ and abstract Taylor polynomials respectively. The action of $G$ on $\Tt$ can be thought of as "re-expansion" of a Taylor polynomial with respect to a different base point. Then, at a formal level,~\eqref{reexpansion} expresses the fact that the difference between a monomial of degree $\alpha$ and its re-expanded version will be a polynomial of degree  $\beta<\alpha$. For example, re-expanding the second  degree monomial $x^2\in\Tt_2$ around $1$ gives us
       $$\Gamma_1x^2-x^2=(x-1)^2-x^2=-2x+1\in\Tt_0\oplus\Tt_1=\bigoplus_{\beta<2}\Tt_\beta.$$
       \end{rem}
 
\subsection{ The rough volatility regularity structure}\label{rvolrs}
We now define the concrete rough volatility regularity structure, tailor-made to~\eqref{roughvolmodel}, as constructed in~\cite{bayer2020regularity}. 
First, we introduce a finite set of symbols that provide the building blocks for the abstract Taylor expansions. 
To this end, let $M\in\N, \kappa\in(0, H)$ and
$$
\mathcal{S} := \Big\{\mathbf{1}, \Xi, I(\Xi),\dots, I(\Xi)^M, \Xi I(\Xi), \Xi I(\Xi)^2,\dots, \Xi I(\Xi)^M\Big\}.
$$
Here, the symbol~$\Xi$ corresponds, up to realisation 
(Definition~\ref{modeldef}) to the underlying noise $dW=\dot{W}$ and $I$ denotes convolution with respect to the kernel $K^H$. The degrees of the symbols are postulated as follows: 
\begin{center}
\begin{tabular}{ccccccc}
\hline
Symbol & $\mathbf{1}$ & $\Xi$ & $I(\Xi)$ & $I(\Xi)^M$  & $\Xi I(\Xi)^M$\\
Degree & 0 & $-\half-\kappa$ & $\displaystyle H-\kappa$ & $\displaystyle M(H-\kappa)$ & $M(H-\kappa)-\half-\kappa$\\
\hline
\end{tabular}
\end{center}
The number~$M$ is chosen to be the smallest number for which $I(\Xi)^{M_+1}\Xi$ has a positive homogeneity (more precisely, this choice implies that the modelled distribution lift of $fdW$ belongs to a space $\mathcal{D}_T^\gamma(\Gamma)$ of positive regularity with $\gamma>0$, see Definition \ref{Ddef},~\eqref{eq:Dgammanorm} below)  to the reconstruction theorem for modelled distributions  Thus it suffices to consider $M$ such that $| \Xi I(\Xi)^{M+1}|=(M+1)(H-\kappa)-\half-\kappa>0$. Thus we take 
\begin{equation}\label{Mchoice}
M:= \Mf(H,\kappa):= \max\bigg\{ m\in\N: m(H-\kappa)-\frac{1}{2}-\kappa\leq 0\bigg\}.
\end{equation}
\begin{rem}
For very small~$\kappa \in (0,H)$, we have that
$m(H-\kappa)-\frac{1}{2}-\kappa\leq 0$
if and only if
$m \leq (\kappa+\frac{1}{2})/((H-\kappa)) =: M$, which tends to $\frac{1}{2H}$ from above.
When~$\kappa$ is close to~$H$, say of the form $\kappa = H-\varepsilon$,
then 
$m(H-\kappa)-\frac{1}{2}-\kappa\leq 0$
if and only if
$m \leq (H-\varepsilon+\frac{1}{2})/\varepsilon =: M$, which tends to infinity.
\end{rem}
The index set and structure space are then defined by $A:=\{|\tau|\;: \tau\in\mathcal{S}\}$ and \begin{equation}\label{Tvol}
\Tt=span\{\mathcal{S}\}=\bigoplus_{\alpha\in A}\Tt_\alpha:=\bigoplus_{\alpha\in A}span\{\tau\in\mathcal{S}: |\tau|= \alpha  \}.
\end{equation}
Turning to the structure group, we let $G:=\{\Gamma_h\;, h\in\R\}\subset\mathscr{L}(\Tt)$ such that for all $h\in\R$: 
\begin{equation}\label{eq:structuregroup1}\Gamma_h\mathbf{1}=\mathbf{1},\qquad 
\Gamma_h\Xi=\Xi,\qquad
\Gamma_hI(\Xi)=I(\Xi)+h\mathbf{1},
\end{equation}
and 
\begin{equation}\label{eq:structuregroup}
\Gamma_h\tau\tau'=\Gamma_h\tau\Gamma_h\tau',
\end{equation}
for all $\tau,\tau'\in\mathcal{S}$ such that $\tau\tau'\in\mathcal{S}$. The expression $\tau\tau'$ should be interpreted as a formal product between symbols. The maps $\Gamma_h$ are then extended to $\Tt$ by linearity. The group property of $G$ is inherited by the additive structure of the real numbers.

 A model for a regularity structure is a concrete interpretation of the abstract Taylor polynomials and their re-expansion rules. Part of the flexibility of the theory is owed to the fact that monomials can be very irregular functions and even Schwartz distributions.   In the next definition, $\mathscr{L}(X,Y)$ denotes the space of continuous, linear operators between two topological vector spaces~$X$ and~$Y$, $X^*$ denotes the continuous dual space of~$X$ and $\langle \cdot,\cdot \rangle: X^*\times X\rightarrow \R$  is the duality pairing $\langle x^*, x \rangle:=x^*(x)$.

\begin{dfn}\label{modeldef} Let $(\Tt, G)$ be the rough volatility regularity structure. A model for $(\Tt, G)$ over $\R$ is a pair $(\Pi, \Gamma)$ of "realisation" and "re-expansion" maps $$\Pi: \R \rightarrow\mathscr{L}(\Tt; (\Cc^\infty_c(\R)^*)\quad\Gamma:\R\times\R\rightarrow G, $$
that satisfy the following properties:
\begin{itemize}
\item For $s,t,z\in\R$, the abstract "Chen's relation" $\Pi_t=\Pi_s\Gamma_{s,t}$ holds and $\Gamma_{s,t}=\Gamma_{s,z}\Gamma_{z,t}$;
\item for all $\tau\in\Tt, \lambda\in(0,1)$,  $s, t$ in a compact set, $\Gamma_{s,t}\in G$, $\phi^\lambda_s(\cdot)=\lambda^{-1}\phi(\lambda^{-1}(\cdot-s))\in C_c^\infty(\R)$ with $\supp(\phi)\subset (-1, 1)$
	\begin{equation}
	\big| \langle \Pi_s\tau, \phi^\lambda_s\rangle \big|\lesssim\lambda^{|\tau|},\quad \Gamma_{s,t}\tau=\tau+\sum_{\tau': |\tau'|<|\tau|}c_{\tau'}(s,t)\tau',
	\end{equation}
	with  $|c_{\tau'}(s,t)|\lesssim |s-t|^{|\tau|-|\tau'|}$.
\end{itemize}
\end{dfn}

 We shall now introduce a (random) It\^o model $(\Pi, \Gamma )$, defined on an underlying probability space $(\Omega, \F, \pr)$, for the rough volatility regularity structure $\Tt$~\eqref{Tvol}. With~$W$ a Brownian motion extended to $\R$ by letting $W=0$ on $(-\infty, 0]$ and $t\geq 0$, let   $$\Pi_t\mathbf{1}=1,\;\Pi_t\Xi:= \dot{W}=(d/dt)W,$$ where the derivative is meant in the sense of distributions, i.e. for each $\phi\in \Cc^\infty_c(\R)$ we have 
\begin{equation}\label{eq:Itomodel1}
    \langle \Pi_t\Xi, \phi\rangle=\langle \dot{W}, \phi\rangle=-\int_\R W_s\dot{\phi}(s)ds.
\end{equation}
Note that we can restrict the white noise $\dot{W}$ to any finite time interval $[0,T]$ by setting $W=0$ on $\R\setminus[0,T]$ and due to stationarity the value of $\Pi_t\Xi$ does not depend on~$t$.
Next, let
\begin{equation}
    \label{eq:IXi}
    \Pi_t I(\Xi) := K^H*\dot{W}-K^H*\dot{W}_t = \widehat{W}^H-\widehat{W}^H(t),
\end{equation}
and for each $m=2,\dots, M$, 
$\Pi_t I(\Xi)^m = (\widehat{W}^H-\widehat{W}^H(t))^M$.
Similar to the rough path framework, one then considers It\^o integrals 
\begin{equation}
\mathbb{W}^m_{s,t}:=\int_{s}^{t}( \widehat{W}^H(r)-\widehat{W}^H(s) )^mdW_r\;, s<t,\; m=1,\dots, M,
\end{equation}
and defines
$$\Pi_s\Xi I(\Xi)^m:=\frac{d}{dt}\mathbb{W}^m_{s,t}$$
in the sense of distributions.  
The maps~$\Pi_{\cdot}$ are then extended to~$\Tt$ by linearity. 
Finally, for each $\Gamma\in G$  the (random) re-expansion maps are given by
\begin{equation}
    \label{eq:Gamma}
    \Gamma_{s,t}:=\Gamma_{\widehat{W}^H_s-\widehat{W}^H_t}
\end{equation}
(i.e. the right hand side results from $\Gamma_h$~\eqref{eq:structuregroup1},~\eqref{eq:structuregroup} evaluated at  $h=\widehat{W}^H_s-\widehat{W}^H_t)$.

In order to emphasise the dependence of the It\^o model on the realisation of the noise we will often write $(\Pi, \Gamma)=(\Pi^{\dot W}, \Gamma^{\dot W})$. For a fixed time horizon $T>0$, we denote the space of models $(\Pi, \Gamma)$ for the rough volatility regularity structure $\Tt$ by $\mathcal{M}_{T}(\Tt)$. The topology on this space will be discussed in the next section.
The object with the lowest homogeneity in~$\Tt$ is~$\Xi$, 
which corresponds to the white noise~$\dot{W}$.
For $\kappa\in(0,H)$, the latter can be considered as a Gaussian random element that takes values in the  Besov space $E:=\mathcal{C}^{-\half-\kappa}$ (see~\cite[Section 2.1]{chandra2017stochastic} for a proof in a more general setting), defined as the subspace of distributions $f\in (C_c^1(\R))^*$ such that for all $K\subset\R$ compact,
\begin{equation}\label{eq:Besovnorm}
    \|f\|_{-\half-\kappa, K}:=\sup_{s\in K}\sup_{\substack{\phi\in \Cc_c^1(\R), \|\phi\|_{\Cc^1}\leq 1\\ \supp(\phi)\subset(-1,1),\lambda\in(0,1]}}\lambda^{\half+\kappa}\big| \langle f, \phi^\lambda_s\rangle \big|<\infty,
\end{equation}
and $\phi_s^\lambda(t)=\lambda^{-1}\phi(\lambda^{-1}(t-s))$ as in Definition~\ref{modeldef}. 
The collection $\{\|f\|_{-\half-\kappa, K}; K\subset\R\;\text{compact}\}$ can be shown to be a family of seminorms that turn $\mathcal{C}^{-\half-\kappa}$ into a Fr\'echet space~\cite[Section~13.3.1]{friz2020course}, but for our purposes we shall restrict to $K=[0,T]$. 

\begin{rem}\label{rem:RVabstractWiener} The triple $(E,\h, \mu)$, where $\h:=L^2[0,T]$ is the Cameron-Martin space of $\dot{W}$ and $\mu:\mathscr{B}(E)\rightarrow[0,1]$ is the law of $\dot{W}$ is an abstract Wiener space.
\end{rem}

 We shall now introduce the  notions of "model lift" and of "model translation" by elements in~$\h$ (such operations have already been considered in~\cite{friz2021precise, friz2022precise}).

\begin{dfn}\label{def:modeltranslation} Let $(\Tt, G)$ be the rough volatility regularity structure and $(\Pi, \Gamma)$ the It\^o model for $(\Tt, G)$. The map 
	$$E\ni \xi\longmapsto\mathcal{L}(\xi)=(\Pi^{\xi}, \Gamma^\xi)\in\mathcal{M}_{T}(\Tt),$$ where the noise $\dot{W}$ is replaced by $\xi$ is called the model lift.  The model translation in the direction of $h\in\h$ is defined by \begin{equation}\label{modeltranslationdilation}
	\mathcal{M}_{T}(\Tt)\ni \Pi^\xi \longmapsto T_h(\Pi^\xi):=\Pi^{\xi+h}\in \mathcal{M}_{T}(\Tt).
	\end{equation}
 The translation $T_h\Gamma$ is defined analogously.
\end{dfn}
For fixed $h\in\h$, the (deterministic) model $(\Pi^h, \Gamma^h)$ is given by
$$
\Pi^h_t\mathbf{1}=1,\quad
\Pi^h_t\Xi=h,\quad
\Pi^h_tI(\Xi)^m=(\widehat{h}-\widehat{h}(t))^m,\quad
\Pi^h_t\Xi I(\Xi)^m=\frac{d}{ds}\mathbb{H}^m_{t,s},
\quad \text{for } m=1,\dots, M,
$$
where 
$$
\mathbb{H}^m_{t,s} :=
\int_{t}^{s}\left(\widehat{h}(r)-\widehat{h}(t) \right)^mh(r)dr.
$$
This last integral is well-defined since $\widehat{h}=\int_0^\cdot K^H(\cdot-r)h(r)dr\in \Cc^{H}[0,T]$ and~$H$ is square-integrable. Note that, for each $t\in[0,T]$ and $m=1,\dots M$, 
$\Pi^h_t I(\Xi)^m$ is function-valued. Nevertheless, the homogeneity of this term is $mH-\half$, which can be negative for small enough~$m$. 
\begin{rem} 
Similar to the case of random rough paths, the lift map $\mathcal{L}$ is well defined on smooth functions since the integrals in~$\Pi$ are then standard Riemann integrals. In general, however, due to the probabilistic step (i.e. the use of It\^o integration) in the construction of~$\Pi$,  
such a map is only well defined $\mu$-almost everywhere in~$E$. 
As a result, $\mathcal{L}$ is only a measurable map on the path space $(\Omega, \mathscr{F}, \pr)=(E, \mathscr{B}(E), \mu )$.
\end{rem}

\subsection{TCIs for It\^o models and the driftless log-price}\label{shiftsec} In this section we prove some continuity estimates for Cameron-Martin shifts of It\^o models and It\^o integrals of the form $\int_{0}^{\cdot}f( \widehat{W}^H_s, s)dW_s$.  Then, we present our result, Theorem~\ref{thm:TCIRV1}, on TCIs for It\^o models and the driftless log-price corresponding to~\eqref{roughvolmodel}.

\begin{dfn}(Model topology)\label{dfn:modeltop} Let $(\Tt, G)$ be a regularity structure and $T>0$. The distances $\|\cdot\|, \VERT\cdot\VERT:\mathcal{M}_{T}(\Tt)^2\rightarrow[0, \infty]$ between two models $(\Pi^1, \Gamma^1), (\Pi^2, \Gamma^2)\in\mathcal{M}_{T}(\Tt)$ are defined by 

\begin{equation}\label{twobar}	
\left\|\Pi^1, \Pi^2\right\| \equiv 
\left\|\Pi^1-\Pi^2\right\|
 := \sup_{\substack{\phi\in C_c^1(\R), \|\phi\|_{\Cc^1}\leq 1\\ \supp(\phi)\subset(-1,1),\lambda\in(0,1]\\ \tau\in\mathcal{S}, s\in[0,T]    }}
 \lambda^{-|\tau|}
 \left| \left\langle \left(\Pi^1_{s}-\Pi^2_{s}\right)\tau, \phi^\lambda_s\right\rangle \right|
	\end{equation}
	and 	
	\begin{equation}\label{threebar}	
	\left\VERT\left(\Pi^1, \Gamma^1\right), \left(\Pi^2, \Gamma^2\right) \right\VERT
  := \left\|\Pi^1-\Pi^2\right\| + \sup_{\substack{s,t\in[0,T], \tau\in\mathcal{S}\\A\ni\beta<|\tau|}}\frac{\big| \big(\Gamma^1_{t,s}- \Gamma^2_{t,s}\big)\tau\big|_\beta}{|t-s|^{|\tau|-\beta}},
	\end{equation}
	where $|\tau|_\beta$ denotes the absolute value of the coefficient of $\tau$ with $|\tau|=\beta$. 	
\end{dfn}

\begin{rem}\label{rem:modelequivalence} It turns out that the two notions of distance introduced in the previous definition are equivalent, provided that the models considered satisfy appropriate admissibility conditions. In the context of singular stochastic PDEs, this observation has been proved in~\cite[Remark~3.5]{cannizzaro2017malliavin}.
A similar observation holds for the case of the It\^o models considered for the rough volatility regularity structure~\cite[Lemma 3.19]{bayer2020regularity}.
\end{rem}

\begin{rem}\label{seprem} In order to avoid technical difficulties related to the well-posedness of the Kantorovich dual problem 
(see for instance the proof of Proposition~\ref{alphaprop}), 
it is of interest to define white noise as a random element taking values in a Polish space. As pointed out in~\cite[Remark 2.6]{hairer2015large}, this space is not separable. However, this can be fixed by defining $\mathcal{C}^{-\half-\kappa}$ as the completion of smooth functions with respect to the seminorms $\|\cdot\|_{-\half-\kappa, K}$. Another alternative would be to define $\dot{W}$ as a random element with values in the (much larger) space  $\mathcal{S}'$ of tempered distributions, which is separable. The same issue is present in the case of the metric space $(\mathcal{M}_{T}(\Tt), \|\cdot\|)$. Once again, the remedy is to define $\mathcal{M}_{T}(\Tt)$ as the completion of the set of smooth, admissible models (i.e.~$\Pi$ such that $[0,T]\ni s\mapsto \Pi_s\in\mathscr{L}(\Tt; (C_c^\infty(\R)^*)$ is smooth) under the metric $\|\cdot\|$.
\end{rem}

The estimates in the following two lemmata are used in the proof of Theorem~\ref{thm:TCIRV1} . Their proofs are deferred to Appendix~\ref{proof:shiftlem} and ~\ref{proof:Itoshiftlem} respectively. 
\begin{lem}(Model shifts)\label{shiftlem} 
Let $T>0$. The map $\h\ni h\mapsto T_h\Pi^\xi\in\mathcal{M}_T(\Tt)$ is locally Lipschitz continuous.
In particular, for $H\in(0,1), M\in\N$ as in~\eqref{Mchoice}, there exists $C_{H,M,T}>0$ and a random variable $K\in\cap_{p\geq 1}L^p(\Omega)$ such that we have the almost sure estimate 
	\begin{equation}
	\big\|T_{h_2}\Pi^{\xi}-T_{h_1}\Pi^{\xi}\big\|\leq C_{H,M,T}K_{h_1}\|h_2-h_1\|_{\h}\vee \|h_2-h_1\|_{\h}^{M+1},
	\quad\text{for all }h_1, h_2\in\h.
	\end{equation}	
\end{lem}
Before we proceed to the analysis of the driftless log-price we introduce the following asumption on the volatility function $f$.
\begin{assu}\label{assu:f} The volatility function  $f:\R\times\R^{+}\rightarrow\R$ in~\eqref{roughvolmodel} satisfies 
\begin{equation}\label{eq:assu f}
    f(x,t)-f(y,t)=f_1(x-y,t)f_2(y,t),
\end{equation}
for all $t>0, x, y\in \R$, where 
\begin{itemize}
\item[(i)] $f_1:\R\times\R^+\rightarrow \R$ is differentiable and 
$|\nabla f_1|\leq G$ for some nondecreasing $G:\R\to[0, \infty]$;
\item[(ii)] for all $T>0$, there exist $C_1, C_2>0$ such that $\sup_{t\in[0,T]}|f_2(t,x)|\leq C_1(1+e^{C_2|x|})$.
\end{itemize}
\end{assu}

\begin{lem}(It\^o integral shifts)\label{Itoshiftlem} Let $(\Omega, \h, \pr)$ be the abstract Wiener space of $\dot{W}$ 
(Remark~\ref{rem:RVabstractWiener}) and $f:\R\times\R^{+}\rightarrow\R$ satisfy Assumption~\ref{assu:f}.
 Then for all $\gamma<\half$ there exists $c>0$ and a random variable $K\in\cap_{p\geq 1}L^p(\Omega)$ such that for all $h\in\h$, we have $\pr$-a.s.,
 	\begin{equation}
	   \bigg\|\bigg(\int_{0}^{\cdot}f( \widehat{W}^H_r, r)dW_r\bigg)(\omega+h)- \bigg(\int_{0}^{\cdot}f( \widehat{W}^H_r, r)dW_r\bigg)(\omega)\bigg\|_{\Cc^\gamma[0,T]}\leq K\big|G(c\|h\|_{\h})\big|\bigg(\|h\|_{\h}+1\bigg).
	\end{equation}	
\end{lem}

\noindent The following is the main result of this section: 
\begin{thm}\label{thm:TCIRV1} 
Let $H\in(0, \half], \kappa\in (0, H), M=\Mf(H,\kappa)$ (as in~\eqref{Mchoice}) and $g^+=g\vee 0$.
Then
\begin{enumerate}
    \item Let $c_\h:\mathcal{M}_T(\Tt)\times\mathcal{M}_T(\Tt)\rightarrow\R$ denote the Cameron-Martin pseudo-metric:
\begin{equation}\label{CMpseudometric}
c_\h(\Pi^1,\Pi^2)=
\begin{dcases}
\|\Pi^{1}_0\Xi-\Pi^{2}_0\Xi\|_{\h},
& \text{if }\Pi_0^{1}\Xi-\Pi_0^{2}\Xi\in\h,\\
\infty, 
& \text{otherwise}.
\end{dcases}
\end{equation}
The law $\mu\in\mathscr{P}( \mathcal{M}_T(\Tt))$ of the It\^o model $(\Pi^\xi, \Gamma^\xi)$ satisfies $\Tt_2(2)$ with respect to~$c_\h$.
\item  
The law  $\mu\in\mathscr{P}( \mathcal{M}_T(\Tt))$ of the It\^o model $(\Pi^\xi, \Gamma^\xi)$ satisfies $\mathscr{T}_\alpha(c)$ with 
$\alpha(t)= t^2\wedge t^{2(M+1)}$
and 
$c(\Pi^1,\Pi^2):=\|\Pi^1-\Pi^2\|^{\frac{1}{M+1}}$.
    \item Let $\gamma\in (0, \half)$ and~$f$ satisfy Assumption~\ref{assu:f} with $G(x)=|x|\vee|x|^r$ for some $r\geq 1$.
    The law $\mu\in\mathscr{P}(\Cc^\gamma[0,T])$ of the driftless log-price $\int_{0}^{\cdot}f(\widehat{W}^H_r,r)dW_r$ satisfies $\mathscr{T}_\alpha(c)$  with $\alpha(t)= t^2\wedge t^{2(r+1)}$ and $c(x,y)=\|x-y\|_{\Cc^\gamma}^{\frac{1}{r+1}}$.
\item Let $f(t,x)=g(t)e^x$ for some $g\in \Cc^1[0,T]$.
Then, for $C>0$, the law $\mu\in\mathscr{P}(C[0,T])$ of $(\log|\int_{0}^{\cdot}f(\widehat{W}^H_r,r)dW_r|)^+$ satisfies $\mathscr{T}_1(C)$.
\end{enumerate}
\end{thm}
\begin{proof}
 Throughout the proof, $(E, \h, \pr)$ is the abstract Wiener space of $\xi$
(Remark~\ref{seprem}), where $\pr\in\mathscr{P}(E)$ is the law of~$\xi$.
\begin{enumerate}
    \item By definition of~$c_\h$ we have 
$$
\inf_{\pi\in\Pi(\nu,\mu)}\iint_{\mathcal{M}_T(\Tt)\times\mathcal{M}_T(\Tt)}c^2_\h(y_1, y_2)d\pi(y_1,y_2)
 = \inf_{\pi\in\Pi(\nu,\pr)}\iint_{E\times E}\|x_1-x_2\|^2_{\h}d\pi(x_1,x_2).
$$
Since $\pr\in\mathscr{T}_2(2)$ with respect to the Cameron-Martin distance $\|\cdot\|_{\h}$~\cite[Section 5.1]{djellout2004transportation}, the conclusion follows.
 \item Let $\mathcal{X}=E$, $\mathcal{Y}=\mathcal{M}_T(\Tt)$, $c_{\mathcal{X}}=c_{\h}, c_{\mathcal{Y}}=c$ and  $\Psi: \mathcal{X}\rightarrow \mathcal{Y}$ be the model lift $\xi\mapsto\Pi^\xi$. 
 Moreover, set $x_1=\xi$, $x_2=\xi+h$ for some $h\in\h$.
 In view of Lemma~\ref{shiftlem} with $h_1=0, h_2=h$, we have
\begin{equation}
\begin{aligned}
c_{\mathcal{Y}}\big(\Psi(x_1),\Psi(x_2)\big)
 = \big\|\Pi^{x_1}-\Pi^{x_2}\big\|^{\frac{1}{M+1}}
 & \leq C K(x_1)\left(\|x_2-x_1\|^{\frac{1}{M+1}}_{\h}\vee\|x_2-x_1\|_{\h}\right)\\
 &=C K(x_1)\bigg(c_{\mathcal{X}}(x_1,x_2)\vee c_{\mathcal{X}}(x_1,x_2)\big)^{\frac{1}{M+1}}\bigg),    
\end{aligned}
\end{equation}
where $K\in\bigcap_{p\geq 1}L^p(\mathcal{X}, \pr)$.
Lemma~\ref{Lem: contraction1} with $r=M+1$ concludes the proof 
(note that Assumption~\ref{assu:Setting} is satisfied since $\pr\in\mathscr{T}_2(2)$ with respect to~$c_\h)$.
\item Let $\mathcal{X}=\Cc^\gamma[0,T]=\mathcal{Y}$, $c_{\mathcal{X}}=c_{\h}$, 
$c_{\mathcal{Y}}=\|\cdot\|_{\Cc^\gamma}$ and  $\Psi: \mathcal{X}\rightarrow \mathcal{Y}$ 
be the map $\omega\mapsto\int f( \widehat{W}^H_r, r)dW_r$. 
Moreover, set $x_1=\omega$, $x_2=\omega+h$ for some $h\in\h$.
From Lemma~\ref{Itoshiftlem} with $h_1=0, h_2=h$ we obtain 
\begin{equation}
c_{\mathcal{Y}}\big(\Psi(x_1),\Psi(x_2)\big)
 \leq C K(x_1)\bigg(c_{\mathcal{X}}(x_1,x_2)\vee c_{\mathcal{X}}(x_1,x_2)\big)^{\frac{1}{r+1}}\bigg). 
\end{equation}
The conclusion follows once again by appealing to Lemma~\ref{Lem: contraction1}. 
\item Let $\mathcal{X}=\Cc^0[0,T]=\mathcal{Y}, c_{\mathcal{X}}=\|\cdot\|_{\Cc^0}$, $c_{\mathcal{Y}}, \Psi$ as in the proof above and fix $h\in\h$. Using a similar argument as in the proof of Lemma~\ref{Itoshiftlem} it follows that
\begin{equation}
    \|\Psi(\omega+h)\|_{\mathcal{X}}\leq C_1Ke^{C_2\|h\|_{\h}},
\end{equation}
where $C_1>0$ and $K\in \bigcap_{p\geq 1}L^p(\mathcal{X}, \pr)$ is a non-negative random variable. Letting $\Phi=(\log|\Psi|)^+$ the latter implies
\begin{equation}
    \|\Phi(\omega+h)\|_{\mathcal{X}}\leq \log(C_1K+1)+C_2\|h\|_{\h}.
\end{equation}
Appealing to the generalised Fernique theorem (\cite{friz2020course}, Theorem 11.7) we deduce that $\Phi$ has Gaussian tails and in particular that $\ex\exp(\lambda \|\Phi\|^2_{\mathcal{X}})<\infty$ for sufficiently small and positive values of~$\lambda$. 
The proof is complete by recalling that the latter is equivalent to Talagrand's $\mathscr{T}_1(C)$ (using~\cite[Theorem 1.13]{10.1214/ECP.v11-1198} with $\alpha(t)=t^2$).
\end{enumerate}
\end{proof}

\begin{rem} 
It is straightforward to replace the right-hand side of~\eqref{eq:assu f} by a finite sum of functions $f_{1,k}(x-y,t)f_{2,k}(x-y,t)$  that satisfy the same properties as~$f_1$ and~$f_2$. Lemma~\ref{Itoshiftlem} then holds with $G$ replaced by $\max\{G_k\}$.
Assumption~\ref{assu:f} covers a wide range of volatility functions that are of interest in rough volatility modeling. 
In particular it includes polynomial volatility models~\cite{jaber2022joint}, where $f(t,x)=g(t)P(x)$ 
for a $\Cc^1$-function $g$ and a polynomial~$P$ of degree~$r$, as well as exponential functions that correspond to lognormal volatility models (as in the rough Bergomi model~\cite{bayer2016pricing}).
\end{rem}

\begin{cor} 
Let $H\in(0, \half]$, $\gamma<\half$ and  $f\in C(\R\times\R^+)$ that satisfies Assumption~\ref{assu:f}.
Then there exist $C, s_0>0$ such that any iid sample $\{X_n; n\in\N\}$ of the driftless log-price $X_\cdot=\int_{0}^{\cdot}f(\widehat{W}^H_r,r)dW_r$ satisfies, for $n\in\N, s>s_0$, 
\begin{equation}
 \pr\left[ \left\|\frac{1}{n}\sum_{k=1}^{n}X_n\right\|_{\Cc^\gamma[0,T]} >s \right]
 \leq \exp\left\{-Cns^{\frac{2}{r+1}}\right\}.
\end{equation}
\end{cor}	
\begin{proof} By the triangle inequality,
$$
\pr\left[ \left\|\frac{1}{n}\sum_{k=1}^{n}X_n\right\|_{\Cc^\gamma[0,T]} >s \right]  \leq \pr\left[ \frac{1}{n}\sum_{k=1}^{n}\left\|X_n\right\|_{\Cc^\gamma[0,T]} >s \right].
$$ The conclusion then follows from Theorem~\ref{thm:TCIRV1}(3) and Corollary~\ref{cor:deviation} with $d$ being the $\gamma$-H\"older metric, $x_0=0$ and $p=r+1$.
\end{proof}

\subsection{TCI for the log-price as a modelled distribution}\label{subsec: Modelled}
Given a fixed regularity structure~$\Tt$, along with a model $(\Pi, \Gamma)\in\mathcal{M}(\Tt)$, 
it is possible to consider $\Tt$-valued functions (or distributions) that can be approximated, up to a given precision, by abstract "Taylor polynomials" consisting of symbols in $\Tt$.
The regularity of such functions can be expressed in terms of the approximation error, similar to the case of classical polynomials and smooth functions.
This concept is captured by the notion of modelled distributions.

\begin{dfn}
Let $\gamma>0$, $\Tt$ the rough volatility regularity structure and $(\Pi, \Gamma)\in\mathcal{M}_T(\Tt)$.
A map $f:[0,T]\rightarrow\Tt$ is said to be a modelled distibution of order $\gamma$ (written $f\in \mathcal{D}_T^\gamma$) if 
\begin{equation}\label{eq:Dgammanorm}
\|f\|_{\mathcal{D}_T^\gamma}:=\sup_{t\in[0,T]}\sup_{ A\ni\beta<\gamma}\big| f(t)\big|_{\beta}+
\sup_{\substack{s\neq t\in[0,T]\\A\ni\beta<\gamma}}\frac{\big| f(t)-\Gamma_{t,s}f(s)\big|_{\beta}}{|t-s|^{\gamma-\beta}}<\infty,
\end{equation}
where $A$ is the index set of~$\Tt$ and , as above, $|\tau|_{\beta}$ denotes the absolute value of the coefficient of $\tau\in\Tt$ with degree~$\beta$. 
To emphasise the dependence of the space~$\mathcal{D}_T^\gamma$ on the given model, we shall often use the notation $\mathcal{D}_T^\gamma(\Gamma)$.
Finally, for $i=1,2$,
$(\Pi^i, \Gamma^i)\in\mathcal{M}_T(\Tt)$, $f_i\in\mathcal{D}_T^\gamma(\Gamma^i)$, 
the distance between~$f_1$ and~$f_2$ is defined by
\begin{equation}\label{eq:modelleddistance}
\big\|f_1;f_2\big\|_{\mathcal{D}_T^\gamma(\Gamma^1), \mathcal{D}_T^\gamma(\Gamma^2) }:=\sup_{t\in[0,T]}\sup_{ A\ni\beta<\gamma}\big| f_1(t)-f_2(t)\big|_{\beta}+
\sup_{\substack{s\neq t\in[0,T]\\A\ni\beta<\gamma}}\frac{\big| f_1(t)-\Gamma^1_{t,s}f_1(s)-f_2(t)+\Gamma^2_{t,s}f_2(s)\big|_{\beta}}{|t-s|^{\gamma-\beta}}.
\end{equation}	
\end{dfn}

\begin{rem}\label{rem: Modeltop}
    For a fixed model $(\Pi, \Gamma)\in\mathcal{M}_{T}(\Tt)$, the space $\mathcal{D}_T^\gamma(\Gamma^i)$ is linear and is a Banach space with respect to the norm $\|\cdot\|_{\mathcal{D}_T^\gamma(\Gamma)}$ (in general, when for non-compact domain, one obtains a Fr\'echet space with respect to a family of seminorms parameterised by compacta).
\end{rem}

At this point, we have defined the rough volatility regularity structure $\Tt$ (Section~\ref{rvolrs}) and studied properties of the It\^o model and 
 driftless log-price (Sections~\ref{shiftsec}). 
One crucial feature of the theory is roughly stated as follows: Once lifted to a (random) modelled distribution, the price process can be expressed as a locally-Lipschitz continuous map of the underlying noise with probability~$1$.
The latter is done with the aid of Hairer's reconstruction operator $\mathscr{R}$ which maps modelled distributions to Schwartz distributions or functions. These operations can be summarised by the diagramme
\begin{center}
	\begin{tikzcd}[column sep=huge, row sep=large]
		\mathcal{M}_{T}(\Tt)\ni \Pi^\xi\arrow[r, "\mathscr{D}_f" ] 
		& f^{\Pi^\xi}\in\mathcal{D}_T^\gamma(\Gamma^\xi) \arrow[d, "\mathscr{R}"] \\
		E\ni\xi \arrow[r, dashrightarrow]\arrow[u, "\mathscr{L}"]
		&
		X\in C([0,T];\R)
	\end{tikzcd}
\end{center}
where $X$, the dashed arrow and $\mathscr{L}$ denote the driftless log-price process $\int fdW$, (as in~\eqref{roughvolmodel}),
the (It\^o) solution map and the model lift (Definition~\ref{def:modeltranslation}).
The rest of this section is devoted to the study of the map~$\mathscr{D}_f$,
which takes as input a (random) It\^o model and returns an expansion of the volatility function as a linear combination of elements of~$\mathscr{T}$.
This (random) modelled distribution has the property that
(at least in a local sense)~\cite[Lemma~3.23]{bayer2020regularity}
\begin{equation}\label{eq:reconstruction}
    \mathscr{R}\mathscr{D}_f(\Pi^\xi)=X,
\end{equation}
and both maps $\mathscr{R}$ and $\mathscr{D}_f$ are continuous with respect to the underlying model in the topology of $\mathcal{M}_{T}(\Tt)$. The main result of this section is Theorem~\ref{thm:TCImodelled}.
\begin{dfn}\label{Ddef} 
Let $T>0, H\in(0, \half], \kappa\in(0, H)$, $M=\Mf(H,\kappa)$ as in~\eqref{Mchoice} and $\Tt$ be the rough volatility regularity structure.
For $f\in \Cc^{M+1}(\R\times\R^+)$ and $\Pi\in\mathcal{M}_{T}(\Tt)$, let $f^\Pi$ be the modelled distribution (with $*$ denoting convolution)
$$
f^\Pi(t):=\sum_{k=0}^{M}\frac{1}{k!}\partial^k_1f\big((K^H*\Pi_t\Xi)(t),t\big)I(\Xi)^k,
\quad\text{for } t\in[0,T].
$$
The map $\mathscr{D}_f$ is then defined by ($\star$ denotes the formal product of symbols in $\Tt$)
$$\mathcal{M}_{T}(\Tt)\ni\Pi \longmapsto \mathscr{D}_f(\Pi):=f^\Pi\star\Xi\in\bigcup_{\gamma>0}\mathcal{D}_T^\gamma(\Gamma).
$$
\end{dfn}

\noindent The following lemma is used to obtain Theorem~\ref{thm:TCImodelled} below and its proof can be found in Appendix~\ref{Subsection: modelbndlem}. Regarding notation, we use $\partial_{1,2}^{m,n}f$ to denote the $m$-th (resp. $n$-th) order partial derivatives with respect to the first (resp. second) variable of a function $f\in  \Cc^{M+1, 1}(\R\times\R^+)$.

\begin{lem}\label{modelbndlem} 
Let 
$f\in \Cc^{M+1, 1}(\R\times\R^+)$, and $0<\gamma<((M+1)(H-\kappa)-\half-\kappa)\wedge(\half-\kappa)$.
\begin{enumerate}
  \item If $G:\R\rightarrow\R$ is a non-decreasing continuous function such that for some $C_{f,T}>0$,
 \begin{equation}\label{fgrowth}
\sup_{s\in[0,T]}\left|\partial_1^{m}f(x,s)\right| + \sup_{s\in[0,T]}\left|\partial_1^{M+1}f(x,s)\right|
 + \sup_{s\in[0,T]}\left|\partial^{m,1}_{1,2}f\big(x,s\big)\right|
 \leq C_{f,T}(1+G(|x|))
\end{equation}
holds for all $x\in\R$, 
 $m=0,\dots, M$,	
then there exists $C_{f,T,M}>0$ such that
\begin{equation*}
\|\mathscr{D}_f(\Pi)\big\|_{\mathcal{D}_T^\gamma}\leq 2C_{f,T, M}\bigg[  1+G\bigg(4\big\|  K^H*\Pi\Xi\big\|_{C[0,T]}  \bigg)\bigg]\bigg(1\vee\big\|K^H*\Pi\Xi\big\|^{M+1}_{\Cc^{H-\kappa}[0,T]}\bigg).
\end{equation*}
\item Let $N>0$, $P:\R^2\rightarrow\R$ a polynomial of degree~$N$ in two variables with $P(0,0)=0$ and assume that there exists $c_{f,T}>0$ such that for all $m=0,\dots, M+1$,  $x,y\in\R$,
\begin{equation}\label{unicontcondition}
\sup_{s\in[0,T]}\left|\partial_1^{m}f(x,s)-\partial_1^{m}f(y,s)\right|
 + \sup_{s\in[0,T]}\left|\partial^{m,1}_{1,2}f(x,s)-\partial^{m,1}_{1,2}f(y,s)\right|\leq c_{f,T}P(|y|, |x-y|).
	\end{equation}
Then for $(\Pi^1, \Gamma^1), (\Pi^2, \Gamma^2)\in\mathcal{M}_T(\Tt)$, 
 there exists $C_{M,f,T}>0$ such that
\begin{align*}
\|\mathscr{D}_f(\Pi^1)&;\mathscr{D}_f(\Pi^2)\big\|_{\mathcal{D}_T^\gamma(\Gamma^1), \mathcal{D}_T^\gamma(\Gamma^2) } \leq  C_{f,M, T} 
 \bigg(\{1\vee \big\|K^H*\Pi^1\Xi\big\|_{\Cc^{H-\kappa}[0,T]}^{N+M+1}  \bigg)\\
 &
\times\bigg(\big\|K^H*\Pi^2\Xi-K^H*\Pi^1\Xi\big\|_{\Cc^{H-\kappa}[0,T]}\vee \big\|K^H*\Pi^2\Xi-K^H\ast\Pi^1\Xi\big\|^{N+M+1}_{\Cc^{H-\kappa}[0,T]}\bigg).
	\end{align*}
\end{enumerate}
\end{lem}

The topology of  $\mathcal{D}_T^\gamma(\Gamma)$ (and the space itself) depends crucially on the choice of the (random) model. This implies in particular that the concept of a TCI for a $\mathcal{D}_T^\gamma(\Gamma)$-valued random element is not meaningful since both the latter and the space are random. Nevertheless, it is possible to obtain TCIs for random elements that take values on the \textit{total space}
\begin{equation}\label{eqref: MDspace}
    \mathcal{M}\ltimes\mathcal{D}^\gamma:=\coprod_{(\Pi, \Gamma)\in \mathcal{M}_T(\mathcal{T}) } \mathcal{D}_T^\gamma(\Gamma)=\bigcup_{(\Pi, \Gamma)\in \mathcal{M}_T(\mathcal{T})}\big\{(\Pi, \Gamma)\big\}\times \mathcal{D}_T^\gamma(\Gamma)
\end{equation}
A metric on this space can be defined with the help of the distance~\eqref{eq:modelleddistance}. In particular, we define $d^\flat_{\gamma}: (\mathcal{M}\ltimes\mathcal{D}^\gamma)\times (\mathcal{M}\ltimes\mathcal{D}^\gamma)\rightarrow [0,\infty]$ by
\begin{equation}\label{eq:flatmetric}
d^\flat_{\gamma}( f_1, f_2) := \left\VERT p(f_1)-p(f_2)\right\VERT+\|f_1;f_2\|_{\mathcal{D}_T^\gamma(p(f_1)),\mathcal{D}_T^\gamma(p(f_2))},
\end{equation}
where $p:\mathcal{M}\ltimes\mathcal{D}^\gamma\rightarrow \mathcal{M}_T(\mathcal{T})$ denotes the base space-projection of a modelled distribution to its corresponding model and $\left\VERT\cdot\right\VERT$ is the model metric~\eqref{threebar}.
\begin{lem} Let $\gamma>0$. The map $d^\flat_{\gamma}$ defines a metric on the total space $\mathcal{M}\ltimes\mathcal{D}^\gamma$.   
\end{lem}
\begin{proof} 
The fact that $d^\flat_{\gamma}$ is symmetric is immediate from its definition. Moreover, $d^\flat_{\gamma}=0$ if and only if $\left\VERT p(f_1)-p(f_2)\right\VERT=0$ and $\|f_1;f_2\|_{\mathcal{D}_T^\gamma(p(f_1)),\mathcal{D}_T^\gamma(p(f_2))}=0$.
Since $\left\VERT\cdot\right\VERT$ is a metric, the first equality is true if and only if $p(f_1)=p(f_2)$. The latter implies that the second equality holds if and only if $f_1=f_2$ 
(in this case $\|\cdot;\cdot\|$ is in fact a norm from Remark~\ref{rem: Modeltop}). 
It remains to show that $d^\flat_{\gamma}$ satisfies the triangle inequality. 
Since  $\left\VERT\cdot\right\VERT$ is a metric, we shall only show it for the second term in~\ref{eq:flatmetric}. Indeed for $f_1,f_2,f_3\in \mathcal{M}\ltimes\mathcal{D}^\gamma, \beta<\gamma, s\neq t\in[0,T]$,
\begin{equation}
    \big| f_1(t)-f_2(t)\big|_{\beta}\leq \big| f_1(t)-f_3(t)\big|_{\beta}+\big| f_2(t)-f_3(t)\big|_{\beta}
\end{equation}
and 
\begin{equation}
   \big| f_1(t)-\Gamma^1_{t,s}f_1(s)-f_2(t)+\Gamma^2_{t,s}f_2(s)\big|_{\beta}\leq \sum_{j=1}^{2}\big| f_j(t)-\Gamma^j_{t,s}f_j(s)-f_3(t)+\Gamma^3_{t,s}f_3(s)\big|_{\beta},
\end{equation}
and the conclusion follows after taking supremum in $s,t,\beta$.
\end{proof}

\begin{rem} Our notation for $d^\flat_\gamma$ is taken from~\cite[Definition~4.13]{varzaneh2022geometry}. There, an analogous metric has been introduced for the total space of controlled rough paths and called the "flat" metric. In that setting, the base space is given by a space of (geometric) rough paths and the fibres are spaces of controlled rough paths.
\end{rem}

\begin{thm}\label{thm:TCImodelled}
Let $\Tt$ be the rough volatility regularity structure, $(\Pi, \Gamma)$ be the It\^o model~\eqref{eq:Itomodel1}-\eqref{eq:Gamma}, $T>0$, $\kappa, H, M, f, \mathscr{D}_f$ as in Definition~\ref{Ddef},    $\gamma\in (0, (M+1)(H-\kappa)-\half-\kappa)\wedge(\half-\kappa))$ and $d^\flat_\gamma$ as in~\eqref{eq:flatmetric}. If~$f$ satisfies the assumptions of Lemma~\ref{modelbndlem}(2) then the law $\mu\in\mathscr{P}(\mathcal{M}\ltimes\mathcal{D}^\gamma)$ of the random element $(\Pi, \Gamma, \mathscr{D}_f(\Pi))$ satisfies $\mathscr{T}_\alpha(c)$ with $c(x,y)=d^\flat_\gamma(x,y)^{\frac{1}{N+M+1}}$ and, for some constant $C>0$, $\alpha(t)=C(t^{2(N+M+1)}\wedge t^2)$.
\end{thm}

\begin{proof} Let $h\in\h$, $(\Pi^1, \Gamma^1)=(\Pi, \Gamma)$ 
(Definition~\ref{def:modeltranslation}), $(\Pi^2, \Gamma^2)=(T_h\Pi, T_h\Gamma)\in\mathcal{M}_T(\mathcal{T})$. From Lemma~\ref{modelbndlem}(2) along with the continuous embedding $\Cc^{H-\kappa}\hookrightarrow\h$ we have 
\begin{align*}
\|\mathscr{D}_f(\Pi^1)&;\mathscr{D}_f(\Pi^2)\big\|_{\mathcal{D}_T^\gamma(\Gamma^1), \mathcal{D}_T^\gamma(\Gamma^2) } \leq  C_{f,M,T} 
 \bigg(\{1\vee \big\|\widehat{W}^H\big\|_{\Cc^{H-\kappa}[0,T]}^{N+M+1}\bigg)\bigg(\|h\|_{\h}\vee \|h\|^{N+M+1}_{\h}\bigg).
	\end{align*}
 Combining this estimate with the one obtained from Lemma~\ref{shiftlem} (with $h_1=0, h_2=h)$ and the equivalence of model norms (Remark~\ref{rem: Modeltop}) it follows that 
 \begin{equation}
d^\flat_{\gamma}( \mathscr{D}_f(\Pi^1),\mathscr{D}_f(\Pi^2))\lesssim \| \Pi^2-\Pi^1\|+\|\mathscr{D}_f(\Pi^2);\mathscr{D}_f(\Pi^1)\|_{\mathcal{D}_T^\gamma(\Gamma^1), \mathcal{D}_T^\gamma(\Gamma^2) }\leq K \|h\|_{\h}\vee \|h\|^{N+M+1}_{\h},
\end{equation}
with~$K$ a random variable with finite moments of all orders. 
Writing $\Pi^1=\mathscr{L}(\xi)$, $\Pi^2=\mathscr{L}(\xi+h)$,
letting $x_1=\xi$, $x_2=\xi+h$, $\mathcal{X}=E$, $\mathcal{Y}=\mathcal{M}\ltimes\mathcal{D}^\gamma$, $
\Psi=\mathscr{D}_f(\mathscr{L}), c_{\mathcal{Y}}= d^\flat_{\gamma}$ and $ c_{\mathcal{Y}}=c_{\h}$,
the last bound translates to  
 \begin{equation}
c\big(\Psi(x_1),  \Psi(x_2)\big)\leq L(x_1) \bigg[c_{\mathcal{X}}(x_1, x_2)\vee c_{\mathcal{X}}(x_1, x_2)^{\frac{1}{M+N+1}}\bigg].
\end{equation}
The conclusion then follows by virtue of Lemma~\ref{Lem: contraction1}.
\end{proof}

\begin{rem} To the best of our knowledge, the question of whether the metric space $(\mathcal{M}\ltimes\mathcal{D}^\gamma, d^\flat_{\gamma})$, with respect to a given regularity structure, is Polish remains open (see however~\cite[Theorem 4.18]{varzaneh2022geometry} for a positive answer in the setting of rough paths controlled by geometric rough paths). 
Nevertheless, we prove Theorem~\ref{thm:TCImodelled} using a generalised contraction principle, Lemma~\ref{Lem: contraction1}, which does not require~$\mathcal{Y}$ to be Polish.
\end{rem}

\noindent We conclude this section with a remark on the reconstruction operator $\mathscr{R}$.

\begin{rem}
In view of~\eqref{eq:reconstruction}, an alternative route to prove Lemma~\ref{Itoshiftlem} is to use the local Lipschitz continuity estimates of the reconstruction operator~\cite[Theorem 3.10, Equation~(3.4)]{hairer2014theory}. 
In summary, for two models $(\Pi^i, \Gamma^i)\in\mathcal{M}_T(\Tt), i=1,2$ and corresponding reconstruction operators $\mathscr{R}^i: \mathcal{D}_T^\gamma(\Gamma_i) \rightarrow C([0,T];\R)$, the following estimate holds for all $\phi\in C_c^\infty(\R), \lambda\in(0, 1], s\in[0,T]$:
\begin{align*}
& \bigg|\left\langle\mathscr{R}^2\mathscr{D}_f(\Pi^2)-\Pi^2_s\mathscr{D}_f(\Pi^2)(s)-\mathscr{R}^1\mathscr{D}_f(\Pi^1)+\Pi^1_s\mathscr{D}_f(\Pi^1)(s), \phi_s^\lambda \right\rangle \bigg| \\
&\leq C\left(\left\|\mathscr{D}_f(\Pi^1)
\right\|_{\mathcal{D}_T^\gamma(\Gamma^1)}
\left\|(\Pi^1, \Gamma^1); (\Pi^2, \Gamma^2)\right\|
 + \left\|\Pi^2\right\|\left\|\mathscr{D}_f(\Pi^1);\mathscr{D}_f(\Pi^2)\right\|_{\mathcal{D}_T^\gamma(\Gamma^1), \mathcal{D}_T^\gamma(\Gamma^2) }\right).
\end{align*}
\end{rem}
This estimate, however, takes into account both the growth of $f$ and the order of Wiener chaos of the random models. The latter may lead to sub-optimal $\h$-continuity estimates for~$X$ as follows: 
Letting $\Pi^2=T_h\Pi, \Pi^1=\Pi$ and invoking Lemma~\ref{shiftlem},~\ref{modelbndlem}(2), one sees that $\|\Pi^2\|\lesssim 1+\|h\|_{\h}\vee \|h\|^{M+1}_{\h}$ and $\left\|\mathscr{D}_f(\Pi^1);\mathscr{D}_f(\Pi^2)\right\|_{\mathcal{D}_T^\gamma(\Gamma^1), \mathcal{D}_T^\gamma(\Gamma^2) }\lesssim \|h\|_{\h}\vee\|h\|^{N+1}_{\h}$, where $N$ is the polynomial growth exponent of~$f$.
This leads to an estimate of the form 
\begin{equation}
	   \bigg\|\bigg(\int_{0}^{\cdot}f( \widehat{W}^H_r, r)dW_r\bigg)(\omega+h)- \bigg(\int_{0}^{\cdot}f( \widehat{W}^H_r, r)dW_r\bigg)(\omega)\bigg\|_{\Cc^\gamma[0,T]}\lesssim \|h\|_{\h}\vee\|h\|^{N+M+1}_{\h},
	\end{equation}	
which implies that the tail probabilities $\pr[\|fdW\|_{\Cc^\gamma}>r]$ decay at most like 
$\exp\{-r^{\frac{2}{N+M+2}}\}$.
It is then strsaightforward to check that, if $f$ is a polynomial of degree $N$, one can obtain a better decay rate of order $\exp\{-r^{\frac{2}{N+1}}\}$ via the arguments of Theorem~\ref{thm:TCIRV1}(3).

\section{TCIs for regularity structures: 2D PAM }\label{Section:PAM}

In the previous section we introduced some elements of regularity structures for rough volatility and proved TCIs for the It\^o model, the driftless log-price and for the joint law of the It\^o model and the modelled distribution lift. 
This section is devoted to the proof of TCIs in the setting of another regularity structure,
which arises in the study of singular SPDEs. 
In the sequel, given $\gamma\in\R$, $\Cc^\gamma$ will denote the space of distributions on $\mathbb{R}^2$, endowed with a seminorm topology similar to that defined in ~\eqref{eq:Besovnorm} (in fact it coincides with a local version of the Besov space $B^\alpha_{\infty, \infty}$).

To be precise, we are concerned with the  two-dimensional Parabolic Anderson Model (2d-PAM) with periodic boundary conditions, solution to the parabolic SPDE
    	\begin{equation}\label{eq:PAM}
    	\partial_tu = \Delta u+u\xi, 
    	\qquad
    	u(0,\cdot) = u_0(\cdot),
    	\end{equation}
    	where $\xi$ is spatial white noise on the $2$-torus $\mathbb{T}^2$ and $u_0\in \Cc^\gamma(\mathbb{T}^2)$ for some $\gamma\geq 0$. In general,  $d$-dimensional white noise is $-(d/2)^{-}$ regular in the sense that, almost surely, $\xi\in  \Cc^{\beta}$ for any $\beta<-d/2$. On the other hand, since the heat kernel is $2$-regularising, the solution $u$ is expected to live in~$\Cc^{\gamma}$, for any $\gamma<-d/2+2$. Thus, for $d=2$, $\gamma+\beta<0$ and as a result the product~$u\xi$ (and a fortiori the concept of solutions) is not classically defined (products of distributions on this scale are well defined when $\gamma+\beta>0$, as in~\cite[Theorem~2.85]{bahouri2011fourier}.
     The theory of regularity structures provides a notion of solutions that are defined as limits of properly re-normalised equations in which the noise $\xi$ is substituted by a smooth approximation $\xi^\epsilon:=\xi*\rho_\epsilon$ and $\rho_\epsilon$ is a mollifying sequence. In particular, letting \begin{equation}\label{eq:PAMconstant}
    	    C_\epsilon:=-\frac{1}{\pi}\log(\epsilon)
    	\end{equation}
    	be a divergent renormalisation constant, the solutions of 
    		\begin{equation}\label{eq:PAMrenorm}
    	\partial_t\tilde{u}_\epsilon=\Delta \tilde{u}_\epsilon+\tilde{u}_\epsilon\big(\xi_\epsilon-C_\epsilon\big), 
    	\qquad
    	\tilde{u}_\epsilon(0,\cdot)=u_0(\cdot),
    	\end{equation}
    	converge uniformly on compacts of $\R^+\times\mathbb{T}^2$ as $\epsilon\to0$ to a well-defined limit $u\in C^0(\R^+\times\mathbb{T}^2)$, 
    	then defined to be the solution of~\eqref{eq:PAM}. While a detailed overview of the solution theory of~\eqref{eq:PAM} is beyond the scope of this work, we shall provide a few facts that are relevant for the proof of our result, Theorem~\ref{thm:PAM}, along with pointers to the literature where necessary.
     
The 2d-PAM regularity structure space $\mathscr{T}$
(detailed in~\cite[Remark~8.8 and Section~9.1]{hairer2014theory}
and in~\cite[Section~3.1]{cannizzaro2017malliavin}) is generated by the symbols
        \begin{equation}\label{eq:PAMbasis}
            \mathcal{S} := \mathcal{W}\cup\mathcal{U}
             = \Big\{\Xi, I(\Xi)\Xi,  X_i\Xi:\; i=1,2\Big\} \cup \Big\{ \mathbf{1}, I(\Xi), X_i:\; i=1,2\Big\},
        \end{equation}
where $X_i$ stands for first degree monomials in each of the spatial variables, $\Xi$ for white noise and~$I$ for convolution with the heat kernel. 
The symbols' degrees are postulated as follows: 
\begin{center}
\begin{tabular}{ccccccc}
\hline
Symbol & $\mathbf{1}$ & $X_i$ & $\Xi$ & $X_i\Xi$ & $I(\Xi)$ & $I(\Xi)\Xi$ \\ 
Degree & 0 & $\displaystyle 1$ &$\displaystyle -1-\kappa$ & $\displaystyle -\kappa$ & $\displaystyle 1-\kappa$ & $\displaystyle -2\kappa$ \\
\hline
\end{tabular}
\end{center}
for $\kappa>0$ small enough. The structure group $G$ is then defined with the help of an additional set of symbols, that encode derivatives of the heat kernel, and its action on the basis vectors of $\mathscr{T}$ has an explicit matrix representation~\cite[Equation~(3.3)]{cannizzaro2017malliavin}.
The mollified noise $\xi_\epsilon$ has a canonical model lift $(\Pi_\epsilon, \Gamma_\epsilon)\in\mathcal{M}(\mathcal{T})$ and the topology on the space of models $\mathcal{M}(\mathcal{T})$ is given by the metric
\begin{equation}\label{twobarPAM}	
\left\|\Pi^1, \Pi^2\right\| \equiv 
\left\|\Pi^1-\Pi^2\right\|
 := \sup_{\substack{\phi\in C_c^2(\R^+\times \mathbb{T}^2), \|\phi\|_{\Cc^2}\leq 1\\ \supp(\phi)\subset B_{\mathfrak{s}}(0,1),\lambda\in(0,1]\\ \tau\in\mathcal{S}, z\in\R^+\times \mathbb{T}^2   }}
 \lambda^{-|\tau|}
 \left| \left\langle \left(\Pi^1_{z}-\Pi^2_{z}\right)\tau, \phi^\lambda_z\right\rangle \right|,
	\end{equation}
where for $z=(t,x),w=(s,y)\in\R^+\times\mathbb{T}^2$,
$\phi_z^\lambda(w)=\lambda^{-2}\phi(\lambda^{-2}(t-s), \lambda^{-1}(x-y))$ 
and 
$B_{\mathfrak{s}}(0,1)$ is the centered unit ball on $\R^+\times \mathbb{T}^2$ endowed with the parabolic distance $|(t,x)-(s,y)|_{\mathfrak{s}}=\sqrt{|t-s|}+|x-y|$ 
(this scaling is due to the fact that time and space play different roles in parabolic PDEs such as~\eqref{eq:PAM}). 

The need for renormalisation of~\eqref{eq:PAM} is linked to convergence properties of the sequence $\{(\Pi_\epsilon, \Gamma_\epsilon): \epsilon>0\}$ of canonical models. In fact, while the latter diverges as $\epsilon\to 0$, it is possible to construct a one-parameter group of transformations $\{M_\epsilon:\epsilon>0\}$, on $\mathcal{M}(\mathcal{T})$, known as the renormalisation group, such that  $(\hat{\Pi}_\epsilon, \hat{\Gamma}_\epsilon):=M_\epsilon[(\Pi_\epsilon, \Gamma_\epsilon)]$ converges in probability to a Gaussian model $(\hat{\Pi}, \hat{\Gamma})\in\mathcal{M}(\mathcal{T})$ (see~\cite[Sections~8.3 and~10.4]{hairer2014theory} for the relevant renormalisation theory).
The solutions of $\eqref{eq:PAMrenorm},~\eqref{eq:PAM}$ can then be expressed respectively as \begin{equation}\label{eq:PAMrecon}  \tilde{u}_\epsilon=\hat{\mathscr{R}}_\epsilon\mathscr{S}(u_0, \hat{\Pi}_\epsilon), u=\mathscr{R}\mathscr{S}(u_0, \hat{\Pi}),
\end{equation} where $\mathscr{S}$ (an abstract solution map between spaces of modelled distributions) and $\hat{\mathscr{R}}_\epsilon, \mathscr{R}$  (the reconstruction operators associated to $\hat{\Pi}_\epsilon, \hat{\Pi}$ respectively) are continuous with respect to the underlying models (see e.g.~\cite{cannizzaro2017malliavin}, Section 3.7 and references therein for a more detailed exposition of the solution theory. Moreover note that for~\eqref{eq:PAM} one has global existence, i.e. its explosion time is infinite with probability $1$).

An inspection of the basis elements $\mathcal{S}$~\eqref{eq:PAMbasis} shows that canonical models for PAM do not enjoy Gaussian concentration. Indeed, the symbol $I(\Xi)\Xi$ is on the second Wiener chaos with respect to~$\xi$. Moreover, since~\eqref{eq:PAM} exhibits (unbounded) multiplicative noise,  a similar observation is true for the solution~$u$.
In the following theorem we prove appropriate TCIs that reflect this heavier tail behaviour for both of these functionals.

     \begin{thm}\label{thm:PAM} Let $\mathcal{T}$ be the 2d-PAM regularity structure and $\mathcal{M}(\mathcal{T})$ be the space of canonical models on $\mathcal{T}$ and $f^+=f\vee 0$. The following hold:
     \begin{enumerate}
         \item The law $\mu\in\mathcal{M}(\mathcal{T})$ of the limiting model $(\hat{\Pi}, \hat{\Gamma})$ satisfies $\mathscr{T}_\alpha(c)$ with $c(\Pi^1, \Pi^2)=\|\Pi^1-\Pi^2\|^{1/2}$ and, for a constant $C>0$, $\alpha(t)=C (t\wedge t^2)$.
         \item Let $u, v\in \Cc(\R^+\times\mathbb{T}^2)$ solve~\eqref{eq:PAM} and  
         \begin{equation}
    \partial_tv=\Delta v+2v\xi,\;\;v(0,\cdot)=1
    \end{equation}
    respectively. For all $(t,x)\in\R^+\times\mathbb{T}^2$ such that $|v(t,x)|^{\frac{1}{2}}\in L^1(\Omega)$, the law $\mu\in \mathscr{P}(\R)$
    of $(\log|u(t,x)|)^+$ satisfies Talagrand's $\mathscr{T}_1(C)$.
     \end{enumerate}         
     \end{thm}
     \begin{proof} Throughout the proof we set $\h=L^2(\mathbb{T}^2)$, the Cameron-Martin space of~$\xi$.
\begin{enumerate}         
\item[(1)] Let $\phi\in C_c^2(\R^+\times \mathbb{T}^2)$ be a test function with support on the unit ball $B_{\mathfrak{s}}(0,1)$. In view of~\cite[Theorem~10.7, Equation~(10.7), Proposition~10.11]{hairer2014theory}, we have the following Wiener chaos decomposition: 
for all $\tau\in\mathcal{S}, z=(t,x)\in \R^+\times \mathbb{T}^2, \lambda\in(0,1)$,
$$
\left\langle\hat{\Pi}_z\tau, \phi_{z}^{\lambda} \right\rangle
 = \sum_{k\leq \#(\tau)}I_k\bigg(  \int_{\R^+\times \mathbb{T}^2}\phi_z^\lambda(y)S_z^{\otimes k}(\hat{\mathfrak{W}}^k\tau)(y)dy\bigg),
$$
where for each $k,y$, $(\hat{\mathfrak{W}}^k\tau)(y)\in\h^{\otimes k}$ is a function (or distribution) on~$k$ copies of~$\mathbb{T}^2$ denoted by $(\hat{\mathfrak{W}}^k\tau)(y; x_1,\dots x_k)$, $S_z\phi(y)=\phi(y-x)$,
$\phi\in\h$    and $ \#(\tau)$ is the number of occurrences of the symbol~$\Xi$ in~$\tau$. The map $I_k: \h^{\otimes k}\rightarrow L^2(\Omega)$ can be thought as a multiple Wiener integral with respect to the white noise~$\xi$.  In particular, for $f\in\mathcal{\h}^{\otimes{k}}$ one can write~\cite[Section~1.1.2]{nualart2006malliavin} 
$$
I_k(f)=\langle f, \xi^{\otimes k}\rangle:=\int_{\mathcal(\mathbb{T}^2)^k} f(x_1,\dots, x_k)d\xi(x_1)\dots d\xi(x_k).
$$
For the sequel, we shall focus on the symbol $\tau=I(\Xi)\Xi$ since this is the one which determines the tail behaviour of $\hat{\Pi}$. Similar but simpler arguments then hold for the symbols that correspond to a Wiener chaos of lower order and thus will be omitted. In this case we have $\#(\tau)=2$ and, in light of~\cite[Theorem 10.19]{hairer2014theory}, 
     \begin{equation}
         \begin{aligned}
             \langle\hat{\Pi}_z\tau, \phi_{z}^{\lambda} \rangle=I_0\bigg(  \int_{\R^+\times \mathbb{T}^2}\phi_z^\lambda(y)(\hat{\mathfrak{W}}^0\tau)(y)dy\bigg)+I_2\bigg(  \int_{\R^+\times \mathbb{T}^2}\phi_z^\lambda(y)S_{z}^{\otimes 2}(\hat{\mathfrak{W}}^2\tau)(y)dy\bigg)
         \end{aligned},
     \end{equation}
     where the first term is a non-random constant, say~$c$, and $\hat{\mathfrak{W}}^0$ can be explicitly specified in terms of the heat kernel~\cite[Theorem~10.19]{hairer2014theory}. 
     Letting $h\in\h$ and denoting the argument of~$I_2$ by $f_{\lambda, \tau, z}$, then
         \begin{align*}
             \langle T_h\hat{\Pi}_z\tau, \phi_{z}^{\lambda} \rangle-c
             & =\langle f_{\lambda, \tau, z}, (\xi+h)^{\otimes 2}\rangle\\
             & = \langle f_{\lambda, \tau, z}, \xi^{\otimes 2}\rangle+ \langle f_{\lambda, \tau, z}, h^{\otimes 2}\rangle+\langle f_{\lambda, \tau, z}, h\otimes\xi\rangle+\langle f_{\lambda, \tau, z}, \xi\otimes h\rangle.
\end{align*}
    Thus 
    \begin{equation}
         \begin{aligned}
             \langle (T_h\hat{\Pi}_z-\hat{\Pi}_z)I(\Xi)\Xi, \phi_{z}^{\lambda} \rangle= \langle f_{\lambda, \tau, z}, h^{\otimes 2}\rangle+ 
   \langle f_{\lambda, \tau, z}, h\otimes\xi\rangle+   
 \langle f_{\lambda, \tau, z}, \xi\otimes h\rangle.
              \end{aligned}
     \end{equation}
     The first term on the right-hand side is deterministic and can be bounded by $\|f_{\lambda, \tau, z}\|_{\h}\|h\|^2_{\h}$ by Cauchy-Schwarz. 
For the other two terms, stochastic Fubini and Cauchy-Schwarz imply
\begin{equation}
    \begin{aligned}
       \int f_{\lambda, \tau, z}(x_1, x_2)h(x_1)dx_1d\xi(x_2)&=\int I_1\bigg(f_{\lambda, \tau, z}(x_1, \cdot)\bigg)h(x_1) dx_1\\&\leq \|h\|_\h\|I_1\big(f_{\lambda, \tau, z}(\bullet, \cdot)\|_{\h} 
    \end{aligned}
\end{equation}
and the symmetric bound
$$ \int f_{\lambda, \tau, z}(x_1, x_2)d\xi(x_1)h(x_2)dx_2\leq \|h\|_\h\|I_1\big(f_{\lambda, \tau, z}(\cdot, \bullet)\|_{\h},
$$
where $\bullet$ indicates the variable that is integrated with respect to Lebesgue measure.
Putting these estimates together it follows that 
\begin{equation}
    \begin{aligned}
        |\langle (T_h\hat{\Pi}_z&-\hat{\Pi}_z)I(\Xi)\Xi, \phi_{z}^{\lambda} \rangle|\\&\leq 
 \bigg( \|I_1\big(f_{\lambda, \tau, z}(\cdot, \bullet)\|_{\h}+\|I_1\big(f_{\lambda, \tau, z}(\bullet, \cdot)\|_{\h} + \|f_{\lambda, \tau, z}\|_{\h}
             \bigg)\|h\|_{\h}\vee\|h\|^2_{\h}.
    \end{aligned}
\end{equation}

By It\^o-Wiener isometry and the estimates of~\cite[Proposition~10.11]{hairer2014theory}, 
the $L^2(\Omega)$-norms of the random prefactors on the right-hand side are bounded (up to a constant) by $\|f_{\lambda, \tau, z}\|_{\h}$ which in turn is bounded by $\lambda^{\kappa+2|\tau|}$.
Finally, from \cite[Theorem~10.7, (10.4), Proposition~3.32]{hairer2014theory}, 
we conclude that, for a random constant $L$ with finite moments of all orders, 
$$
 \|(T_h\hat{\Pi}_z-\hat{\Pi}_z)\|\leq L\|h\|_{\h}\vee\|h\|^2_{\h}.
 $$
 In view of Lemma~\ref{Lem: contraction1} the proof is complete.
\item[(2)] A solution theory for a shifted version of~\eqref{eq:PAM} in the direction of a Cameron-Martin space element $h\in\h$ was developed in~\cite[Theorem 3.26]{cannizzaro2017malliavin}.
Thus, solutions of 
    		\begin{equation}\label{eq:PAMh}
    	\partial_tu^h=\Delta u^h+u(\xi+h), \;\; u(0,\cdot)=u_0(\cdot)
    	\end{equation}
    	can be understood as limits, as $\epsilon\to 0$, of the renormalised equations 
    	\begin{equation}\label{eq:PAMhren}
    	\partial_t\tilde{u}^{h_\epsilon}_\epsilon=\Delta \tilde{u}^{h_\epsilon}_\epsilon+\tilde{u}^{h_\epsilon}_\epsilon\big(\xi_\epsilon+h_\epsilon-C_\epsilon\big), \;\; \tilde{u}^{h_\epsilon}_\epsilon(0,\cdot)=u_0(\cdot),
    	\end{equation}
    	where $h_\epsilon=h*\rho_\epsilon$ and $ C_\epsilon$ as in~\eqref{eq:PAMconstant}. Moreover, one can write $\tilde{u}^{h_\epsilon}:=\mathscr{R}^{M_\epsilon}\mathscr{S}(u_0, M_\epsilon T_h\Pi_\epsilon)$ similar to~\eqref{eq:PAMrecon}.
    	
    	Proceding to the main body of the proof, let $T>0, \epsilon<1$ and $(t,x)\in (0, T]\times\mathbb{T}^2$. 
    	Due to the linearity of~\eqref{eq:PAMhren}, we obtain via the Feynman-Kac formula
    	\begin{equation}
    	\begin{aligned}
    	\tilde{u}^{h_\epsilon}_\epsilon(\tau,y)=\ex^y_{B}\bigg[\int_{\tau}^Tu_0(B_r)\exp\bigg(-\int_{\tau}^{r}\big\{ h_\epsilon(B_s)+\xi_\epsilon(B_s)-C_\epsilon\big\}ds     \bigg) dr,            \bigg],
    	\end{aligned}
    	\end{equation}
    	where $\tau=T-t, y\in\mathbb{T}^2$ and~$B$ is a Brownian motion independent of~$\xi$.
     From Cauchy-Schwarz inequality and the fact that $C_\epsilon>0$, we obtain
    		\begin{equation}\label{eq:Ferniqueprelim}
    	\begin{aligned}
    	\big|\tilde{u}^{h_\epsilon}_\epsilon(\tau, y)\big|& \leq\ex^y_{B}\bigg[\int_{\tau}^Tu^2_0(B_r)\exp\bigg(-2\int_{\tau}^{r} h_\epsilon(B_s)ds     \bigg) dr            \bigg]^{\frac{1}{2}}\\&\times\ex^y_B\bigg[\int_{\tau}^T\exp\bigg(-2\int_{\tau}^{r}\big\{ \xi_\epsilon(B_s)-2C_\epsilon\big\}ds     \bigg) dr    \bigg]^{\frac{1}{2}}\\&\leq |v_1^{h_\epsilon}(\tau,y)|^{\frac{1}{2}}|v_2^{\epsilon}(\tau,y)|^{\frac{1}{2}},
    	\end{aligned}
    	\end{equation}
    	where, from another application of Feynman-Kac, $v_1^{h_\epsilon}$ solves 
    	\begin{equation}
    \partial_tv^{h_\epsilon}_1=\Delta v^{h_\epsilon}_1+2v^{h_\epsilon}_1h_{\epsilon},
    \qquad
     v^{h_\epsilon}_1(0,\cdot)=u^2_0(\cdot),
    \end{equation}
    and 
    \begin{equation}
    \partial_tv^{\epsilon}_2=\Delta v^{\epsilon}_2+2v^{\epsilon}_2\big(\xi_\epsilon-2C_\epsilon\big), 
    \qquad
     v^{\epsilon}_2(0,\cdot)=1.
    \end{equation}
     As $\epsilon\to 0$, $\partial_tv^{h_\epsilon}_1$ converges uniformly in $(\tau, y)$ to the smooth solution of the well-posed parabolic PDE
     	\begin{equation}
     \partial_tv^{h}_1=\Delta v^{h}_1+2v^{h}_1h,
     \qquad
     v^{h}_1(0,\cdot)=u^2_0(\cdot).
     \end{equation}
     Writing the solution in mild formulation and applying Gr\"onwall's inequality,
     \begin{equation}
    \| v^{h}_1(t, \cdot)\|_{C(\mathbb{T}^{2})}\leq C\|u_0\|^2_{C(\mathbb{T}^{2})}\bigg(\frac{T}{\tau}\bigg) e^{2\|h\|_{L^2}t}.
     \end{equation}
     Moreover, from the renormalisation theory of~\eqref{eq:PAM}~\cite[Theorem~3.24]{cannizzaro2017malliavin}, $v^{\epsilon}_2$ converges to a well-defined limit $v_2(t,x)$ as $\epsilon\to 0$. Thus, taking $\epsilon\to 0$ in~\eqref{eq:Ferniqueprelim}, we obtain
$$
\big|u^{h}(\tau, y)\big| \leq C|v_2(\tau,y)|^{\frac{1}{2}}\|u_0\|_{C(\mathbb{T}^{2})}\bigg(\frac{T}{\tau}\bigg) e^{\|h\|_{L^2}T},
$$
which implies that
     \begin{equation}\label{eq:Fernique}
     \begin{aligned}
     \big(\log\big|u^{h}(\tau, y)\big|\big)^+& \leq\log\bigg(1+C|v_2(\tau,y)|^{\frac{1}{2}}\|u_0\|_{C(\mathbb{T}^{2})}\bigg(\frac{T}{\tau}\bigg)\bigg)+ T\|h\|_{L^2}.
     \end{aligned}
     \end{equation}
     Since $\big(\log|u^{h}(\tau, y)|\big)^+=\big(\log|\mathscr{R}\mathscr{S}(u_0, T_h\hat{\Pi})(\tau, y)|\big)^+=:\mathscr{G}_{\tau, y}(\xi+h)$ where 
     
     $$\mathscr{G}_{\tau,y}: (\Cc^{-1-\kappa}, \h, \gamma)\rightarrow \R $$ is a measurable map
     from the Wiener space of the noise,  we can rewrite~\eqref{eq:Fernique} as
     \begin{equation}
     \begin{aligned}
    \mathscr{G}_{\tau, y}(\xi+h)& \leq\log\bigg(1+C|v_2(\tau,y)|^{\frac{1}{2}}\|u_0\|_{C(\mathbb{T}^{2})}\bigg(\frac{T}{\tau}\bigg)\bigg)+ T\|h\|_{L^2}.
     \end{aligned}
     \end{equation}    
    By assumption, the non-negative random variable \begin{equation}\label{condition}
     X_{\tau,y}:=\log\bigg(1+C|v_2(\tau,y)|^{\frac{1}{2}}\|u_0\|_{C(\mathbb{T}^{2})}\bigg(\frac{T}{\tau}\bigg)\bigg)\in L^1(\Omega),
     \end{equation}
     and appealing to the generalised Fernique theorem~\cite[Theorem 11.7]{friz2020course}, 
     we deduce that~$\mathscr{G}_{\tau,y}$ has Gaussian tails. 
     In particular, letting~$\Phi$ denote the cumulative distribution function of a standard Normal distribution, then for all $a>0$ and $r>a$,
     $$
     \pr\big[\mathscr{G}_{\tau,y}>r\big]
     \leq\exp\bigg\{ -\half\left(\hat{a}+\frac{r-a}{T}\right)^2\bigg\},
     $$
     with 
     $\hat{a}=\Phi^{-1}(P_a)=\Phi^{-1}\big(\pr[X_{\tau,y}\leq a ]    \big)$,
     and hence for some $\lambda>0$ sufficiently small,  
     $$
     \ex\left[ \exp\left(\lambda   \mathscr{G}_{\tau,y}^2  \right)\right] < \infty.
     $$
     The proof follows by noting that the latter is equivalent to Talagrand's $\mathscr{T}_1(C)$ (see~\cite[Theorem 1.13]{10.1214/ECP.v11-1198} with $\mathcal{X}=\R$ equipped with the standard metric and $\alpha(t)=t^2$).
\end{enumerate}
\end{proof}     
\begin{rem} Theorem~\ref{thm:PAM}(2) implies that, with enough integrability, the solution to~\eqref{eq:PAM} has pointwise log-normal tails: 
with the notation from the proof, for all $a>0, r>e^a$,
$$
\pr\big[ |u(\tau, y)|>r    \big]  \leq\exp\left\{ -\half\left[\hat{a}+\frac{\log(r)-a}{T}\right]^2\right\}.
$$
\end{rem}
     
\begin{rem}
Assumption~\eqref{condition} is satisfied for example when $\tau$ is sufficiently small for all $y\in\mathbb{T}^2$.
Indeed, as shown in~\cite{gu2018moments}, the solution of the 2d PAM has finite moments of all orders for small~$t$. 
Moreover, from~\cite{matsuda2022integrated}, $\ex|u(t, 0)|$ explodes when~$t$ is sufficiently large.
\end{rem}


\section{Weighted logarithmic Sobolev inequalities}\label{Section:WLSIs}
This section can be read independently from the rest of this work and is devoted to weighted logarithmic Sobolev inequalities ($\WLSIs$) for Gaussian functionals. 
In particular, we consider functionals~$\Psi$, defined on an abstract Wiener space $(\Omega, \h, \gamma)$, that take values in a finite-dimensional vector space. We include this analysis here for the following reasons: 1) As explained in the introduction, the tools for obtaining such inequalities are similar in flavour to the ones we used to obtain TCIs: instead of studying $\h$-continuity properties for the functionals of interest, WLSIs rely on $\h$-differentiability properties. 2) WLSIs imply Talagrand's 2-TCI with respect to a weighted metric on $\R^m$ (Theorem~\ref{thm:WLSIconsequences}(2)).

First we extend a contraction principle for $\WLSIs$ proved in~\cite{bartl2020functional} and present some implications of WLSIs. Then, we leverage tools from Malliavin calculus to show that a wide class of  Gaussian functionals satisfy $\WLSIs$ with appropriate weights.

\begin{dfn} 
A probability measure $\mu\in\mathscr{P}(\R^m)$ satisfies $\WLSI(G)$ with $G:\R^m\rightarrow[0,\infty]$ 
if there exists a constant $C>0$ such that 
 \begin{equation}
    Ent_\mu(f^2):=\int_{\R^m}f^2\log\bigg(\frac{f^2}{\int f^2}\bigg)d\mu\leq 
    2C\int_{\R^m}|\nabla f|^2 G d\mu
    \end{equation}
    holds for all differentiable $f:\R^m\rightarrow \R$ for which the right-hand side is finite.
\end{dfn}

\begin{prop}[WLSI contraction principle]\label{Prop:logSobolev}
    Let $(\Omega, \h, \gamma)$ be an abstract Wiener space. Let $\Psi: \Omega\rightarrow \R^m$ be Malliavin differentiable and $G:\R^m\rightarrow[0, \infty]$ such that $$|D\Psi|^2_{\h}\leq c G(\Psi)$$
    holds $\gamma$-a.s. for a constant $c>0$. Then
    \begin{itemize}
    \item[(i)] $\mu=\gamma\circ\Psi^{-1}$ satisfies $\WLSI(2cG)$;
    \item[(ii)] For $i=1,\dots, m$, let~$\mu_i$ denote the $i$-th marginal of~$\mu$. 
    If $G\in L^1(\mu)$, then there exists a measurable function $G_i:\R^{m-1}\rightarrow[0,\infty]$ such that $\mu_i$ satisfies $\WLSI(G_i)$.
    \end{itemize}
\end{prop}

\begin{proof}
    (i) Let $f:\R^m\rightarrow\R$ differentiable. Since $\Psi$ is Malliavin differentiable, so is $f(\Psi)$ and moreover $Df(\Psi)=\nabla f(\Psi) D\Psi$. Recalling~\cite{gross1975logarithmic} that $\gamma$ satisfies $\LSI(2)$, we have
      \begin{align*}
            Ent_{\mu}(f^2)= Ent_{\gamma}((f\circ\Psi)^2)
            &\leq 2\int_{\Omega }\big\|Df(\Psi)\big\|^2_{\h}d\gamma\\&\leq2\int_{\Omega}|\nabla f(\Psi)|^2\big\|D\Psi\big\|^2_{\h}d\gamma\leq 2c\int_{\Omega}|\nabla f(\Psi)|^2G(\Psi)d\gamma.
      \end{align*}
      (ii) Without loss of generality, we shall prove the inequality for $i=1$. Since~$\mu$ is a Borel probability measure on a Polish space, we can apply the disintegration theorem~\cite[Theorem A.5.4]{dupuis2011weak} to write
      $$\mu(dx)=\pi(dx_1| x_2, \dots, x_{m})\mu_1(dx_2,\dots, dx_{m}),
      $$
      where~$\Pi$ is a stochastic kernel on the first coordinate, conditional on $(x_2,\dots, x_{m})$. Now, for any $f=f(x_2,\dots,x_{m})$, the log-Sobolev inequality from~$(i)$, applied to the function $\tilde{f}(x_1,\dots, x_m):=f(x_2,\dots, x_m)$, yields
       \begin{equation}    
      \begin{aligned}        Ent_{\mu_1}(f^2)&=\int_{\R^{m-1}}f^2\log\bigg(\frac{f^2}{\int f^2}\bigg)\bigg(\int_{\R}\pi(dx_1|x_2,\dots, x_m)\bigg)d\mu_1 \\&=   
            Ent_{\mu}(\tilde{f}^2)\\&\leq 2c\int_{\R^{m}}\big|\nabla\tilde{f}|^2Gd\mu\\&=2c\int_{\R^{m}}\big|\nabla f(x_2, \dots, x_m)\big|^2G(x_1, \dots, x_m)d\mu(x_1, \dots, x_m)\\&
            =2c\int_{\R^{m-1}}\big|\nabla f(x_2, \dots, x_m)\big|^2G_1(x_2, \dots, x_m)d\mu_1(x_2, \dots, x_m),
             \end{aligned}   
      \end{equation}
      where $G_1:=2c\int_{\R}G(x_1,x_2 \dots, x_m)\pi(dx_1|x_2,\dots,x_{m})$
      is finite $\mu_1$-a.e. since $\tilde{G}\in L^1(\mu_1)$.
\end{proof} 
\noindent Setting $\Psi\equiv1$ we recover the contraction principle from~\cite[Lemma 6.1]{bartl2020functional} for Lipschitz transformations of the Wiener measure. 
Our generalisation covers non-linear transformations with polynomial growth, and in particular, allows us to prove weighted~$\mathrm{LSIs}$ for functionals of elements in the m-th Wiener chaos over~$\Omega$. The following theorem summarises some useful consequences of WLSIs:

\begin{thm}\label{thm:WLSIconsequences} Let $\mu\in\mathscr{P}(\R^m)$ satisfy $WLSI(G)$ and $d_{G}$ denote the weighted Riemannian distance associated to~$G$, i.e. for all $x,y\in\R^m$ and $C_{xy}$ the set of all absolutely continuous paths $\gamma:[0,1]\rightarrow \R^m$ with $\gamma(0)=x, \gamma(1)=y$,
    $$d_{G}(x,y):=\inf_{\gamma\in C_{xy}}\int_{0}^{1}\sqrt{G^{-1}\big(\gamma(t)\big)|\gamma'(t)|^2dt}.$$  Then the following hold:
\begin{enumerate}
    \item (Bobkov-Ledoux \cite{bobkov1999exponential}) If $G\in L^p(\mu)$, for some $p\geq 2$, then for all $f:\R^m\rightarrow\R$, $\mu$-centered and $1$-Lipschitz one has $\|f\|_{L^p(\mu)}\leq \sqrt{p-1}\|G\|_{L^p(\mu)}$.
    \item (Cattiaux-Guillin-Wu \cite{cattiaux2011some})  $\mu\in\mathscr{T}_2(1)$ with respect to $d_{G/2}$.
\end{enumerate}
\end{thm}

\begin{rem}i) Without loss of generality, it suffices to assume that the weight $G$ is strictly positive, so that the metric $d_G$ above is well-defined. ii) In view of Theorem~\ref{thm:WLSIconsequences}(1), WLSIs are weaker functional inequalities compared to TCIs. In particular, WLSIs only imply finiteness of moments, while TCIs (Proposition~\ref{alphaprop}) imply finiteness of exponential moments for a measure of interest.
\end{rem}

A wide class of Gaussian functionals whose law satisfies a $\WLSI$ is given below. 

\begin{Ex}(Polynomial functionals of Gaussian processes) Let $T>0$ and~$\gamma$ be the law of a one-dimensional continuous, non-degenerate Gaussian process~$X$ on $\Cc_0[0,T]$,
$\h$ its Cameron-Martin space and~$\h'$ the Hilbert space of deterministic integrands with respect to~$X$.
For $\{h_k\}_{k=1}^{m}\subset\h'$, 
let~$ X(h_k)$ denote the Wiener integral $\int h_kdX$ and consider the random vector 
$$
\Psi:=\Big(X^{p_1}(h_1),\dots, X^{p_m}(h_m)\Big),
$$
for some $p_1,\ldots, p_m\geq 1$.
The functional $\Psi$ is Malliavin differentiable with $$
D\Psi = \Big(p_1 X^{p_1-1}(h_1)i(h_1), \dots, p_m X(h_m)^{p_m-1}i(h_m)\Big),
$$
where $i: \h' \rightarrow \h$ is a Hilbert-space isometry and
$$
\|D\Psi\|_{\h}
\leq \left(\max_{k=1,\dots, m}{p_k}\|i(h_k)\|_{\h}\right)\sum_{k=1}^{m}\left|X(h_k)^{p_k-1}\right|
\leq C\left(1+\sum_{k=1}^{m}\left|X(h_k)^{p_k}\right|\right).
$$
Thus, in view of Proposition~\ref{Prop:logSobolev} the law~$\mu$ of $\Psi$ satisfies $\WLSI(G)$ with $G(x)=(1+|x|_{\ell^1})^{2}$.
Applying the same proposition, along with the product rule for Malliavin derivatives, it is straightforward to deduce that any polynomial in~$m$ variables of~$\Psi$ satisfies a $\WLSI$ with an appropriate weight function.
\end{Ex}

\begin{rem} 
For a standard Wiener process~$X$, then 
$\h=H^1_0[0,T]$, i.e. the space of absolutely continuous functions $f$ with a square-integrable weak derivative and $f(0)=0$, $\h'=L^2[0,T]$ and $i(h)=\int_{0}^{\cdot}h_sds$.     
\end{rem}

\begin{Ex}(Gaussian rough paths) Let $T>0$ and~$X$ be a $d$-dimensional continuous Gaussian process on $[0,T]$ that lifts to an $\alpha$-H\"older geometric rough path $\mathbf{X}=(X, \mathbb{X})$ (Definition~\ref{dfn:RP}).
Here, for each $0\leq s\leq t\leq T$,
$\mathbb{X}_{s,t}=\int_{s}^{t}(X_r-X_s)\otimes dX_r$.
For fixed $s,t$ consider the $\R\times\R^d\times\R^{d\otimes d}$-
valued functional
$$
\Psi=(\Psi_1, \Psi_2,  \Psi_3)=\big(\|X\|_{\Cc^\alpha}, \mathbf{X}_{s,t})=\big(\|X\|_{\Cc^\alpha}, X_{s,t}, \mathbb{X}_{s,t}),
$$
defined on the abstract Wiener space $(\Cc^\alpha([0,T]; \R^d), \h, \gamma)$, where, as in the previous example, $\gamma, \h$ denote the law of~$X$ on $\Cc^\alpha([0,T]$ and its Cameron-Martin space respectively.

Regarding the Malliavin differentiability of the first component, note that $\|\cdot\|_{\Cc^\alpha}$ is $\h$-Lipschitz continuous in the sense of~\cite{enchev1993rademacher} (see also~\cite[Exercise 1.2.]{nualart2006malliavin}). Hence, it is indeed differentiable and the triangle inequality gives the estimate
$$
\|D\Psi_1\|_{\h}=\|D\|X\|_{\Cc^\alpha}\|_{\h}\leq 1.
$$

Turning to $\Psi_3$, we show that it is $\h$-continuously Fr\'echet differentiable. To this end let $\epsilon>0$, $h\in\h$, $\omega\in\Omega$ and note that 
$$
\frac{\mathbb{X}_{s,t}(\omega+\epsilon h)-\mathbb{X}_{s,t}(\omega)}{\epsilon} = \int_{s}^{t}X_{s,r}\otimes dh_r+\int_{s}^{t} h_{s,r}\otimes X_r+\epsilon\int_{s}^{t} h_{s,r}\otimes dh_r,
$$
where all the terms on the right-hand side are well-defined Young integrals. Hence the directional derivative 
$$
D^h\mathbb{X}_{s,t}:=\frac{d}{d\epsilon}\mathbb{X}_{s,t}(\omega+\epsilon h)\bigg|_{\epsilon=0}
$$
exists and is linear in~$H$ by linearity of the integrals. Moreover, standard Young estimates along with complementary Cameron-Martin regularity furnish
$$  \big|D^h\mathbb{X}_{s,t}\big|\leq C\|X\|_{\Cc^{\alpha}}\|h\|_{\h},$$ where $C>0$ is a constant that depends on the embedding $ \h\hookrightarrow \Cc^{q-var}$, for all $q$ such that $1/q+1/(2\rho)>1$ and $\alpha<1/(2\rho)$ (for more details on such estimates we refer the interested reader to~\cite[Section~2.2]{friz2010generalized}). 
As a result, the linear operator $\h\ni h\mapsto D^h\mathbb{X}_{s,t}\in\R^{d\otimes d}$ is bounded $\gamma$-a.e. which in particular means that $\mathbb{X}_{s,t}$ is $\h$-continuously Fr\'echet differentiable. Combining the latter with the square-integrability of $\mathbb{X}_{s,t}$, we deduce from~\cite[Lemma~4.1.2 and Proposition~4.1.3]{nualart2006malliavin} that $\Psi_{3}$ is Malliavin differentiable with
$$
\|D\Psi_3\|_{\h}=\sup_{\|h\|_{\h}\leq 1}|D^h\Psi_3|\leq C\|X\|_{\Cc^\alpha}.
$$
The same property holds trivially for $\Psi_2$ since
$h\mapsto X_{s,t}+h_{s,t}$ is Lipschitz continuous and hence $\|D\Psi_2\|_{\h}\leq C$ for a constant $C$ that depends on the embedding $\h\hookrightarrow\Omega$. Putting these estimates together we conclude that $$ \|D\Psi\|_{\h}\leq \sum_{k=1}^{3}\|D\Psi_k\|_{\h}\leq C\bigg( 1+ \|X\|_{\Cc^\alpha} \bigg)=C(1+\Psi_1),$$
hence, in view of Proposition~\ref{Prop:logSobolev}(i), the law~$\mu$ of $\Psi$ satisfies $\WLSI(G)$
with weight function $G:\R\times\R^d\times\R^{d\otimes d}\rightarrow [0, \infty]$, $G(x_1, \mathbf{x}_2)=(1+x_1)^2$.

Finally, it is possible to obtain a $\WLSI$ for the law  $\mu_1$ of the Gaussian rough path $\mathbf{X}_{s,t}=(\Psi_2, \Psi_3)$, for fixed times $s,t\in [0,T]$, by noting that 
$\|G\|_{L^1(\mu)}=\ex[1+\|X\|^2_{\Cc^\alpha} ]<\infty$.
Proposition~\ref{Prop:logSobolev}(ii) then implies that $\mu_1$ satisfies $\WLSI(\tilde{G})$ with $\tilde{G}(\mathbf{x_2})=\int_{\R}(1+x_1)^2\pi(dx_1|\mathbf{x_2})$.
\end{Ex}

At this point we recall a slight generalisation of~\cite[Proposition 4.1.3]{nualart2006malliavin}, that we need for our last example.

\begin{lem} \label{Lem:Dloc} Let $\Psi:\Omega\rightarrow \R^m$ be an $\h$-locally Lipschitz continuous functional in the sense that there exists a real-valued random variable $\Phi$ that is $\gamma$-a.e. finite and such that 
\begin{equation}\label{eq:H-lip}
    |\Psi(\omega+h)-\Psi(\omega)|\leq |\Phi(\omega)|\|h\|_{\h}
\end{equation}
holds for all~$H$ in bounded sets of $\h$ $\gamma$-a.e.
Then $\Psi$ is locally Malliavin differentiable a.e. in the sense that there exist sequences $A_n\subset\Omega$ of measurable sets with $A_n\uparrow A\subset \Omega$, $\gamma(A)=1$,  and Malliavin differentiable functionals $\Psi_n:\Omega\to\R^m$ such that $\Psi=\Psi_n$ on~$\Omega_n$ (the derivative is then defined by $D\Psi:=D\Psi_n$ on~$\Omega_n$).  
Moreover, if $\Phi\in L^2(\Omega)$ then $D\Psi\in L^2(\Omega)$.
\end{lem}

\begin{proof} The proof is essentially the same as that of~\cite{nualart2006malliavin}, Proposition 4.1.3. In particular, the same arguments work by replacing the condition of  $\h$-differentiability with $\h$-local Lipschitz continuity, the $\h$-Fr\'echet derivative $D\Psi$ by the random variable~$\Phi$ and work with the localising sequence 
$$
A_n=\bigg\{ \omega\in \Omega: \sup_{\|h\|_\h\leq1/n}|\Psi(\omega+i(h))|\leq n,\;|\Phi(\omega)|\leq n      \bigg\}, \quad n\in\N.
$$
Note that the finiteness of $\Phi$ guarantees that $A=\bigcup_{n\in\N}A_n$ is a set of probability~$1$. 
These arguments imply existence of the Malliavin derivative of $\Psi_n$. Moreover, we obtain the almost sure estimate $\|D\Psi_n\|_{\h}\leq C_n$ (where $C_n\rightarrow\infty$ as $n\to\infty)$, hence $D\Psi_n\in L^2(\Omega)$. To conclude that $D\Psi\in L^2(\Omega)$, we take advantage of the a.s. existence and linearity of the map
$$h\longmapsto D^h\Psi_n=\frac{d}{d\epsilon}\Psi_n(\omega+\epsilon h)\bigg|_{\epsilon=0}.$$
In view of~\eqref{eq:H-lip} it follows that this linear map is bounded and thus we obtain
$$ \|D\Psi_n\|_{\h}=\sup_{\|h\|_\h\leq 1}|D^h\Psi_n|\leq |\Phi|,\;\; \gamma\text{-a.e}.
$$
Since this bound is uniform in $n$ and $D\Psi=D\Psi_n$ on $\Omega_n$, the conclusion follows.
\end{proof} 

\begin{Ex}(Gaussian RDEs) Let $p\in (1, 3)$, $T>0$ and~$Y$ solve the RDE
\begin{equation}
    \label{Eq: RDE}
    dY=V(Y)d\mathbf{X}, \quad Y_0\in\R^m,
\end{equation}
on the interval $[0,T]$.
Here, for simplicity, we assume that $V=(V_1,\dots, V_d)$ is a collection of $C_b^\infty$ vector fields on $\R^m$ and, as in the previous example,~$X$ is a continuous Gaussian process that lifts to a geometric $p$-(or $1/p$-)H\"older rough path $\mathbf{X}$. With the same abstract Wiener space as in the previous example, consider the $\R^{m+1}$-valued functional
$$\Psi=(\Psi_1, \Psi_2)=(\|\mathbf{X}\|_{\text{p-var}}, Y_T),$$
where $\|\mathbf{X}\|_{\text{p-var}}$ is the (inhomogeneous) $p$-variation "norm" (Definition~\ref{dfn:RP}).
Regarding the differentiability of $\Psi_1$, the triangle inequality along with the estimates on the shifted rough path from the previous example furnish
\begin{equation}
\begin{aligned}
    |\Psi_1(\omega+h)-\Psi_1(\omega)|&\leq \bigg[\sup_{(t_i)\in\mathcal{P}[0,T]}\sum_{i}\bigg(|h_{t_i,t_{i+1}}|+\big|\mathbb{X}_{t_i,t_{i+1}}(\omega+h)-\mathbb{X}_{t_i,t_{i+1}}(\omega)\big| \bigg)^{p}\bigg]^\frac{1}{p}\\&
\leq C(\|h\|_{\h}+ 2\|X\|_{\text{p-var}}\|h\|_{\h}+\|h\|^2),
\end{aligned}
\end{equation}
for a constant that depends on the various embeddings of the underlying $p$-variation spaces. Since $\|X\|_{\text{p-var}}$ is square-integrable, we deduce that, for all~$H$ in bounded sets of $\h$, satisfies the assumptions of Lemma~\ref{Lem:Dloc}. Thus we have 
$$ \|D\Psi_1\|_{\h}\leq \|X\|_{\text{p-var}}\leq  \|\mathbf{X}\|_{\text{p-var}} ,\;\; \gamma-a.e.$$
Turning to $\Psi_2$, the functional~$Y_T$ is $\gamma$-a.s. continuously $\h$-differentiable and 
\begin{equation}
     D^hY_t=\int_{0}^{t}J_{t\leftarrow s}^{\mathbf{X}}V_j(Y_s)dh_s,
\end{equation}
where $J_{t\leftarrow s}^{\mathbf{X}}$ is the Jacobian of the solution flow of~\eqref{Eq: RDE} with respect to the initial condition. The latter satisfies itself a linear RDE and satisfies the estimate
$$    \|J_{\cdot\leftarrow 0}^{\mathbf{X}}\|_{\text{p-var}}\leq C\exp\bigg(c\|\mathbf{X}\|^p_{\text{p-var}}\bigg)$$
for some constants $c,C>0$ (see e.g.~\cite[Equation~(20.17)]{friz2010multidimensional}; notice that the estimate there is formulated in terms of the homogeneous $p$-variation norm which is bounded above by the homogeneous one considered here). In view of the latter, along with the shift invariance of the Jacobian flow and the boundedness of the vector fields~$V$, one has 
$$\|D\Psi_2\|_{\h}=\|DY_T\|_{\h}\leq C_T\exp\bigg(c\|\mathbf{X}\|^p_{\text{p-var}}\bigg). $$
Thus $$  \|D\Psi\|_{\h}\leq C\bigg(\|\mathbf{X}\|+\exp\bigg(c\|\mathbf{X}\|^p_{\text{p-var}}\bigg)\bigg), $$
which in view of Proposition~\eqref{Prop:logSobolev}(i) implies that the law of $\Psi$ satisfies a $\WLSI$ with weight function $G(x)=(x_1+e^{x_1^p})$.
 \end{Ex}

\appendix
\section{Technical proofs}\label{Section:App}
This section is devoted to the proofs of some technical lemmas used throughout this work. As is customary, we remark that values of unimportant constants may change from line to line without a change in notation.
\subsection{Proof of Lemma~\ref{shiftlem}}
\label{proof:shiftlem}
From the definition of the distance~\eqref{twobar} it suffices to obtain estimates for the basis elements $\mathcal{S}\subset\Tt$. 
Let $\lambda\in(0,1]$, $T>0$, $s\in[0,T]$, $h_1, h_2\in\h$ and $\phi\in C_c^1(\R)$ such that 
$\|\phi\|_{\Cc^1}\leq 1$ and $\supp(\phi)\subset(-1,1)$. 
Starting with the symbol~$\Xi$,
\begin{equation}\label{Lip1}
\begin{aligned}
\big|\left\langle \big(T_{h_2}\Pi^{\xi}_{s}-T_{h_1}\Pi^{\xi}_{s}\big)\Xi , \phi^\lambda_s\right\rangle \big|
  & = \left|
  \left\langle \dot{W}+h_2, \phi_s^\lambda\right\rangle - \left\langle \dot{W}+h_1,\phi_s^\lambda\right\rangle\right|
  \\
  &=\frac{1}{\lambda}\bigg|\int_{(s-\lambda, s+\lambda)} 
  \big( h_2(t)-h_1(t) \big)\phi(\lambda^{-1}(t-s))  dt\bigg|\\&\leq\lambda^{-1}\|h_2-h_1\|_{\h} \bigg(\lambda \int_{(-1, 1)}\phi(z)^2 dz\bigg)^{\half}\\&\leq \sqrt{2} \lambda^{-\half}\|\phi\|_{\Cc^1}\|h_2-h_1\|_{\h}\leq \sqrt{2} \lambda^{-\half}\|h_2-h_1\|_{\h},
\end{aligned}	
\end{equation}
where we used Cauchy-Schwarz and $z=\lambda^{-1}(t-s)$.
Next let $m\in\{1,\dots, M\}$ and consider
\begin{align*}
	&\big|\big\langle \big(T_{h_2}\Pi^{\xi}_{s}-T_{h_1}\Pi^{\xi}_{s}\big)I(\Xi)^m, \phi^\lambda_s\big\rangle\big|\\
&=\frac{1}{\lambda}\bigg|\int_{(s-\lambda,s+\lambda )}
\left[\left(\widehat{W}^H(t)-\widehat{W}^H(s)+\widehat{h}_2(t)-\widehat{h}_2(s)\right)^m-\left(\widehat{W}^H(t)-\widehat{W}^H(s)+\widehat{h}_1(t)-\widehat{h}_1(s)\right)^m \right]\\
&\quad\times\phi(\lambda^{-1}(t-s))dt\bigg|.
\end{align*}
Letting $x_i=\widehat{W}^H(s,t)+\widehat{h}_i(s,t):=\widehat{W}^H(t)-\widehat{W}^H(s)+\widehat{h}_i(t)-\widehat{h}_i(s), i=1, 2$ and using the binomial identity
\begin{equation}\label{binom}
x_2^m= x_1^{m}+\sum_{k=1}^{m}\binom{m}{k}x_1^{m-k}(x_2-x_1)^{k},
\end{equation}
we obtain
	\begin{equation}\label{Lip2con}
	\begin{aligned}
	&\frac{1}{\lambda}\bigg|\int_{(s-\lambda,s+\lambda )} \big[\big(\widehat{W}^H(s,t)+\widehat{h}_2(s,t)\big)^m-\big(\widehat{W}^H(s,t)+\widehat{h}_1(s,t)\big)^m \big]\phi(\lambda^{-1}(t-s))dt\bigg|
	\\&=\frac{1}{\lambda}\bigg|\int_{(s-\lambda,s+\lambda )}\sum_{k=1}^{m}\binom{m}{k}\big[\widehat{W}^H(s,t)+\widehat{h}_1(s,t)\big]^{m-k}\big[   \widehat{h}_2(s,t)-\widehat{h}_1(s,t)   \big]^k\phi(\lambda^{-1}(t-s))dt\bigg|\\&	
	\leq   \lambda^{-1+m(H-\kappa)}\sum_{k=1}^{m}\binom{m}{k}\big\|\widehat{W}^H+\widehat{h}_1\big\|_{\Cc^{H-\kappa}}^{m-k}\big\|   \widehat{h}_2-\widehat{h}_1\big\|_{\Cc^{H-\kappa}}^k\int_{(s-\lambda,s+\lambda )}|\phi(\lambda^{-1}(t-s))|  dt\\&\leq m!\lambda^{-1+m(H-\kappa)}\sum_{k=1}^{m}C_H^k\big\|\widehat{W}^H+\widehat{h}_1\big\|_{\Cc^{H-\kappa}}^{m-k}\big\|   h_2-h_1\big\|_{\h}^k\bigg(\lambda\int_{(-1, 1 )} |\phi(z)|ds \bigg)
	\\&	\leq C_m\left(1\vee C^m_H\right)\lambda^{m(H-\kappa)}\|\phi\|_{\Cc^1}
 \left(1\vee\left\|\widehat{W}^H+\widehat{h}_1\right\|_{\Cc^{H-\kappa}}^{m-1} \right)
 \left(\left\|  h_2-h_1\right\|_{\h}\vee \left\|h_2-h_1\right\|^{m}_{\h}\right),
	\end{aligned}	
	\end{equation}
	where we used the change of variables $z=\lambda^{-1}(t-s)$ and the continuity of the linear operator $K^H:\h\rightarrow  \Cc^H$ and  $C_H$ is an upper bound for the operator norm.
	 Turning to $\Xi I(\Xi)^m$,
\begin{align*}
	&\left\langle \big(T_{h_2}\Pi^{\xi}_{s}-T_{h_1}\Pi^{\xi}_{s}\big)\Xi I(\Xi)^m, \phi^\lambda_s\right\rangle \\
	&
	=\left\langle (\dot{W}+h_2)(\widehat{W}^H-\widehat{W}^H(s)+\widehat{h}_2-\widehat{h}_2(s))^m-(\dot{W}+h_1)(\widehat{W}^H-\widehat{W}^H(s)+\widehat{h}_1-\widehat{h}_1(s))^m, \phi^\lambda_s\right\rangle\\
	&
	=\left\langle (h_2-h_1)(\widehat{W}^H-\widehat{W}^H(s)+\widehat{h}_1-\widehat{h}_1(s))^m, \phi^\lambda_s\right\rangle\\
	&+\left\langle (h_2-h_1)\big[(\widehat{W}^H-\widehat{W}^H(s)+\widehat{h}_2-\widehat{h}_2(s))^m-(\widehat{W}^H-\widehat{W}^H(s)+\widehat{h}_1-\widehat{h}_1(s))^m\big], \phi^\lambda_s\right\rangle\\
	&
	+\left\langle (\dot{W}+h_1)\big[(\widehat{W}^H-\widehat{W}^H(s)+\widehat{h}_2-\widehat{h}_2(s))^m-(\widehat{W}^H-\widehat{W}^H(s)+\widehat{h}_1-\widehat{h}_1(s))^m\big], \phi^\lambda_s\right\rangle
	\\
	&=:I_1+I_2+I_3.
\end{align*}
	An application of the Cauchy-Schwarz inequality then yields
	\begin{equation}\label{I1}
	\begin{aligned}
	|I_1|&=\bigg|\big\langle (h_2-h_1)(\widehat{W}^H-\widehat{W}^H(s)+\widehat{h}_1-\widehat{h}_1(s))^m, \phi^\lambda_s\big\rangle     \bigg|\\&=\frac{1}{\lambda}\bigg|\int_{(s-\lambda, s+\lambda )} \big(h_2(t)-  h_1(t)\big)\big(\widehat{W}^H(s,t)+\widehat{h}_1(s,t)\big)^m\phi(\lambda^{-1}(t-s))dt   \bigg|	\\&
	\leq \lambda^{-1+m(H-\kappa)}\big\|\widehat{W}^H+\widehat{h}_1\big\|^m_{\Cc^{H-\kappa}}\|h_2-h_1\|_{\h}\bigg(\lambda\int_{(-1, 1 )} \phi(z)^2dt \bigg)^{\frac{1}{2}}\\&
	\leq \sqrt{2}\lambda^{-\half+m(H-\kappa)}\|\phi\|_{\Cc^1}\big\|\widehat{W}^H+\widehat{h}_1\big\|^m_{\Cc^{H-\kappa}}\|h_1-h_2\|_{\h}.
	\end{aligned}
	\end{equation}
	For $I_2$ we have the estimate 
	\begin{equation}\label{I2}
	|I_2|\leq C\lambda^{-\half+m(H-\kappa)}
	\left(1\vee\big\|\widehat{W}^H+\widehat{h}_1\big\|_{\Cc^{H-\kappa}}^{m-1} \right)
	\left(\big\|  h_2-h_1\big\|_{\h}\vee \big\|h_2-h_1\big\|^{m+1}_{\h}\right).
	\end{equation}
	This can be proved using very similar arguments as the ones used to obtain~\eqref{Lip2con}. To avoid repetition, its proof will be omitted.
	Turning to $I_3$ we have 
		\begin{equation*}
	\begin{aligned}
	I_3&=\left\langle h_1\big[(\widehat{W}^H-\widehat{W}^H(s)+\widehat{h}_2-\widehat{h}_2(s))^m-(\widehat{W}^H-\widehat{W}^H(s)+\widehat{h}_1-\widehat{h}_1(s))^m\big], \phi^\lambda_s\right\rangle\\&
	+\left\langle \dot{W}\big[(\widehat{W}^H-\widehat{W}^H(s)+\widehat{h}_2-\widehat{h}_2(s))^m-(\widehat{W}^H-\widehat{W}^H(s)+\widehat{h}_1-\widehat{h}_1(s))^m\big], \phi^\lambda_s\right\rangle\\&
	=\frac{1}{\lambda}\int_{(s-\lambda, s+\lambda )}
	\left[(\widehat{W}^H(s,t)+\widehat{h}_2(s,t)-(\widehat{W}^H(s,t)+\widehat{h}_1(s,t))^m\right]
	h_1(t)\phi(\lambda^{-1}(t-s))dt\\
	&
	-\frac{1}{\lambda^2}\int_{(s-\lambda, s+\lambda )}\phi'(\lambda^{-1}(t-s))\int_{s}^{t}\left[(\widehat{W}^H(s,r)+\widehat{h}_2(s,r))^m-(\widehat{W}^H(s,r)+\widehat{h}_1(s,r))^m\right]dW_rdt\\
	& =:J_1+J_2.
		\end{aligned}
	\end{equation*}
In view of~\eqref{binom} we have 
	\begin{equation}\label{J1}
\begin{aligned}
&|J_1| = \frac{1}{\lambda}\bigg|\int_{(s-\lambda, s+\lambda )}\sum_{k=1}^{m}\binom{m}{k}\left[\widehat{W}^H(s,t)+\widehat{h}_1(s,t)\big]^{m-k}\big[\widehat{h}_2(s,t)-\widehat{h}_1(s,t)\right]^{k}h_1(t)\phi(\lambda^{-1}(t-s))dt\bigg|\\&
\leq m!\lambda^{-1+m(H-\kappa)}\sum_{k=1}^{m}\big\|\widehat{W}^H+\widehat{h}_1\big\|_{\Cc^{H-\kappa}}^{m-k}\|\widehat{h}_2-\widehat{h}_1\|^{k}_{\Cc^{H-\kappa}}\int_{(s-\lambda, s+\lambda )}|h_1(t)|\big|\phi(\lambda^{-1}(t-s))\big|dt\\&
\leq m!\lambda^{-\half+m(H-\kappa)}
\left(1\vee \big\|\widehat{W}^H+\widehat{h}_1\big\|_{\Cc^{H-\kappa}}^{m-1}\right)
\left(1\vee C^m_H)\big(\|h_2-h_1\|_{\h}\vee \|h_2-h_1\|_{\h}^{m}\right)\|h_1\|_{\h}\|\phi\|_{L^2(-1,1)},
\end{aligned}
\end{equation}
where we used Cauchy-Schwarz and the continuity of the embedding $\h\hookrightarrow \Cc^H$  once again. 
Finally, let 
$$
M_{st} := \int_{0}^{t}\left[(\widehat{W}^H(s,r)+\widehat{h}_2(s,r))^m-(\widehat{W}^H(s,r)+\widehat{h}_1(s,r))^m\right]dW_r, \quad s\leq t.
$$
By the BDG inequality and~\eqref{binom} we have 
\begin{equation*}
\begin{aligned}
&\ex\left[\big|M_{st}-M_{ss}\big|^p\right]
\leq \left(\int_{s}^{t}\ex\left[\left|\left(\widehat{W}^H(s,r)+\widehat{h}_2(s,r)\right)^m
- \left(\widehat{W}^H(s,r)+\widehat{h}_1(s,r)\right)^m\right|^2\right]dr\right)^{\frac{p}{2}}\\
&
\leq C_m^p\left\{\int_{s}^{t}\left(\sum_{k=1}^{m}
\ex\left[\left|\widehat{W}^H(s,r)+\widehat{h}_1(s,r)\right|^{m-k}\right]\left|\widehat{h}_2(s,r)-\widehat{h}_1(s,r)\right|^{k}\right)^2dr\right\}^{\frac{p}{2}}\\
&
\leq C_m^p\left\{\sum_{k=1}^{m}
\bigg[\ex\big\|\widehat{W}^H+\widehat{h}_1\big\|^{(m-k)}_{\Cc^{H-\kappa}}\|\widehat{h}_2-\widehat{h}_1\|^{k}_{\Cc^{H-\kappa}}\bigg]^{p}\bigg[\int_{s}^{t}(r-s)^{2m(H-\kappa)}dr\bigg]\right\}^{\frac{p}{2}}
\\
&
\leq C^p_{H,m}\left(1\vee \ex\left[\left\|\widehat{W}^H+\widehat{h}_1\right\|_{\Cc^{H-\kappa}}^{p(m-1)}\right]\right)\left(\|h_2-h_1\|^p_{\h}\vee \|h_2-h_1\|_{\h}^{mp}\right)(t-s)^{pm(H-\kappa)+p/2}.
\end{aligned}
\end{equation*}
An application of the Kolmogorov continuity criterion for two-parameter processes~\cite[Theorem 3.13]{friz2020course} furnishes
\begin{equation*}
\begin{aligned}
\big|M_{st}-M_{ss}\big|\leq K_{h_1}\big(\|h_2-h_1\|_{\h}\vee \|h_2-h_1\|_{\h}^{m}\big)(t-s)^{m(H-\kappa)+\half-\kappa}
\end{aligned}
\end{equation*}
almost surely, where $K_{h_1}$ is a random variable with finite moments of all orders (this estimate can also be obtained by the Young bounds used for the proof of Lemma~\ref{Itoshiftlem} in Section~\ref{proof:Itoshiftlem} below). Plugging this estimate into the expression for $J_2$ then yields
	\begin{equation}\label{J2}
\begin{aligned}
|J_2|&\leq \lambda^{-2}\int_{(s-\lambda, s+\lambda )}|\phi'(\lambda^{-1}(t-s))|\big|M_{st}-M_{ss}\big|dt\\&\leq  K\big(\|h_2-h_1\|_{\h}\vee \|h_2-h_1\|_{\h}^{m}\big)\lambda^{-2}\int_{(s-\lambda, s+\lambda )}|\phi'(\lambda^{-1}(t-s))|(t-s)^{m(H-\kappa)+\half-\kappa}dt\\&
\leq K\big(\|h_2-h_1\|_{\h}\vee \|h_2-h_1\|_{\h}^{m}\big)\lambda^{-2+1+m(H-\kappa)+\half-\kappa}\int_{(-1,1)}|\phi'(t)|t^{m(H-\kappa)+\half-\kappa}dt\\&
\leq CK\big(\|h_2-h_1\|_{\h}\vee \|h_2-h_1\|_{\h}^{m}\big)\lambda^{m(H-\kappa)-\half-\kappa}\|\phi\|_{\Cc^1}.
\end{aligned}
\end{equation}
A combination of~\eqref{I1},~\eqref{J1},~\eqref{J2} implies the almost sure bound
$$
\big|\big\langle \big(T_{h_2}\Pi^{\xi}_{s}-T_{h_1}\Pi^{\xi}_{s}\big)\Xi I(\Xi)^m, \phi^\lambda_s\big\rangle\big|\leq C_{m,H}\lambda^{m(H-\kappa)-\half-\kappa}K_{h_1}\|h_2-h_1\|_{\h}\vee \|h_2-h_1\|_{\h}^{m+1},
$$
where $K_{h_1}$ is a random variable with finite moments of all orders. 
The proof follows from this with~\eqref{Lip1}-\eqref{Lip2con}-\eqref{I2} 
and recalling the definition of model distance~\eqref{twobar}.
\subsection{Proof of Lemma~\ref{Itoshiftlem}}\label{proof:Itoshiftlem}
  Let $p\geq(\half-\kappa)^{-1}$. For $t\in[0,T], h\in L^2[0,T]$ we have 
\begin{equation}\label{eq:Itoshift}
\begin{aligned}
     &\int_{s}^{t}f( \widehat{W}^H_r+\widehat{h}_r, r)d\bigg(W_r+\int_{0}^{r}h_zdz\bigg)-\int_{s}^{t}f( \widehat{W}^H_r, r)dW_r\\&= \int_{s}^{t}f_1( \widehat{h}_r, r)f_2(\widehat{W}^H_r,r)h_rdr+\int_{s}^{t}f_1(\widehat{h}_r,r)f_2\big( \widehat{W}^H_r, r\big)dW_r
\end{aligned}   
\end{equation}
From the Besov variation embedding, there exists $q$ with $1/p+1/q>1$ such that $\|\widehat{h}\|_{q-var;[s,t]}\leq C|t-s|^{H-\kappa}\|h\|_{\h}$. Moreover, 
for each $u,v\in[s,t]$, we have by the mean value inequality
\begin{equation}
    \begin{aligned}
        |f_1(\widehat{h}_u,u)-f_1(\widehat{h}_v,v)|^q&\leq C\big|G(\|\widehat{h}\|_{\infty})\big|^q\bigg(  |\widehat{h}_u-\widehat{h}_v|^q+|u-v|^q  \bigg).
    \end{aligned}
\end{equation}
Hence, by monotonicity of $G$ it follows that 
\begin{equation}
    \begin{aligned}
        \big\|f_1(\widehat{h}_\cdot,\cdot)\big\|_{q-var; [s,t]}&\leq  C\big|G(c\|h\|_{\h})\big|\bigg(  \|\widehat{h}\|_{q-var;[s,t]}+|t-s|\bigg)\leq C_{T}\big|G(c\|h\|_{\h})\big|\bigg(\|h\|_{\h}+1\bigg)|t-s|^{H-\kappa}.
    \end{aligned}
\end{equation}
Finally, the growth assumptions on $f_2$ guarantee that $\int_{s}^{\cdot}f_2(\widehat{h}_r, r)dW_r$ is $1/p$-H\"older continuous on $[s,t]$ which in turn implies that $$\bigg\|\int_{s}^{\cdot}f_2(\widehat{W}^H_r, r)dW_r \bigg\|_{\text{p-var};[s,t]}\leq \bigg\|\int_{s}^{\cdot}f_2(\widehat{W}^H_r, r)dW_r\bigg\|_{\Cc^{1/p}[s,t]}|t-s|^{1/p}. $$
Combining the last two estimates with Young's inequality we obtain the almost sure bound
\begin{equation}
    \bigg\|\int_{0}^{\cdot}f_2(\widehat{W}^H_r, r)f_1(\widehat{h}_r, r)dW_r\bigg\|_{\Cc^{H-\kappa+1/p}[0,T]}\leq C\big|G(c\|h\|_{\h})\big|\bigg(\|h\|_{\h}+1\bigg)K,
\end{equation}
where $K$ is a random variable with finite moments of all orders. As for the Riemann integral in~\eqref{eq:Itoshift}, Cauchy-Schwarz yields a similar estimate for its $\Cc^{1/2}[0,T]$ norm.

\subsection{Proof of Lemma~\ref{modelbndlem}(1)}\label{Subsection: modelbndlem}
Let $t\neq s\in[0,T]$. By linearity of $\Gamma_{t,s}$ and~\eqref{eq:structuregroup} we have 
	\begin{equation*}
	\begin{aligned}
	\Gamma_{t,s}f^{\Pi}(s)&=\sum_{k=0}^{M}\frac{1}{k!}\partial^k_1f\big((K^H*\Pi_s\Xi)(s),s\big)\Gamma_{t,s}I(\Xi)^k=\sum_{k=0}^{M}\frac{1}{k!}\partial^k_1f\big((K^H*\Pi_s\Xi)(s),s\big)\big(\Gamma_{t,s}I(\Xi)\big)^k\\&=\sum_{k=0}^{M}\sum_{m=0}^{M-k}\frac{1}{k!m!}\partial^{k+m}_1f\big((K^H*\Pi_t\Xi)(t),s\big)\big[ (K^H*\Pi_s\Xi)(s)-(K^H*\Pi_t\Xi)(t)     \big]^m\big(\Gamma_{t,s}I(\Xi)\big)^k\\&
	+\sum_{k=0}^{M}\frac{1}{k!(M-k)!}\bigg(\int_{(K^H*\Pi_t\Xi)(t)}^{(K^H*\Pi_s\Xi)(s)}\partial^{M+1}_1f\big(x,s\big)\big[ (K^H*\Pi_s\Xi)(s)-x\big]^{M-k}dx\bigg)\big(\Gamma_{t,s}I(\Xi)\big)^k,
	\end{aligned}
	\end{equation*}
	where we Taylor-expanded $\partial_1^kf$ around $(K^H*\Pi_t\Xi)(t)$ up to the $(M-k)$th degree. Writing $\partial_1^{k+m}f(\cdot,s)=\partial_1^{k+m}f(\cdot,t)+\int_{t}^{s}\partial_{1,2}^{k+m,1}f(\cdot,x)dx$, we obtain
		\begin{equation}
	\begin{aligned}
	\Gamma_{t,s}f^{\Pi}(s)&=\sum_{k=0}^{M}\sum_{m=0}^{M-k}\frac{1}{k!m!}\partial^{k+m}_{1}f\big((K^H*\Pi_t\Xi)(t),t\big)\big[ (K^H*\Pi_s\Xi)(s)-(K^H*\Pi_t\Xi)(t)     \big]^m\big(\Gamma_{t,s}I(\Xi)\big)^k\\&+R(t,s),
	\end{aligned}
	\end{equation}
	where the remainder is given by
	\begin{equation}\label{Rkform}
	\begin{aligned}
	R&(t,s):=\sum_{k=0}^{M}\frac{1}{k!(M-k)!}\bigg(\int_{(K^H*\Pi_t\Xi)(t)}^{(K^H*\Pi_s\Xi)(s)}\partial^{M+1}_1f\big(x,s\big)\big[ (K^H*\Pi_s\Xi)(s)-x\big]^{M-k}dx\bigg)\big(\Gamma_{t,s}I(\Xi)\big)^k
	\\&	+\sum_{k=0}^{M}\sum_{m=0}^{M-k}\frac{1}{k!m!}\bigg(\int_{t}^{s}\partial^{k+m,1}_{1,2}f\big((K^H*\Pi_t\Xi)(t),x\big)dx\bigg)\big[ (K^H*\Pi_s\Xi)(s)-(K^H*\Pi_t\Xi)(t)     \big]^m\big(\Gamma_{t,s}I(\Xi)\big)^k\\&
	=	\sum_{k=0}^{M}\tilde{R}_k(t,s)\big(\Gamma_{t,s}I(\Xi)\big)^k.
	\end{aligned}
	\end{equation}
	In view of the relations
$\Gamma_{t,s}I(\Xi)=I(\Xi)-\big(\Pi_tI(\Xi) \big) (s)\mathbf{1}$ 
and 
$$
(K^H*\Pi_s\Xi)(s)-(K^H*\Pi_t\Xi)(t)=\Pi_tI(\Xi)(s)\equiv \Pi_tI(\Xi)(s)\mathbf{1},
$$
which are directly inferred from~\eqref{eq:IXi}-\eqref{eq:structuregroup}-\eqref{eq:Gamma} 
and the binomial identity~\eqref{binom}, we obtain 
	\begin{equation}
\begin{aligned}
\Gamma_{t,s}f^{\Pi}(s)&-R(t,s)=\sum_{k=0}^{M}\sum_{m=0}^{M-k}\frac{1}{k!m!}\partial^{k+m}_{1}f\big((K^H*\Pi_t\Xi)(t),t\big)\big[ \Pi_tI(\Xi)(s)\big]^m\big[I(\Xi)-\big(\Pi_tI(\Xi) \big)(s) \mathbf{1}\big]^k\\&
=\sum_{k=0}^{M}\sum_{m'=k}^{M}\frac{1}{k!(m'-k)!}\partial^{m'}_{1}f\big((K^H*\Pi_t\Xi)(t),t\big)\big[ \Pi_tI(\Xi)(s)\big]^{m'-k}\big[I(\Xi)-\big(\Pi_tI(\Xi) \big)(s) \mathbf{1}\big]^k\\&
=\sum_{m'=0}^{M}\frac{\partial^{m'}_{1}f\big((K^H*\Pi_t\Xi)(t),t\big)}{m'!}\sum_{k=0}^{m'}\binom{m'}{k}\big[ \Pi_tI(\Xi)(s)\big]^{m'-k}\big[I(\Xi)-\big(\Pi_tI(\Xi) \big)(s) \mathbf{1}\big]^k\\&
=\sum_{m'=0}^{M}\frac{1}{m'!}\partial^{m'}_{1}f\big((K^H*\Pi_t\Xi)(t),t\big)I(\Xi)^{m'}=f^{\Pi}(t),
\end{aligned}
\end{equation}
where we set $m+k=m'$ and then interchanged the order of summation to obtain the third equality. Turning to the remainder, we have 
\begin{align}\label{remainderform}
R(t,s)
 & = \sum_{k=0}^{M}\tilde{R}_k(t,s)\big(\Gamma_{t,s}I(\Xi)\big)^k=\sum_{k=0}^{M}\tilde{R}_k(t,s)\big[I(\Xi)-\big(\Pi_tI(\Xi) \big) (s)\mathbf{1}\big]^k\nonumber\\
 & = \sum_{k=0}^{M}\tilde{R}_k(t,s)\big[I(\Xi)+\big(\Pi_sI(\Xi) \big) (t)\mathbf{1}\big]^k
=\sum_{k=0}^{M}\tilde{R}_k(t,s)\sum_{\ell=0}^{k}\binom{k}{\ell}\big(\Pi_sI(\Xi)(t) \big)^{k-\ell}I(\Xi)^{\ell}\nonumber\\
 & = \sum_{\ell=0}^{M} \left\{\sum_{k=\ell}^{M}\tilde{R}_k(t,s)\binom{k}{\ell}\big(\Pi_sI(\Xi)(t) \big)^{k-\ell}\right\}I(\Xi)^{\ell}.
\end{align}
The analogous expressions for  $\mathscr{D}(\Pi)$ follow by multiplying throughout by the symbol $\Xi$. The latter is possible since, in view of~\eqref{eq:structuregroup}, we have $\Gamma_{t,s}(I(\Xi)^k\Xi)=(\Gamma_{t,s}I(\Xi))^k\Xi$.
In view of~\eqref{remainderform}, along with the fact that $f^{\Pi}(t)-\Gamma_{t,s}f^{\Pi}(s)=R(t,s)$, we see that in order to estimate $\|f^\Pi\|_{\mathcal{D}^\gamma_{T}}$ (recall~\eqref{eq:Dgammanorm}), one has to bound the terms~$\tilde{R}_k$.
To this end, let $k\in{\ell,\dots,M}$ and $a=(K^H*\Pi_t\Xi)(t), b=(K^H*\Pi_s\Xi)(s)$. 
Assuming first $a<b$, a change of variable gives
\begin{equation}
\begin{aligned}
\int_{a}^{b}\partial^{M+1}_1f\big(x,s\big)\big[ b-x\big]^{M-k}dx&=\int_{0}^{b-a}\partial^{M+1}_1f\big(x+a,s\big)(b-a-x)^{M-k}dx\\&
\leq \int_{0}^{b-a}\big|\partial^{M+1}_1f\big(x+a,s\big)\big|(b-a-x)^{M-k}dx\\&
\leq C_{f,T}\int_{0}^{b-a}\big(1+G(|x+a|)\big)(b-a-x)^{M-k}dx\\&
\leq \frac{C_{f,T}}{M+1-k}(b-a)^{M+1-k}(1+G(|b|+2|a|)),
\end{aligned}
\end{equation}
where we used the assumption~\eqref{fgrowth}.
The case $b<a$ is symmetric, namely
\begin{align*}
-\int_{b}^{a}\partial^{M+1}_1 f\big(x,s\big)\big[ b-x\big]^{M-k}dx
 & =-\int_{0}^{a-b}\partial^{M+1}_1f\big(x+b,s\big)(-x)^{M-k}dx\\
& \leq \int_{0}^{a-b}\big|\partial^{M+1}_1f\big(x+b, s\big)\big|x^{M-k}dx\\
& \leq C_{f,T}\int_{0}^{a-b}(1+G(|x+b|))x^{M-k}dx\\
& \leq \frac{C_{f,T}}{M+1-k}(a-b)^{M+1-k}(1+G(|a|+2|b|)).
\end{align*}
Since~$G$ is non-decreasing, we combine both cases to obtain
$$\bigg|\int_{a}^{b}\partial^{M+1}_1f\big(x,s\big)\big[ b-x\big]^{M-k}dx\bigg|\leq \frac{C_{f,T}}{M+1-k}\bigg[1+G(2|a|+2|b|)\bigg]\big|b-a\big|^{M+1-k}.$$
Substituting back $a$ and $b$, the latter yields
\begin{align}\label{Rk1bnd}
&\int_{(K^H*\Pi_t\Xi)(t)}^{(K^H*\Pi_s\Xi)(s)}\partial^{M+1}_1f\big(x,s\big)\big[ (K^H*\Pi_s\Xi)(s)-x\big]^{M-k}dx\nonumber\\
 & \leq \frac{C_{f,T}}{M+1-k}\bigg[1+G\bigg(2\big|(K^H*\Pi_s\Xi)(s)\big|+2\big|(K^H*\Pi_t\Xi)(t)\big|\bigg)\bigg]\big|(K^H*\Pi_t\Xi)(t)-(K^H*\Pi_s\Xi)(s)\big|^{M+1-k}\nonumber\\
  & \leq \frac{C_{f,T}}{M+1-k}\bigg[1+G\bigg(4\big\|(K^H*\Pi\Xi)\big\|_{C[0,T]}\bigg)\bigg]\big\|K^H*\Pi\Xi\big\|^{M+1-k}_{\Cc^{H-\kappa}[0,T]}|t-s|^{(M+1-k)(H-\kappa)},
\end{align}
with $\kappa<H$, recalling that $K^H*\Pi\Xi$ is function-valued for the models we consider. 
It remains to estimate the second summand in~\eqref{Rkform}, smoother since $f(x,\cdot)$ is bounded. Indeed, 
\begin{equation*}
\begin{aligned}
\bigg|\int_{t}^{s}\partial^{k+m,1}_{1,2}f\big((K^H*\Pi_t\Xi)(t),x\big)dx\bigg|&\leq  |t-s| \sup_{x\in[0,T]}\big|\partial^{k+m,1}_{1,2}f\big((K^H*\Pi_t\Xi)(t),x\big)\big|\\&
\leq C_{f,T}|t-s|\bigg[    1+G\bigg(\big\|  K^H*\Pi\Xi\big\|_{C[0,T]}  \bigg)\bigg].
\end{aligned}
\end{equation*}
Thus, 
\begin{equation}\label{Rk2bnd}
\begin{aligned}
&\sum_{m=0}^{M-k}\frac{1}{k!m!}\bigg(\int_{t}^{s}\partial^{k+m,1}_{1,2}f\big((K^H*\Pi_t\Xi)(t),x\big)dx\bigg)\big[(K^H*\Pi_s\Xi)(s)-(K^H*\Pi_t\Xi)(t)     \big]^m\\&
\leq  C_{f,T}\bigg[    1+G\bigg(\big\|  K^H*\Pi\Xi\big\|_{C[0,T]}  \bigg)\bigg]\sum_{m=0}^{M-k}\frac{1}{k!m!}\big\|K^H*\Pi\Xi\big\|^{m}_{\Cc^{H-\kappa}[0,T]} |t-s|^{1+m(H-\kappa)}\\
& \leq \frac{C_{f,T}|t-s|(M+1-k)}{k!}\bigg[  1+G\bigg(\big\|  K^H*\Pi\Xi\big\|_{C[0,T]}  \bigg)\bigg]
\big(1\vee T^{(M-k)(H-\kappa)}\big) \bigg(1\vee\big\|K^H*\Pi\Xi\big\|^{M-k}_{\Cc^{H-\kappa}[0,T]}\bigg).
\end{aligned}
\end{equation}
Putting~\eqref{Rk1bnd},~\eqref{Rk2bnd} together we have 
\begin{equation}\label{Rkbnd}
\begin{aligned}
\frac{|\tilde R_k(t,s)|}{|t-s|+|t-s|^{(M+1-k)(H-\kappa)}}
\leq  \frac{C_{f,T}(M+1-k)}{k!}&\left(1\vee T^{(M-k)(H-\kappa)}\right) \bigg(1\vee\big\|K^H*\Pi\Xi\big\|^{M-k+1}_{\Cc^{H-\kappa}[0,T]}\bigg)\\&\times\bigg[  1+G\bigg(4\big\|  K^H*\Pi\Xi\big\|_{C[0,T]}  \bigg)\bigg].
\end{aligned}
\end{equation}
Since the models we consider are admissible (in the sense of~\cite[Lemma 3.19]{bayer2020regularity}), then $\Pi_sI(\Xi)(t)=K^H*\Pi_s\Xi(t)-K^H*\Pi_s\Xi(s)$.
In view of the latter, plugging the last estimate into~\eqref{remainderform} yields, for $\ell=0,\dots, M$ and $\beta=\ell(H-\kappa)$
\begin{equation}
\begin{aligned}
|R(t,s)|_{\beta}&=\bigg|  \sum_{k=\ell}^{M}\tilde{R}_k(t,s)\binom{k}{\ell}\big(\Pi_sI(\Xi)(t) \big)^{k-\ell}    \bigg|\leq \sum_{k=\ell}^{M}|\tilde{R}_k(t,s)|\binom{k}{\ell}|\Pi_sI(\Xi)(t)|^{k-\ell}\\&
\leq \sum_{k=\ell}^{M}\binom{k}{\ell}|\tilde{R}_k(t,s)|\big\|K^H*\Pi\Xi\big\|^{k-\ell}_{\Cc^{H-\kappa}[0,T]}|t-s|^{(k-\ell)(H-\kappa)}\\&
\leq C_{f,T}\big(1\vee T^{(M-\ell)(H-\kappa)}\big)\bigg(1\vee\big\|K^H*\Pi\Xi\big\|^{M-\ell+1}_{\Cc^{H-\kappa}[0,T]}\bigg)\bigg[  1+G\bigg(4\big\|  K^H*\Pi\Xi\big\|_{C[0,T]}  \bigg)\bigg]\times\\&\times\sum_{k=\ell}^{M}\binom{k}{\ell}\frac{M+1-k}{k!}\bigg( |t-s|^{1+(k-\ell)(H-\kappa)} + |t-s|^{(M+1-\ell)(H-\kappa)}         \bigg).
\end{aligned}
\end{equation}
Therefore, modulo a constant,
\begin{equation}
\begin{aligned}
\frac{|R(t,s)|_{\beta}}{|t-s|^{\frac{1}{2}+\kappa-\beta}}\lesssim\sum_{k=\ell}^{M}\binom{k}{\ell}\frac{M+1-k}{k!}\bigg( |t-s|^{\half-\kappa+k(H-\kappa)} + |t-s|^{(M+1)(H-\kappa)-\half-\kappa}         \bigg).
\end{aligned}
\end{equation}
Since $\kappa<H$ and $M$ is chosen according to~\eqref{Mchoice}, the exponents on the right-hand side are positive. Hence, for any $0<\gamma<((M+1)(H-\kappa)-\half-\kappa)\wedge(\half-\kappa)$,
\begin{equation}
\begin{aligned}
\frac{|R(t,s)|_{\beta}}{|t-s|^{\frac{1}{2}+\kappa+\gamma-\beta}}&\lesssim(M+1-\ell)\sum_{k=\ell}^{M}\frac{1}{\ell!(k-\ell)!}\bigg( |t-s|^{\half-\kappa-\gamma+k(H-\kappa)} + |t-s|^{(M+1)(H-\kappa)-\half-\kappa-\gamma}         \bigg)\\&
\leq \frac{(M+1-\ell)(M-\ell)}{\ell!}\bigg(1\vee T^{\alpha}   \bigg),
\end{aligned}
\end{equation}
where $\alpha=\half-\kappa-\gamma+M(H-\kappa)$.
Combining the last estimates and applying crude bounds for the terms that depend on~$\ell$, we obtain 
$$
\sup_{\substack{s\neq t\in[0,T]\\A\ni\beta<\frac{1}{2}+\kappa+\gamma}}\frac{\big| f^\Pi(t)-\Gamma_{t,s}f^\Pi(s)\big|_{\beta}}{|t-s|^{\frac{1}{2}+\kappa+\gamma-\beta}}
\leq C_{f,T, M}\bigg[  1+G\bigg(4\big\|  K^H*\Pi\Xi\big\|_{C[0,T]}  \bigg)\bigg]\bigg(1\vee\big\|K^H*\Pi\Xi\big\|^{M+1}_{\Cc^{H-\kappa}[0,T]}\bigg).
$$
After multiplying by~$\Xi$, the same estimate also implies that for $\ell=0,\dots,M$ and $\beta'=\ell(H-\kappa)-\half-\kappa$, 
\begin{equation}
\begin{aligned}
\sup_{\substack{s\neq t\in[0,T]\\A\ni\beta'<\gamma}}&\frac{\big| \mathscr{D}(\Pi)(t)-\Gamma_{t,s}\mathscr{D}(\Pi)(s)\big|_{\beta'}}{|t-s|^{\gamma-\beta'}}=\sup_{\substack{s\neq t\in[0,T]\\A\ni\beta<\frac{1}{2}+\kappa+\gamma}}\frac{\big| f^\Pi(t)-\Gamma_{t,s}f^\Pi(s)\big|_{\beta}}{|t-s|^{\frac{1}{2}+\kappa+\gamma-\beta}}\\&\leq C_{f,T,M}\bigg[  1+G\bigg(4\big\|  K^H*\Pi\Xi\big\|_{C[0,T]}  \bigg)\bigg]\bigg(1\vee\big\|K^H*\Pi\Xi\big\|^{M+1}_{\Cc^{H-\kappa}[0,T]}\bigg).
\end{aligned}
\end{equation}

\subsection{Proof of Lemma~\ref{modelbndlem}(2)}
Let $i=1,2, t\in[0,T]$. Since $$\mathscr{D}_f(\Pi^i)(t)=f^{\Pi^i}\star\Xi(t)=\sum_{m=0}^{M}\frac{1}{m!}\partial_1^{m}f\big((K^H*\Pi^i_t\Xi)(t),t\big)I(\Xi)^k\Xi,$$
our assumptions on~$f$ directly yield
\begin{equation}
\begin{aligned}
    \sup_{t\in[0,T]}\sup_{ A\ni\beta<\gamma}\big| \mathscr{D}_f(\Pi^1)(t)&-\mathscr{D}_f(\Pi^2)(t)\big|_{\beta}\leq c_{f,T, M}(1\vee\big\|K^H*\Pi^1\Xi\big\|_{C[0,T]}^N)\\&\times\big\|K^H*\Pi^1\Xi-K^H*\Pi^2\Xi\big\|_{C[0,T]}\vee\big\|K^H*\Pi^1\Xi-K^H*\Pi^2\Xi\big\|^N_{C[0,T]}.
\end{aligned}
\end{equation}
Turning to the remainder terms, a computation similar to~\eqref{remainderform} furnishes
$$\mathscr{D}_f(\Pi^i)(t)-\Gamma^i_{t,s}\mathscr{D}_f(\Pi^i)(s)=\sum_{\ell=0}^{M}\bigg\{\sum_{k=\ell}^{M}R^i_k(t,s)\binom{k}{\ell}\big(\Pi^i_sI(\Xi)(t) \big)^{k-\ell}\bigg\}I(\Xi)^{\ell}\Xi,$$
where, using the notation $(K^H*\Pi^i\Xi)(t,s):=(K^H*\Pi^i_s\Xi)(s)-(K^H*\Pi^i_t\Xi)(t)$, $i=1,2$,
	\begin{equation}\label{Rkform2}
\begin{aligned}
R^i_k(t,s)
 &:= \frac{\left(K^H*\Pi^i\Xi\right)(t,s)^{M+1-k}}{k!(M-k)!}
\left[\int_{0}^{1}\partial^{M+1}_1f\bigg((K^H*\Pi^i_t\Xi)(t)+\theta\bigg[ (K^H*\Pi^i\Xi)(t,s)\bigg],s\bigg)(1-\theta)^Md\theta\right]\\
&	+\sum_{m=0}^{M-k}\frac{1}{k!m!}\big[   (K^H*\Pi^i\Xi)(t,s) \big]^m     \bigg(\int_{t}^{s}\partial^{k+m,1}_{1,2}f\big((K^H*\Pi^i_t\Xi)(t),x\big)dx\bigg)\\&=:A^i_k(t,s)+B^i_k(t,s).
\end{aligned}
\end{equation}
Thus, 
\begin{equation}\label{eq:Remdecomp}
\begin{aligned}
\mathscr{D}_f(\Pi^2)(t)&-\Gamma^2_{t,s}\mathscr{D}_f(\Pi^2)(s)-\mathscr{D}_f(\Pi^1)(t)+\Gamma^1_{t,s}\mathscr{D}_f(\Pi^1)(s)\\&=\sum_{\ell=0}^{M}\bigg\{\sum_{k=\ell}^{M}\big[R^2_k(t,s)-R^1_k(t,s)\big]\binom{k}{\ell}\big(\Pi^2_sI(\Xi)(t) \big)^{k-\ell}\bigg\}I(\Xi)^{\ell}\Xi\\&
+\sum_{\ell=0}^{M}\bigg\{\sum_{k=\ell}^{M}R^1_k(t,s)\binom{k}{\ell}\big[\big(\Pi^2_sI(\Xi)(t)\big)^{k-\ell}-\big(\Pi^1_sI(\Xi)(t) \big)^{k-\ell}\big]\bigg\}I(\Xi)^{\ell}\Xi.
\end{aligned}
\end{equation} 
Starting from the first term on the right-hand side we have 
\begin{equation}\label{ABdecomp}
\begin{aligned}
R^2_k(t,s)-R^1_k(t,s)=A^2_k(t,s)-A^1_k(t,s)+B^2_k(t,s)-B^1_k(t,s)
\end{aligned}
\end{equation} 
and
\begin{equation}
\begin{aligned}
A^2_k(t,s)-A^1_k(t,s) & = \frac{1}{k!(M-k)!}\bigg[\bigg( (K^H*\Pi^2\Xi)(t,s)\bigg)^{M+1-k}-\bigg( (K^H*\Pi^1\Xi)(t,s)\bigg)^{M+1-k}
\bigg]\\
&\times\bigg(\int_{0}^{1}\partial^{M+1}_1f\bigg((K^H*\Pi^1_t\Xi)(t)+\theta\bigg[ (K^H*\Pi^1\Xi)(t,s)\bigg],s\bigg)(1-\theta)^Md\theta\bigg)\\
&
+\frac{\bigg( (K^H*\Pi^1\Xi)(t,s)\bigg)^{M+1-k}}{k!(M-k)!}
\int_{0}^{1}\bigg\{\partial^{M+1}_1f\bigg((K^H*\Pi^2_t\Xi)(t)+\theta\bigg[ (K^H*\Pi^2\Xi)(t,s)\bigg],s\bigg)\\
&-\partial^{M+1}_1f\bigg((K^H*\Pi^1_t\Xi)(t)+\theta\bigg[ (K^H*\Pi^1\Xi)(t,s)\bigg],s\bigg)\bigg\}(1-\theta)^Md\theta.
\end{aligned}
\end{equation} 
From the growth assumptions on~$f$ we obtain
\begin{equation}
\begin{aligned}
&k!(M-k)!|A^2_k(t,s)-A^1_k(t,s)|\\&\leq C_{f,T,M}1\vee \bigg\|K^H*\Pi^1\Xi\bigg\|_{\Cc^{H-\kappa}}^{N}\bigg|\bigg( (K^H*\Pi^2\Xi)(t,s)\bigg)^{M+1-k}-\bigg( (K^H*\Pi^1\Xi)(t,s)\bigg)^{M+1-k}
\bigg|\\&
+C'_{f,T,M}\bigg| (K^H*\Pi^1\Xi)(t,s)\bigg|^{M+1-k}\bigg(1\vee\big\|K^H*\Pi^1\Xi\big\|_{C^{H-\kappa}[0,T]}^N\bigg)\\&
\times \bigg| (K^H*\Pi^2\Xi)(t,s)-(K^H*\Pi^1\Xi)(t,s)\bigg|\vee \bigg| (K^H*\Pi^2\Xi)(t,s)-(K^H*\Pi^1\Xi)(t,s)\bigg|^N.
\end{aligned}
\end{equation} 
Now, letting $x_i=(K^H*\Pi^i_s\Xi)(s)-(K^H*\Pi^i_t\Xi)(t)$, and re-expanding $x_2^{M+1-k}$ around $x_1$ using~\eqref{binom} we continue the last estimate as follows:
\begin{equation}\label{Aestimate}
\begin{aligned}
k!(M-k)!|A^2_k(t,s)&-A^1_k(t,s)|\leq C_{f,T,M}1\vee \bigg\|K^H*\Pi^1\Xi\bigg\|_{\Cc^{H-\kappa}}^{N}\\&\times\sum_{j=1}^{M+1-k}\binom{M+1-k}{j}\bigg\|K^H*\Pi^1\Xi\bigg\|_{\Cc^{H-\kappa}}^{M+1-k-j}\bigg\|  K^H*\Pi^2\Xi-K^H*\Pi^1\Xi \bigg\|_{\Cc^{H-\kappa}}^j\\&
+C'_{f,T,M}|t-s|^{(M+1-k)(H-\kappa)}\bigg\| K^H*\Pi^1\Xi\bigg\|_{\Cc^{H-\kappa}}^{M+1-k}\bigg(1\vee\big\|K^H*\Pi^1\Xi\big\|_{C^{H-\kappa}[0,T]}^N\bigg)\\&
\times \bigg\|K^H*\Pi^2\Xi-K^H*\Pi^1\Xi\bigg\|_{\Cc^{H-\kappa}}\vee \bigg\|K^H*\Pi^2\Xi-K^H*\Pi^1\Xi\bigg\|^{N}_{\Cc^{H-\kappa}}\\&
\leq C'_{f,T,M}|t-s|^{(M+1-k)(H-\kappa)}1\vee \bigg\|K^H*\Pi^1\Xi\bigg\|_{\Cc^{H-\kappa}}^{N+M+1-k}\\&\times \bigg\|K^H*\Pi^2\Xi-K^H*\Pi^1\Xi\bigg\|_{\Cc^{H-\kappa}}\vee \bigg\|K^H*\Pi^2\Xi-K^H*\Pi^1\Xi\bigg\|^{N+M+1-k}_{\Cc^{H-\kappa}}.
%
\end{aligned}
\end{equation}
 As for the last terms in~\eqref{ABdecomp} we write
\begin{equation}
\begin{aligned}
&B^2_k(t,s)-B^1_k(t,s)\\&=\sum_{m=0}^{M-k}\frac{1}{k!m!}\int_{t}^{s}\bigg(\partial^{k+m,1}_{1,2}f\big((K^H*\Pi^2_t\Xi)(t),x\big)-\partial^{k+m,1}_{1,2}f\big((K^H*\Pi^1_t\Xi)(t),x\big)\bigg)dx\big[ (K^H*\Pi^1\Xi)(t,s)   \big]^m\\&
+\sum_{m=0}^{M-k}\frac{1}{k!m!}\bigg(\int_{t}^{s}\big[\partial^{k+m,1}_{1,2}f\big((K^H*\Pi^2_t\Xi)(t),x\big)-\partial^{k+m,1}_{1,2}f\big((K^H*\Pi^1_t\Xi)(t),x\big)\big]dx\bigg)\\&\quad\quad\times\bigg(\big[ (K^H*\Pi^2\Xi)(t,s)    \big]^m-\big[ (K^H*\Pi^1\Xi)(t,s)\big]^m\bigg)\\&
+\sum_{m=0}^{M-k}\frac{1}{k!m!}\bigg(\int_{t}^{s}\partial^{k+m,1}_{1,2}f\big((K^H*\Pi^1_t\Xi)(t),x\big) dx\bigg)\\&\quad\quad\times\bigg(\big[ (K^H*\Pi^2\Xi)(t,s)    \big]^m-\big[ (K^H*\Pi^1\Xi)(t,s)\big]^m\bigg).
\end{aligned}
\end{equation}
These terms can be bounded by using the following facts: 1) From~\eqref{unicontcondition},~$f$ and its derivatives are bounded on their second argument over compact time intervals. In particular, all the terms above can be bounded, up to a constant, in time by $|t-s|$ 2)~$f$ and its derivatives have at most polynomial growth of degree~$N$ on their first argument. The latter, along with another polynomial re-expansion argument yield
\begin{equation}\label{Bestimate}
\begin{aligned}
&|B^2_k(t,s)-B^1_k(t,s)|\leq C_{f,M,T}|t-s| 1\vee \bigg\|K^H*\Pi^1\Xi\bigg\|_{\Cc^{H-\kappa}}^{N+M-k}\\&
\times \bigg\|K^H*\Pi^2\Xi-K^H*\Pi^1\Xi\bigg\|_{\Cc^{H-\kappa}}\vee \bigg\|K^H*\Pi^2\Xi-K^H*\Pi^1\Xi\bigg\|^{N+M-k}_{\Cc^{H-\kappa}}.
\end{aligned}
\end{equation}
In view of~\eqref{ABdecomp},~\eqref{Aestimate} and~\eqref{Bestimate}, it follows that 

\begin{equation}
\begin{aligned}
&|R^2_k(t,s)-R^1_k(t,s)|\leq |A^2_k(t,s)-A^1_k(t,s)|+|B^2_k(t,s)-B^1_k(t,s)|\\&
\leq C'_{f,T, M}\bigg(|t-s|^{(M+1-k)(H-\kappa)}+|t-s|\bigg)\bigg(1\vee \bigg\|K^H*\Pi^1\Xi\bigg\|_{\Cc^{H-\kappa}}^{N+M+1-k}\bigg)\\&\times \bigg\|K^H*\Pi^2\Xi-K^H*\Pi^1\Xi\bigg\|_{\Cc^{H-\kappa}}\vee \bigg\|K^H*\Pi^2\Xi-K^H*\Pi^1\Xi\bigg\|^{N+M+1-k}_{\Cc^{H-\kappa}}.
\end{aligned}
\end{equation}

Returning to the first term in~\eqref{eq:Remdecomp} we have, using that $\Pi^i_sI(\Xi)(t)=K^H*\Pi^i_t\Xi-K^H*\Pi^i_s\Xi$,
\begin{equation}
\begin{aligned}
\bigg|&\sum_{k=\ell}^{M}\big[R^2_k(t,s)-R^1_k(t,s)\big]\binom{k}{\ell}\big(\Pi^2_sI(\Xi)(t) \big)^{k-\ell}\bigg|\\&\leq \sum_{k=\ell}^{M}2^{k-\ell+1}\binom{k}{\ell}\big|R^2_k(t,s)-R^1_k(t,s)\big|\bigg( \big|\Pi^2_sI(\Xi)(t)-\Pi^1_sI(\Xi)(t) \big|^{k-\ell}+ \big|\Pi^1_sI(\Xi)(t)\big|^{k-\ell}\bigg)\\&
\leq 2^{M-\ell+1}C'_{f,T, M}\sum_{k=\ell}^{M}\binom{k}{\ell}\bigg(|t-s|^{(M+1-\ell)(H-\kappa)}+|t-s|^{1+(k-\ell)(H-\kappa)}\bigg)\bigg(1\vee \big\|K^H*\Pi^1\Xi\big\|_{\Cc^{H-\kappa}[0,T]}^{N+M+1-k}  \bigg)\\& \times  \bigg(\big\|K^H*\Pi^2\Xi-K^H*\Pi^1\Xi\big\|_{\Cc^{H-\kappa}[0,T]}\vee \big\|K^H*\Pi^2\Xi-K^H*\Pi^1\Xi\big\|^{N+M+1-k}_{\Cc^{H-\kappa}[0,T]}\bigg)\\&\times \bigg( \big\|K^H*\Pi^2\Xi-K^H*\Pi^1\Xi\big\|_{\Cc^{H-\kappa}[0,T]}^{k-\ell}+ \big\|K^H*\Pi^1\Xi\big\|_{\Cc^{H-\kappa}[0,T]}^{k-\ell}\bigg)\\&
\leq   C'_{f, T, M}|t-s|^{(M+1-\ell)(H-\kappa)} 
\bigg(1\vee \big\|K^H*\Pi^1\Xi\big\|_{\Cc^{H-\kappa}[0,T]}^{N+M+1}  \bigg)\\&\times 
\bigg(\big\|K^H*\Pi^2\Xi-K^H*\Pi^1\Xi\big\|_{\Cc^{H-\kappa}[0,T]}\vee \big\|K^H*\Pi^2\Xi-K^H*\Pi^1\Xi\big\|^{N+M+1-\ell}_{\Cc^{H-\kappa}[0,T]}\bigg).
\end{aligned}
\end{equation}
Finally, for the second term in~\eqref{eq:Remdecomp}, similar estimates along with Lemma~\ref{modelbndlem} furnish
\begin{equation*}
\begin{aligned}
\bigg|\sum_{k=\ell}^{M}  & R^1_k(t,s)\binom{k}{\ell}
\left[\left(\Pi^2_sI(\Xi)(t)\right)^{k-\ell} - \left(\Pi^1_sI(\Xi)(t) \right)^{k-\ell}\right]\bigg|\\
 & \leq C_{f,M, T}|t-s|^{(M+1-\ell)(H-\kappa)} \bigg(1\vee \big\|K^H*\Pi^1\Xi\big\|_{\Cc^{H-\kappa}[0,T]}^{N+M+1}  \bigg)\\
 & \times \bigg(\big\|K^H*\Pi^2\Xi-K^H*\Pi^1\Xi\big\|_{\Cc^{H-\kappa}[0,T]}\vee \big\|K^H*\Pi^2\Xi-K^H*\Pi^1\Xi\big\|^{N+M+1-\ell}_{\Cc^{H-\kappa}[0,T]}\bigg).
\end{aligned}
\end{equation*}
A combination of the estimates in the last two displays concludes the proof.


\bibliographystyle{siam}
\bibliography{TPC}

\end{document}